\documentclass{article}
\usepackage{algorithm2e}
\usepackage{url}

\usepackage{fullpage}
\pdfoutput=1
\usepackage{amsmath, amsfonts,mathrsfs,bm}
\usepackage{graphicx, xcolor}
\usepackage{multirow,enumerate}
\usepackage{tikz,pgfplots}
\usepackage{booktabs,colortbl}
\usepackage{arydshln}
\newcommand{\ra}[1]{\renewcommand{\arraystretch}{#1}}
\usepackage{mathrsfs,xspace}
\usepackage{algorithmic}

\usepackage{blindtext}
\usepackage{rotating}
\usepackage[]{cite}
\usepackage{float}
\usepackage{hyperref}
\usepackage{cleveref}
\Crefname{ALC@unique}{Line}{Lines}

\usetikzlibrary{patterns,arrows,shapes.misc}
\tikzset{cross/.style={cross out, draw=black, minimum size=2*(#1-\pgflinewidth), inner sep=0pt, outer sep=0pt},
cross/.default={5pt}}

\newtheorem{theorem}{Theorem}[section]

\newtheorem{corollary}[theorem]{Corollary}
\newtheorem{definition}{Definition}

\newenvironment{proof}[1][Proof]{\begin{trivlist}
\item[\hskip \labelsep {\bfseries #1}]}{\end{trivlist}}

\newcommand{\R}{\mathbb{R}}
\newcommand{\C}{\mathbb{C}}
\newcommand{\comp}[1]{#1^\mathsf{c}}

\newcommand{\ie}{\textit{i.e.}}
\newcommand{\eg}{\textit{e.g.}}

\newcommand{\response}[1]{#1}

\newcommand{\SF}[2]{\m{Z}\left(#1;#2\right)}
\newcommand{\SFW}[2]{\widetilde{\m{Z}}\left(#1;#2\right)}

\definecolor{box}{RGB}{126, 107, 53}
\definecolor{near}{RGB}{0, 171, 196}
\definecolor{far}{RGB}{228,73,73}

\newcommand{\B}{\mathcal{B}}
\newcommand{\F}{\mathcal{F}}
\newcommand{\N}{\mathcal{N}}
\renewcommand{\L}{\mathscr{L}}

\newcommand{\rd}{\mathcal{R}}
\newcommand{\sk}{\mathcal{S}}
\newcommand{\sktop}{\B_{t}}

\newcommand{\prox}{\mathcal{P}}
\newcommand{\notprox}{\mathcal{O}}

\newcommand{\rdnocolor}{\mathcal{R}}
\newcommand{\sknocolor}{\mathcal{S}}

\newcommand{\idx}[3]{{\mathsf #1}_{#2#3}}
\newcommand{\m}[1]{{\mathsf #1}}

\newcommand{\I}{\mathcal{I}}
\newcommand{\J}{\mathcal{J}}
\newcommand{\K}{\mathcal{K}}

\newcommand{\lmat}{[}
\newcommand{\rmat}{]}

\renewcommand{\O}{O}

\newcommand{\scinote}[3]{#1\text{e$#2$#3}}

\usepackage{mathtools,mathabx}

\definecolor{lightergray}{RGB}{240,240,230}
\colorlet{lightgray}{black!20}
\colorlet{darkgray}{black!50}
\colorlet{darkergray}{black!80}

\newcommand{\mypic}[1]{\scalebox{0.55}{
  \begin{tikzpicture}
  \def \slf {#1}
  \pgfmathsetmacro{\nbrone}{\slf-1}
  \pgfmathsetmacro{\nbrtwo}{\slf+1}
  \pgfmathsetmacro{\farone}{\slf-2}
  \pgfmathsetmacro{\fartwo}{\slf+2}

    \clip[rounded corners](-1.5,-2.5) rectangle (8.5,9.5);
    \draw[draw=black,fill=lightergray,line width=0.5mm, rounded corners](-1.5,-2.5) rectangle (8.5,9.5);

    \filldraw[draw=black,thick,fill=white] (0,0) rectangle (8,8);
    \foreach \x in {0,...,7}
    {
    	\draw[thick] (\x,0) -- (\x,8);
    	\draw[thick] (0,\x) -- (8,\x);
    }

    \foreach \x in {0,...,7}
    {
      \foreach \y in {0,...,7}
      {
      	\ifthenelse{\x<\slf \AND \y<\slf}{
      		\filldraw[fill=lightgray,draw=black,thick] (\x+0.5,8-\y-1) rectangle (\x+1,8-\y-0.5);
      		\pgfmathifthenelse{\x>=\y-2 && \x<=\y+2}{"\noexpand\draw[line width=1.5] (\x+0.75,8-\y-0.75) node[cross] {};"}{}\pgfmathresult;
      	}{};

        \ifthenelse{\x<\slf \AND \y=\x}{
        	\filldraw[fill=lightgray,draw=black,thick] (\x,8-\y-0.5) rectangle (\x+0.5,8-\y);
        	\draw[line width=1.5] (\x+0.25,8-\y-0.25) node[cross] {};
        }{};
  		\ifthenelse{\x<\slf \AND \NOT \y<\slf}{
  			\filldraw[fill=lightgray,draw=black,thick] (\x+0.5,8-\y-1) rectangle (\x+1,8-\y);
  			\pgfmathifthenelse{\x>=\y-1 && \x<=\y+1}{"\noexpand\draw[line width=1.5] (\x+0.75,8-\y-0.5) node[cross] {};"}{}\pgfmathresult;
  		}{};
  		\ifthenelse{\NOT \x<\slf \AND \y<\slf}{
  			\filldraw[fill=lightgray,draw=black,thick] (\x,8-\y-1) rectangle (\x+1,8-\y-0.5);
  			 \pgfmathifthenelse{\x>=\y-1 && \x<=\y+1}{"\noexpand\draw[line width=1.5] (\x+0.5,8-\y-0.75) node[cross] {};"}{}\pgfmathresult
  		}{};
  		 \pgfmathifthenelse{\x==\slf-2 && \y==\slf}{"\noexpand\draw[line width=1.5] (\x+0.75,8-\y-0.5) node[cross] {};"}{}\pgfmathresult;
  		 \pgfmathifthenelse{\x==\slf && \y==\slf-2}{"\noexpand\draw[line width=1.5] (\x+0.5,8-\y-0.75) node[cross] {};"}{}\pgfmathresult;
  		\ifthenelse{\NOT \x<\slf \AND \NOT \y<\slf}{\filldraw[fill=lightgray,draw=black,thick] (\x,7-\y) rectangle (\x+1,7-\y+1)}{};
  		\pgfmathifthenelse{\x==\slf && \y==\slf && \slf > 0}{"\noexpand\draw[line width=1.5] (\x+0.5,8-\y-0.5) node[cross] {};"}{}\pgfmathresult;
  	  }
  	  }

    \foreach \x in {0,...,7}{
    	\ifthenelse{\x=\slf}{\fill[opacity = 0.54, box,rounded corners=1ex] (\x,8) -- (\x+1, 8) -- (\x+1, 9) -- (\x,9) -- cycle}{};
 		\ifthenelse{\x=\nbrone \OR \x=\nbrtwo}{\fill[opacity = 0.54, near,rounded corners=1ex] (\x,8) -- (\x+1, 8) -- (\x+1, 9) -- (\x,9) -- cycle}{};
 		\pgfmathifthenelse{\x<=\farone || \x>=\fartwo}{"\noexpand\fill[opacity = 0.54, far,rounded corners=1ex] (\x,8) -- (\x+1, 8) -- (\x+1, 9) -- (\x,9) -- cycle;"}{}\pgfmathresult;
	}
	\foreach \x in {0,...,7}{
    	\ifthenelse{\x=\slf}{\fill[opacity = 0.54, box,rounded corners=1ex] (-1,7-\x) -- (-1,7-\x+1) -- (0,7-\x+1) -- (0,7-\x) -- cycle}{};
 		\ifthenelse{\x=\nbrone \OR \x=\nbrtwo}{\fill[opacity = 0.54, near,rounded corners=1ex] (-1,7-\x) -- (-1,7-\x+1) -- (0,7-\x+1) -- (0,7-\x) -- cycle}{};
 		 \pgfmathifthenelse{\x<=\farone || \x>=\fartwo}{"\noexpand\fill[opacity = 0.54, far,rounded corners=1ex] (-1,7-\x) -- (-1,7-\x+1) -- (0,7-\x+1) -- (0,7-\x) -- cycle;"}{}\pgfmathresult;
	}
    \foreach \x in {1,...,8}
      	\draw (\x-0.5,8.5) node{\Large$\B_\x$};

    \foreach \y in {1,...,8}
      	\draw (-0.5,8-\y+0.5) node{\Large$\B_\y$};

  \draw[draw=black,thick] (0,-1.5) -- (8,-1.5);
  \draw (-0.5,-1.5) node{\Large$\Omega$};

  \foreach\x in {0,...,8}
  	\draw[draw=black,thick] (\x,-1.5+0.25) -- (\x,-1.5-0.25);

  \foreach\x in {0,...,7}
  {
 		\ifthenelse{\x=\slf}{\fill[opacity = 0.54, box,rounded corners=1ex] (\x,-1.5-0.25) -- (\x+1, -1.5-0.25) -- (\x+1, -1.5+0.25) -- (\x,-1.5+0.25) -- cycle}{};
 		\ifthenelse{\x=\nbrone \OR \x=\nbrtwo}{\fill[opacity = 0.54, near,rounded corners=1ex] (\x,-1.5-0.25) -- (\x+1, -1.5-0.25) -- (\x+1, -1.5+0.25) -- (\x,-1.5+0.25) -- cycle}{};
 		 \pgfmathifthenelse{\x<=\farone || \x>=\fartwo}{"\noexpand\fill[opacity = 0.54, far,rounded corners=1ex] (\x,-1.5-0.25) -- (\x+1, -1.5-0.25) -- (\x+1, -1.5+0.25) -- (\x,-1.5+0.25) -- cycle;"}{}\pgfmathresult;

 	\ifthenelse{\NOT \x<\slf}{
	  	\foreach\p in {1,...,4}
	 	{
	 		\ifthenelse{\x=\slf \OR \x=\nbrone \OR \x=\nbrtwo \OR \x=\farone \OR \x=\fartwo}{}{\filldraw[fill=far,draw=darkergray,thick] (\x+\p/4-1/8,-1.5) circle (0.1);};

	 		\ifthenelse{\x=\slf}{\filldraw[fill=box,draw=darkergray,thick] (\x+\p/4-1/8,-1.5) circle (0.1)}{};
	 		\ifthenelse{\x=\nbrone \OR \x=\nbrtwo}{\filldraw[fill=near,draw=darkergray,thick] (\x+\p/4-1/8,-1.5) circle (0.1)}{};
	 		\ifthenelse{\x=\farone \OR \x=\fartwo}{\filldraw[fill=far,draw=darkergray,thick] (\x+\p/4-1/8,-1.5) circle (0.1)}{};
	  	}
  	}{
  		\ifthenelse{\x=0}{
  			\ifthenelse{\x=\slf \OR \x=\nbrone \OR \x=\nbrtwo \OR \x=\farone \OR \x=\fartwo}{}{\filldraw[fill=far,draw=darkergray,thick] (\x+4/4-1/8,-1.5) circle (0.1);};
	 		\ifthenelse{\x=\slf}{\filldraw[fill=box,draw=darkergray,thick] (\x+4/4-1/8,-1.5) circle (0.1)}{};
	 		\ifthenelse{\x=\nbrone \OR \x=\nbrtwo}{\filldraw[fill=near,draw=darkergray,thick] (\x+4/4-1/8,-1.5) circle (0.1)}{};
	 		\ifthenelse{\x=\farone \OR \x=\fartwo}{\filldraw[fill=far,draw=darkergray,thick] (\x+4/4-1/8,-1.5) circle (0.1)}{};
  		}{
  		\ifthenelse{\x=7}{
  			\ifthenelse{\x=\slf \OR \x=\nbrone \OR \x=\nbrtwo \OR \x=\farone \OR \x=\fartwo}{}{\filldraw[fill=far,draw=darkergray,thick] (\x+1/4-1/8,-1.5) circle (0.1);};

	 		\ifthenelse{\x=\slf}{\filldraw[fill=box,draw=darkergray,thick] (\x+1/4-1/8,-1.5) circle (0.1)}{};
	 		\ifthenelse{\x=\nbrone \OR \x=\nbrtwo}{\filldraw[fill=near,draw=darkergray,thick] (\x+1/4-1/8,-1.5) circle (0.1)}{};
	 		\ifthenelse{\x=\farone \OR \x=\fartwo}{\filldraw[fill=far,draw=darkergray,thick] (\x+1/4-1/8,-1.5) circle (0.1)}{};
  		}{
  		\foreach\p in {1,4}
	 	{
	 		\ifthenelse{\x=\slf \OR \x=\nbrone \OR \x=\nbrtwo \OR \x=\farone \OR \x=\fartwo}{}{\filldraw[fill=far,draw=darkergray,thick] (\x+\p/4-1/8,-1.5) circle (0.1);};

	 		\ifthenelse{\x=\slf}{\filldraw[fill=box,draw=darkergray,thick] (\x+\p/4-1/8,-1.5) circle (0.1)}{};
	 		\ifthenelse{\x=\nbrone \OR \x=\nbrtwo}{\filldraw[fill=near,draw=darkergray,thick] (\x+\p/4-1/8,-1.5) circle (0.1)}{};
	 		\ifthenelse{\x=\farone \OR \x=\fartwo}{\filldraw[fill=far,draw=darkergray,thick] (\x+\p/4-1/8,-1.5) circle (0.1)}{};
	  	}
  		};};
  	}

  	}

  	\draw[draw=black,line width=0.25mm](-1.5,-0.5) -- (8.5,-0.5);

  	\draw[use as bounding box,draw=black,line width=0.5mm, rounded corners](-1.5,-2.5) rectangle (8.5,9.5);
  \end{tikzpicture}}

}

\newcommand{\mypicnext}{\scalebox{0.55}{
  \begin{tikzpicture}
  \def \slf {0}
  \pgfmathsetmacro{\nbrone}{\slf-1}
  \pgfmathsetmacro{\nbrtwo}{\slf+1}
  \pgfmathsetmacro{\farone}{\slf-2}
  \pgfmathsetmacro{\fartwo}{\slf+2}
    \pgfmathsetmacro{\farthree}{\slf-3}
  \pgfmathsetmacro{\farfour}{\slf+3}

    \clip[rounded corners](-1.5,-2.5) rectangle (8.5,9.5);
    \draw[draw=black,fill=lightergray,line width=0.5mm, rounded corners](-1.5,-2.5) rectangle (8.5,9.5);

    \filldraw[draw=black,thick,fill=white] (0,0) rectangle (8,8);
    \foreach \x in {0,...,7}
    {
    	\draw[thick] (\x,0) -- (\x,8);
    	\draw[thick] (0,\x) -- (8,\x);
    }

    \foreach \x in {4,...,7}
    {
    	\filldraw[fill=lightgray,draw=black,thick] (\x+0.5,8-\x-1) rectangle (\x+1,8-\x-0.5);
    	\draw[line width=1.5] (\x+0.75,8-\x-0.75) node[cross] {};
    	\filldraw[fill=lightgray,draw=black,thick] (\x,8-\x-0.5) rectangle (\x+0.5,8-\x);
    	\draw[line width=1.5] (\x+0.25,8-\x-0.25) node[cross] {};
    }

    \foreach \x in {0,...,3}
    {
      \foreach \y in {0,...,3}
      {
      	\filldraw[fill=lightgray,draw=black,thick] (\x,7-\y) rectangle (\x+1,7-\y+1);
      	\pgfmathifthenelse{\x>=\y-1 && \x<=\y+1}{"\noexpand\draw[line width=1.5] (\x+0.5,8-\y-0.5) node[cross] {};"}{}\pgfmathresult
      }
    }

    \foreach \x in {0,...,7}{
    	\ifthenelse{\x=\slf}{\fill[opacity = 0.54, box,rounded corners=1ex] (\x,8) -- (\x+1, 8) -- (\x+1, 9) -- (\x,9) -- cycle}{};
 		\ifthenelse{\x=\nbrone \OR \x=\nbrtwo}{\fill[opacity = 0.54, near,rounded corners=1ex] (\x,8) -- (\x+1, 8) -- (\x+1, 9) -- (\x,9) -- cycle}{};
 		\ifthenelse{\x=\farone \OR \x=\fartwo \OR \x=\farthree \OR \x=\farfour}{\fill[opacity = 0.54, far,rounded corners=1ex] (\x,8) -- (\x+1, 8) -- (\x+1, 9) -- (\x,9) -- cycle}{};
	}
	\foreach \x in {0,...,7}{
    	\ifthenelse{\x=\slf}{\fill[opacity = 0.54, box,rounded corners=1ex] (-1,7-\x) -- (-1,7-\x+1) -- (0,7-\x+1) -- (0,7-\x) -- cycle}{};
 		\ifthenelse{\x=\nbrone \OR \x=\nbrtwo}{\fill[opacity = 0.54, near,rounded corners=1ex] (-1,7-\x) -- (-1,7-\x+1) -- (0,7-\x+1) -- (0,7-\x) -- cycle}{};
 		\ifthenelse{\x=\farone \OR \x=\fartwo\OR \x=\farthree \OR \x=\farfour}{\fill[opacity = 0.54, far,rounded corners=1ex] (-1,7-\x) -- (-1,7-\x+1) -- (0,7-\x+1) -- (0,7-\x) -- cycle}{};
	}
    \foreach \x/\xeval in {1/9,2/10,3/11,4/12}
      	\draw (\x-0.5,8.5) node{\Large$\B_{\xeval}$};
    \draw[black,very thick] (4.5,8.5) -- (7.5,8.5);
    \draw(6,8.5) node[above] {\large(inactive DOFs)};

    \foreach \y/\yeval in {1/9,2/10,3/11,4/12}
      	\draw (-0.5,8-\y+0.5) node{\Large$\B_{\yeval}$};
     \draw[black,very thick] (-0.5,8-4.5) -- (-0.5,8-7.5);

  \draw[draw=black,thick] (0,-1.5) -- (8,-1.5);
  \draw (-0.5,-1.5) node{\Large$\Omega$};

  \foreach\x in {0,...,4}
  	\draw[draw=black,thick] (2*\x,-1.5+0.25) -- (2*\x,-1.5-0.25);

  \foreach\x in {0,...,4}
  {
 		\ifthenelse{\x=\slf}{\fill[opacity = 0.54, box,rounded corners=1ex] (2*\x,-1.5-0.25) -- (2*\x+2, -1.5-0.25) -- (2*\x+2, -1.5+0.25) -- (2*\x,-1.5+0.25) -- cycle}{};
 		\ifthenelse{\x=\nbrone \OR \x=\nbrtwo}{\fill[opacity = 0.54, near,rounded corners=1ex] (2*\x,-1.5-0.25) -- (2*\x+2, -1.5-0.25) -- (2*\x+2, -1.5+0.25) -- (2*\x,-1.5+0.25) -- cycle}{};
 		\ifthenelse{\x=\farone \OR \x=\fartwo \OR \x=\farthree \OR \x=\farfour}{\fill[opacity = 0.54, far,rounded corners=1ex] (2*\x,-1.5-0.25) -- (2*\x+2, -1.5-0.25) -- (2*\x+2, -1.5+0.25) -- (2*\x,-1.5+0.25) -- cycle}{};

}
\foreach\p in {4,5,8,9,12,13,16,17,20,21,24,25,28,29}{
	\ifthenelse{\p < 9}{
		\filldraw[fill=box,draw=darkergray,thick] (\p/4-1/8,-1.5) circle (0.1);
	}{};
	\ifthenelse{\p>8 \AND \p < 17}{
		\filldraw[fill=near,draw=darkergray,thick] (\p/4-1/8,-1.5) circle (0.1)
	}{};
	\ifthenelse{\p>16}{ %\AND \p < 25}{
		\filldraw[fill=far,draw=darkergray,thick] (\p/4-1/8,-1.5) circle (0.1)
	}{};
}

  	\draw[draw=black,line width=0.25mm](-1.5,-0.5) -- (8.5,-0.5);

  	\draw[use as bounding box,draw=black,line width=0.5mm, rounded corners](-1.5,-2.5) rectangle (8.5,9.5);
  \end{tikzpicture}}

}

\newcommand{\mypicnextnext}{\scalebox{0.55}{
  \begin{tikzpicture}
  \def \slf {3}
  \pgfmathsetmacro{\nbrone}{\slf-1}
  \pgfmathsetmacro{\nbrtwo}{\slf+1}
  \pgfmathsetmacro{\farone}{\slf-2}
  \pgfmathsetmacro{\fartwo}{\slf+2}
    \pgfmathsetmacro{\farthree}{\slf-3}
  \pgfmathsetmacro{\farfour}{\slf+3}

    \clip[rounded corners](-1.5,-2.5) rectangle (8.5,9.5);
    \draw[draw=black,fill=lightergray,line width=0.5mm, rounded corners](-1.5,-2.5) rectangle (8.5,9.5);

    \filldraw[draw=black,thick,fill=white] (0,0) rectangle (8,8);
    \foreach \x in {0,...,7}
    {
    	\draw[thick] (\x,0) -- (\x,8);
    	\draw[thick] (0,\x) -- (8,\x);
    }

    \foreach \x in {0,1,2,4,5,6,7}
    {
    	\filldraw[fill=lightgray,draw=black,thick] (\x+0.5,8-\x-1) rectangle (\x+1,8-\x-0.5);
    	\draw[line width=1.5] (\x+0.75,8-\x-0.75) node[cross] {};
    	\filldraw[fill=lightgray,draw=black,thick] (\x,8-\x-0.5) rectangle (\x+0.5,8-\x);
    	\draw[line width=1.5] (\x+0.25,8-\x-0.25) node[cross] {};
    }

    \foreach \x in {0,...,2}
    {
      \foreach \y in {0,...,2}
      {
      	\filldraw[fill=lightgray,draw=black,thick] (\x+0.5,7-\y) rectangle (\x+1,7-\y+0.5);
      	\pgfmathifthenelse{\x>=\y-2 && \x<=\y+2}{"\noexpand\draw[line width=1.5] (\x+0.75,8-\y-0.75) node[cross] {};"}{}\pgfmathresult;
      }
    }

    \filldraw[fill=lightgray,draw=black,thick] (3,7-3) rectangle (4,7-3+1);
    \draw[line width=1.5] (3+0.5,8-3-0.5) node[cross] {};

    \foreach \x in {0,...,2}
    {
    	\filldraw[fill=lightgray,draw=black,thick] (\x+0.5,7-3) rectangle (\x+1,7-3+1);
    	\filldraw[fill=lightgray,draw=black,thick] (3,7-\x) rectangle (3+1,7-\x+0.5);
    }

    \draw[line width=1.5] (3+0.5,8-2-0.75) node[cross] {};
    \draw[line width=1.5] (2+0.75,8-3-0.5) node[cross] {};

    \draw[line width=1.5] (1+0.75,8-3-0.5) node[cross] {};
    \draw[line width=1.5] (3+0.5,8-1-0.75) node[cross] {};

    \foreach \x in {0,...,3}{
    	\ifthenelse{\x=\slf}{\fill[opacity = 0.54, box,rounded corners=1ex] (\x,8) -- (\x+1, 8) -- (\x+1, 9) -- (\x,9) -- cycle}{};
 		\ifthenelse{\x=\nbrone \OR \x=\nbrtwo}{\fill[opacity = 0.54, near,rounded corners=1ex] (\x,8) -- (\x+1, 8) -- (\x+1, 9) -- (\x,9) -- cycle}{};
 		\ifthenelse{\x=\farone \OR \x=\fartwo\OR \x=\farthree\OR \x=\farfour}{\fill[opacity = 0.54, far,rounded corners=1ex] (\x,8) -- (\x+1, 8) -- (\x+1, 9) -- (\x,9) -- cycle}{};
	}
	\foreach \x in {0,...,3}{
    	\ifthenelse{\x=\slf}{\fill[opacity = 0.54, box,rounded corners=1ex] (-1,7-\x) -- (-1,7-\x+1) -- (0,7-\x+1) -- (0,7-\x) -- cycle}{};
 		\ifthenelse{\x=\nbrone \OR \x=\nbrtwo}{\fill[opacity = 0.54, near,rounded corners=1ex] (-1,7-\x) -- (-1,7-\x+1) -- (0,7-\x+1) -- (0,7-\x) -- cycle}{};
 		\ifthenelse{\x=\farone \OR \x=\fartwo\OR \x=\farthree\OR \x=\farfour}{\fill[opacity = 0.54, far,rounded corners=1ex] (-1,7-\x) -- (-1,7-\x+1) -- (0,7-\x+1) -- (0,7-\x) -- cycle}{};
	}
    \foreach \x/\xeval in {1/9,2/10,3/11,4/12}
      	\draw (\x-0.5,8.5) node{\Large$\B_{\xeval}$};
    \draw[black,very thick] (4.5,8.5) -- (7.5,8.5);
    \draw(6,8.5) node[above] {\large(inactive DOFs)};

    \foreach \y/\yeval in {1/9,2/10,3/11,4/12}
      	\draw (-0.5,8-\y+0.5) node{\Large$\B_{\yeval}$};
     \draw[black,very thick] (-0.5,8-4.5) -- (-0.5,8-7.5);

  \draw[draw=black,thick] (0,-1.5) -- (8,-1.5);
  \draw (-0.5,-1.5) node{\Large$\Omega$};

  \foreach\x in {0,...,4}
  	\draw[draw=black,thick] (2*\x,-1.5+0.25) -- (2*\x,-1.5-0.25);

  \foreach\x in {0,...,3}
  {
 		\ifthenelse{\x=\slf}{\fill[opacity = 0.54, box,rounded corners=1ex] (2*\x,-1.5-0.25) -- (2*\x+2, -1.5-0.25) -- (2*\x+2, -1.5+0.25) -- (2*\x,-1.5+0.25) -- cycle}{};
 		\ifthenelse{\x=\nbrone \OR \x=\nbrtwo}{\fill[opacity = 0.54, near,rounded corners=1ex] (2*\x,-1.5-0.25) -- (2*\x+2, -1.5-0.25) -- (2*\x+2, -1.5+0.25) -- (2*\x,-1.5+0.25) -- cycle}{};
 		\ifthenelse{\x=\farone \OR \x=\fartwo\OR \x=\farthree\OR \x=\farfour}{\fill[opacity = 0.54, far,rounded corners=1ex] (2*\x,-1.5-0.25) -- (2*\x+2, -1.5-0.25) -- (2*\x+2, -1.5+0.25) -- (2*\x,-1.5+0.25) -- cycle}{};

}
\foreach\p in {8,9,16,17,24,25,28,29}{
	\ifthenelse{\p < 9}{
		\filldraw[fill=far,draw=darkergray,thick] (\p/4-1/8,-1.5) circle (0.1);
	}{};
	\ifthenelse{\p>8 \AND \p < 17}{
		\filldraw[fill=far,draw=darkergray,thick] (\p/4-1/8,-1.5) circle (0.1)
	}{};
	\ifthenelse{\p>16 \AND \p < 25}{
		\filldraw[fill=near,draw=darkergray,thick] (\p/4-1/8,-1.5) circle (0.1)
	}{};
    \ifthenelse{\p>24}{
		\filldraw[fill=box,draw=darkergray,thick] (\p/4-1/8,-1.5) circle (0.1)
	}{};
}

  	\draw[draw=black,line width=0.25mm](-1.5,-0.5) -- (8.5,-0.5);

  	\draw[use as bounding box,draw=black,line width=0.5mm, rounded corners](-1.5,-2.5) rectangle (8.5,9.5);
  \end{tikzpicture}}

}

\begin{document}

\title{A recursive skeletonization factorization based on \\strong admissibility}

\author{
	Victor Minden\thanks{Institute for Computational and Mathematical Engineering, Stanford University, Stanford, CA 94305 (\url{vminden@stanford.edu}). {\bf Funding:} Stanford Graduate Fellowship in Science \& Engineering and U.S. Department of Energy Computational Science
Graduate Fellowship (grant number DE-FG02-97ER25308).}\and
  Kenneth L. Ho%
  \thanks{Department of Mathematics, Stanford University, Stanford, CA 94305.  Current address: TSMC Technology Inc., 2851 Junction Ave., San Jose, CA 95134. (\url{klho@alumni.caltech.edu}). {\bf Funding:} National Science Foundation Mathematical Sciences Postdoctoral Research Fellowship (grant number DMS-1203554).}
\and
	Anil Damle%
	\thanks{Department of Mathematics, University of California, Berkeley, Berkeley, CA 94720 (\url{damle@berkeley.edu}). {\bf Funding:} National Science Foundation Mathematical Sciences Postdoctoral Research Fellowship (grant number DMS-1606277).}%
	\and
	Lexing Ying\thanks{Department of Mathematics and Institute for Computational and Mathematical Engineering, Stanford University, Stanford, CA 94305 (\url{lexing@math.stanford.edu}). {\bf Funding:} National Science Foundation (grant number DMS-1521830) and U.S. Department of Energy Advanced Scientific Computing Research program (grant number DE-FC02-13ER26134 and DE-SC0009409).}
}

\maketitle

\begin{abstract}
We introduce the strong recursive skeletonization factorization (RS-S), a new approximate matrix factorization based on recursive skeletonization for solving discretizations of linear integral equations associated with elliptic partial differential equations in two and three dimensions (\response{and other} matrices with similar hierarchical rank structure).  Unlike previous skeletonization-based factorizations, RS-S uses a simple modification of skeletonization, strong skeletonization, which compresses only far-field interactions.  This leads to an approximate factorization in the form of a product of many block unit-triangular matrices that may be used as a preconditioner or moderate-accuracy direct solver, with dramatically reduced rank growth.  We further combine the strong skeletonization procedure with alternating near-field compression to obtain the hybrid recursive skeletonization factorization (RS-WS), a modification of RS-S that exhibits reduced storage cost in many settings.  Under suitable rank assumptions both RS-S and RS-WS exhibit linear computational complexity, which we demonstrate with a number of numerical examples.
\end{abstract}

\ra{1.4}

\section{Introduction}
\label{sec:intro}

Given a kernel function $\response{K(z)}$, we consider the integral equation
 \begin{align}\label{eq:bie}
  a(x) u(x) + b(x) \int_{\Omega} K(x-y) c(y) u(y) \, dy = f(x), \quad x \in \Omega \subset \R^{d}
  \end{align}
in dimension $d=2$ or $3$.  Here, $a(x),$ $b(x)$, and $c(y)$ are given functions that typically represent material parameters, $f(x)$ is some known right-hand side, and $u(x)$ is the unknown function to be determined.

We focus in this paper on the case where $\response{K(z)}$ is associated with some underlying elliptic partial differential equation ($\ie$, it is the Green's function or its derivative).  For optimal complexity of our methods the kernel $\response{K(z)}$ should not exhibit significant oscillation away from the origin, though this is not strictly necessary to apply the basic machinery.  In this setting, \eqref{eq:bie} remains rather general and includes problems such as the Laplace equation, the \response{Lippmann}-Schwinger equation, and the Helmholtz equation in the low- to moderate-frequency regime.  Further, while we concentrate on the case where $u(x)$ is scalar-valued, extension to the vector-valued case (\eg, the Stokes or elasticity equations) is straightforward.

Discretization of \eqref{eq:bie} using typical approaches such as collocation, the Nystr\"om method, or the Galerkin method leads to a linear system with $N$ degrees of freedom (DOFs)
\begin{align}\label{eq:linear}
\m{K}u &= f,
\end{align}
where the entries of the matrix $\m{K}\in\C^{N\times N}$ are dictated by the kernel $\response{K(z)}$ and the discretization scheme.  For example, in the case where our domain is the unit square $\Omega=[0,1]^2$ a simple Nystr\"om approximation to the integral using a regular grid with $\sqrt{N}$ points in each direction yields the discrete system
\begin{align}\label{eq:discbie}
\left[a(x_i) + w_i\right]u_i + \frac{b(x_i)}{N}\sum_{i\ne j}K(x_i - x_j)c(x_j)u_j = f(x_i),\quad i=1,\dots,N,
\end{align}
where the discrete solution $\{u_i\} \approx \{u(x_i)\}$ approximates the continuous solution on the grid and each term $w_i u_i$ corresponds to some discretization of diagonal entries of $\m{K}$.  \response{Because $K(z)$ is frequently singular at the origin, this discretization may be more involved than that of the off-diagonal entries.}  While more complicated and higher-order discretization schemes exist, \eqref{eq:discbie} illustrates the key feature that off-diagonal entries of $\m{K}$ are given essentially by \emph{kernel interactions between distinct points in space}.  In this paper we develop a method exploiting this fact and its consequences to efficiently solve \eqref{eq:linear}.

\subsection{Background and previous work}
Because $\m{K}$ in \eqref{eq:linear} is dense and generally large in practice, traditional direct factorizations of $\m{K}$ such as the LU factorization are typically too expensive due to the associated $\O(N^3)$ time complexity and $\O(N^2)$ storage cost.

Given the availability of fast schemes for applying $\m{K}$ such as fast multipole methods (FMMs) \cite{fong,fastmultipole,fmm3d,kifmm}, iterative methods such as the conjugate gradient method (CG) \cite{cg} form a tempting alternative to direct methods.  For first-kind integral equations or problems where $a(x)$, $b(x),$ or $c(x)$ exhibit high contrast, however, convergence is typically slow leading to a lack of robustness.  In other words, while each iteration is relatively fast, the number of iterations necessary to attain reasonable accuracies can be unreasonably large.

The above considerations have led to the development of a plethora of alternative methods for solving \eqref{eq:linear} approximately by exploiting properties of the kernel $\response{K(z)}$ and the underlying physical structure of the problem.  In particular, such methods take advantage of the fact that $\m{K}$ exhibits \emph{hierarchical block low-rank structure}.

A large body of work pioneered by Hackbusch and collaborators on the algebra of $\mathcal{H}$-matrices (and $\mathcal{H}^2$-matrices) provides an important and principled theoretical framework for obtaining linear or quasilinear complexity when working with matrices exhibiting such structure \cite{Hackbusch,HackbuschB,HackbuschK}.  Inside the asymptotic scaling of this approach, however, lurk large constant factors that hamper practical performance, particularly in the 3D case.

\begin{figure}
\centering
  \begin{tikzpicture}
 %\clip[rounded corners](-1.5,-0.5) rectangle (8.5,3.5);
\filldraw[very thick,draw=black,fill=box,fill opacity=0.54] (0,0) rectangle (2,2);
\filldraw[very thick,draw=black,fill=box,fill opacity=0.54] (5,0) rectangle(7,2);
\draw[<->,thick] (2.2,1) --  (4.8,1) node[above,midway] {$D'\ge D$};

\draw[|-|,thick] (0-0.3,0) --  (0-0.3,2) node[left,midway] {$D$};
\draw[|-|,thick] (0,2+0.3) --  (2,2+0.3) node[above,midway] {$D$};
\draw[|-|,thick] (5,2+0.3) --  (7,2+0.3) node[above,midway] {$D$};
\draw[|-|,thick] (7+0.3,0) --  (7+0.3,2) node[right,midway] {$D$};

\foreach \x in {0,...,7}
  \foreach \y in {0,...,7}
    \filldraw[fill=box,draw=darkergray,thick] (\x/4 + 0.125,\y/4 +0.125) circle (0.05);

\foreach \x in {0,...,7}
  \foreach \y in {0,...,7}
    \filldraw[fill=box,draw=darkergray,thick] (\x/4 + 0.125+5,\y/4 +0.125) circle (0.05);

    %\draw[use as bounding box,draw=black,line width=1mm, rounded corners](-1.5,-0.5) rectangle (8.5,3.5);
  \end{tikzpicture}
  \caption{Given two boxes in $\R^2$ each with sidelength $D$ and with corresponding DOF sets $\B_1$ and $\B_2$, in the strong admissibility setting the associated off-diagonal blocks $\idx{K}{\B_1}{\B_2}$ and $\idx{K}{\B_2}{\B_1}$ are assumed to be numerically low rank as long as the boxes are separated by a distance of at least $D$.  In contrast, in the weak admissibility setting the boxes need only be non-overlapping.\label{fig:admiss}}
\end{figure}
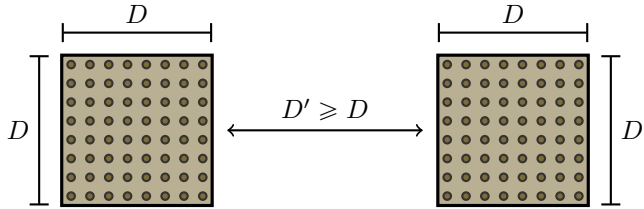

The $\mathcal{H}$-matrix literature classifies matrices with hierarchical block low-rank structure into two categories based on which off-diagonal blocks of the matrix are compressed.  Given a quadtree or octree data structure partitioning the domain into small boxes, let $\B_1$ and $\B_2$ be sets of DOFs corresponding to distinct boxes at the same level of the tree each with sidelength $D$.  For \emph{strongly-admissible} hierarchical matrices, the off-diagonal block $\idx{K}{\B_1}{\B_2}$ is compressed only if $\B_1$ and $\B_2$ are \emph{well-separated} as in the FMM --- that is, if $\B_1$ and $\B_2$ are separated by a distance of at least $D$ as in \cref{fig:admiss}.  In contrast, \emph{weakly-admissible} hierarchical matrices compress not only well-separated interactions but also interactions corresponding to DOFs in \emph{adjacent} boxes.  The inclusion of nearby interactions under weak admissibility typically increases the required approximation rank, but it also affords a much simpler geometric and algorithmic structure.

A number of more recent methods have been developed for hierarchically rank-structured matrices with the aim of more efficient practical performance based on weakly-admissible rank structure.  Examples include algorithms for hierarchical semi-separable (HSS) matrices \cite{fasthss,fastulv,superfast}, hierarchical off-diagonal low-rank (HODLR) matrices \cite{siva,mhodlr}, and methods based on recursive skeletonization (RS)  \cite{martinsson-rokhlin,domainsAd,rskel}, among other related schemes \cite{bremer,chen}\response{.}  In general, methods based strictly on \response{weak admissibility require} allowing ranks of off-diagonal blocks to grow non-negligibly with $N$ to attain a fixed target accuracy.
This has led to the development of more involved methods such as the hierarchical interpolative factorization (HIF) of Ho \& Ying \cite{hifie} and the method of Corona et al. \cite{corona2013}, which combine RS with additional compression steps based on geometric considerations to obtain greater efficiency at the cost of a more complicated algorithm.

There has been much less work on improved algorithms for solving \eqref{eq:linear} based directly on strong-admissibility.  The stand-out example is the recent ``inverse fast multipole method'' (IFMM) of \response{Coulier et al.} and Ambikasaran \& Darve \cite{ifmm1,ifmm2}, which assumes a general $\mathcal{H}^2$-matrix is given and provides a framework for approximately applying the inverse operator in the language of the FMM.  Further, a factorization based on block elimination and \response{strong admissibility} has been recently introduced by Sushnikova \& Oseledets \cite{oseledets} for the ``sparse analogue'' of our integral equation setting (that is, discretizations of elliptic partial differential equations).

\subsection{Contributions}
Based on the RS process of Martinsson \& Rokhlin \cite{martinsson-rokhlin} in the block-elimination form of Ho \& Ying \cite{hifie}, we introduce \emph{strong skeletonization}, an \response{extension} of skeletonization for the \response{strong admissibility} setting.  Using this in a recursive fashion like the original RS factorization, we develop the \emph{strong recursive skeletonization factorization}, an approximate factorization of $\m{K}$ into the product of many block unit-triangular matrices and a block diagonal matrix with time complexity linear in the number of DOFs, under suitable rank-scaling assumptions.  \response{Using low-accuracy approximations to off-diagonal blocks yields an effective preconditioner for iterative methods applied to \eqref{eq:linear}, whereas at higher accuracies the resulting factorization can be used as a direct solver.}

\response{Like the IFMM \cite[Appendix C]{ifmm2}, our factorization uses a bottom-up traversal of the quadtree or octree decomposition of space to compress well-separated interactions on a level-by-level basis.  This allows efficient on-the-fly construction of a nested ``skeleton'' basis for representing far-field interactions at different levels during the factorization process, in contrast to the typical recursive $\mathcal{H}$-matrix inversion algorithm.  Using skeletonization to maintain problem structure and exploit accelerated compression techniques (see \cref{sec:proxy}), we obtain what may be thought of as a multiplicative analogue of the FMM, using the same strong admissibility structure.  This gives a factorization of $\m{K}$ or $\m{K}^{-1}$ with simple constituent factors that is easy to understand and implement.}

\response{As an extension to our approach, we combine the original weak-admissibility-based skeletonization process with our strong-admissibility-based skeletonization and introduce the \emph{hybrid recursive skeletonization factorization}, which uses additional compression steps like HIF or the method of Corona et al. but does so without the need for spatial geometry beyond the boxes of the tree decomposition.  This additional compression reduces memory usage for practical performance gains in many cases.}

\section{Preliminaries}

In the remainder of this paper we adopt the following notation.  For a positive integer $N$, the index set $\{1,2,\dots,N\}$ is denoted by $[N]$.  We write matrices or matrix-valued functions in the sans serif font (\eg, $\m{A}\in\C^{N\times N}$) but make no such distinction for vectors (\eg, $x\in\C^N$).  Given a vector or matrix, the norms $\|x\|$ or $\|\m{A}\|$ refer to the standard Euclidean vector norm and corresponding induced matrix norm, respectively.  The math-calligraphic font is used to indicate index sets (\eg, $\I = \{i_1,i_2,\dots,i_r\}$ with each $i_j$ a positive integer) that we use to index blocks of a matrix (\eg, $\idx{A}{\I}{\J} = \m{A}(\I,\J) \in \C^{|\I|\times|\J|}$, using MATLAB\textsuperscript{\textregistered} notation).  Therefore, each index set has an implicit ordering, though we use the term ``set'' as opposed to \response{``vector''} to avoid conflation.  Because we are working with matrices discretizing integral equations, indices in an index set are typically associated with points in $\R^d$ (\eg, Nystr\"om or collocation points or centroids of elements).  As such, we will use the more general term ``DOF sets'' to refer to both the index set $\B$ and the corresponding points $\{x_i\}_{i\in\B}$ in $\R^d$.  Finally, to denote ordered sets of positive integers that are not associated with points in the domain nor used to index matrices we use the math-script font (\eg, $\L$).

\subsection{Block-structured elimination}\label{sec:block_elimination}
We begin with a brief review of block-structured elimination and its efficiency, which is central to the skeletonization algorithm.

Let $\m{A}\in\C^{N\times N}$ be an $N\times N$ matrix and suppose $[N] = \I\cup\J\cup\K$ is a partition of the index set of $\m{A}$ such that both $\idx{A}{\I}{\K} = 0$ and $\idx{A}{\K}{\I} = 0$, \ie, we have the block structure
\begin{align*}
\m{A} &= \left\lmat \begin{array}{l|l|l}
\idx{A}{\I}{\I}&\idx{A}{\I}{\J}& \\\hline
\idx{A}{\J}{\I}&\idx{A}{\J}{\J}&\idx{A}{\J}{\K} \\ \hline
&\idx{A}{\K}{\J}&\idx{A}{\K}{\K}\\
\end{array}\right\rmat,
\end{align*}
up to permutation.  Assuming that the block $\idx{A}{\I}{\I}$ is invertible, the DOFs $\I$ can be decoupled as follows.  First, define the matrices $\m{L}$ and $\m{U}$ as
\begin{align}\label{eq:landu}
\m{L} &\equiv \left\lmat \begin{array}{c|c|c}
\m{I}&& \\\hline
-\idx{A}{\J}{\I}^{\phantom{-1}}\idx{A}{\I}{\I}^{-1}&\m{I}& \\ \hline
&&\m{I}\\
\end{array}\right\rmat , \quad
\m{U} \equiv \left\lmat \begin{array}{c|c|c}
\m{I}&-\idx{A}{\I}{\I}^{-1}\idx{A}{\I}{\J}^{\phantom{-1}}& \\\hline
&\m{I}& \\ \hline
&&\m{I}\\
\end{array}\right\rmat .
\end{align}
with the same block partitioning as $\m{A}$.  Then, applying these operators on the left and right of $\m{A}$ yields
\begin{align}\label{eq:block_elimination}
\m{L}\m{A}\m{U} = \left\lmat \begin{array}{c|c|c}
\idx{A}{\I}{\I}&& \\\hline
&\idx{S}{\J}{\J}&\idx{A}{\J}{\K} \\ \hline
&\idx{A}{\K}{\J}&\idx{A}{\K}{\K}\\
\end{array}\right\rmat ,
\end{align}
where
$
\idx{S}{\J}{\J}=\idx{A}{\J}{\J} - \idx{A}{\J}{\I}^{\phantom{-1}}\idx{A}{\I}{\I}^{-1}\idx{A}{\I}{\J}^{\phantom{-1}}$
is the only nonzero block of the resulting matrix that has been modified.

We say that $\idx{S}{\J}{\J}$ is related to $\idx{A}{\J}{\J}$ through a \emph{Schur complement update}.  Note that, while we choose here to write block-elimination in its simplest form, in practice it can be numerically advantageous to work with a factorization of $\idx{A}{\I}{\I}$ as is done by Ho \& Ying \cite{hifie} as opposed to inverting the submatrix directly.  Either way, the cost of computing $\idx{S}{\J}{\J}$ is $\O(|\I|^3 + |\I|\cdot|\J|^2)$.

\subsection{Compression via the interpolative decomposition}
Another key linear algebra tool of which we will make heavy use is the \emph{interpolative decomposition} \cite{id}.
\begin{definition}\label{def:id}
  Given both a matrix $\idx{A}{\I}{\J}\in\C^{|\I|\times |\J|}$ with rows indexed by $\I$ and columns indexed by $\J$ and a tolerance $\epsilon > 0$, an $\epsilon$-accurate interpolative decomposition (ID) of $\idx{A}{\I}{\J}$ is a partitioning of $\J$ into DOF sets
  associated with so-called skeleton columns $\sknocolor\subset \J$ and redundant columns $\rdnocolor = \J \setminus \sknocolor$ and a corresponding interpolation matrix $\,\m{T}$ such that
\begin{align*}
\left\|\idx{A}{\I}{\rdnocolor}-\idx{A}{\I}{\sknocolor}\m{T}\right\| \le \epsilon\left\|\idx{A}{\I}{\J}\right\|,
\end{align*}
or equivalently, assuming $\idx{A}{\I}{\J} = \lmat\begin{array}{cc}\idx{A}{\I}{\rd} & \idx{A}{\I}{\sk} \end{array} \rmat$,
\begin{align*}
\left\|\idx{A}{\I}{\J} -\idx{A}{\I}{\sk}\lmat\begin{array}{cc}\m{T} & \m{I} \end{array}\rmat\right\| \le \epsilon\left\|\idx{A}{\I}{\J}\right\|.
\end{align*}
In other words, the redundant columns are approximated as a linear combination of the
skeleton columns to within the prescribed relative accuracy, leading to a low-rank factorization of $\idx{A}{\I}{\J}$.
\end{definition}

Note that, while the ID error bound can be attained trivially by taking $\sknocolor=\J$, it is desirable to keep $|\sknocolor|$ as small as possible. The typical algorithm to compute an ID uses a strong rank-revealing QR factorization as detailed by Gu \& Eisenstat\cite{rrqr}, though in practice a standard greedy column-pivoted QR tends to be sufficient.  In either case, the computational complexity is $\O(|\I|\cdot|\J|^2)$.

\section{The strong recursive skeletonization factorization (RS-S)}
\label{sec:rs}

We work in the context of a given tree decomposition of the domain (quadtree or octree) such that each leaf-level box of the tree contains a bounded number of DOFs independent of $N$.  When the tree is uniformly refined (\ie, each box has either $0$ or $2^d$ children and all leaf boxes are at the same level), it is straight-forward to use the strong admissibility criterion illustrated in \cref{fig:admiss} to identify which pairs of boxes are strongly-admissible and which are not, as in \cref{fig:neighbors}: given a box, the \emph{near-field} region for that box is the region corresponding to adjacent boxes at the same level of tree, and the \emph{far-field} region is the remainder of the domain.  The case where the tree is not uniformly refined is similar, though we will illustrate our method with uniform trees in what follows.

\begin{figure}
\centering
% Proxy surface
\scalebox{0.75}{
  \begin{tikzpicture}
    \clip[rounded corners](-3.5,-2.5) rectangle (4.5,3.5);
  \foreach \x in {-4,...,4}
      \foreach \y in {-3,...,3}
        \pgfmathifthenelse{\y<-1 || \y > 1 || \x < -1 || \x > 1}{"\noexpand\filldraw[fill=far,draw=black,thick,fill opacity=0.7] (\x,\y) rectangle (\x+1, \y+1);"}{}\pgfmathresult;
     \foreach \x in {-1,...,1}
      \foreach \y in {-1,...,1}
        \pgfmathifthenelse{\y<0 || \y > 0 || \x < 0 || \x > 0}{"\noexpand\filldraw[fill=near,draw=black,thick,fill opacity=0.7] (\x,\y) rectangle (\x+1, \y+1);"}{}\pgfmathresult;

        % \ifthenelse{\y<-1 \OR \y >1 \OR \x<-1 \OR \x>1}{\filldraw[fill=far,draw=black,thick] (\x,\y) rectangle (\x+1, \y+1)}{\filldraw[fill=near,draw=black,thick] (\x,\y) rectangle (\x+1, \y+1)};
        \draw[line width=2pt, draw=black] (-1,-1) rectangle (2,2);
  \filldraw[fill=box,draw=black,thick,fill opacity=0.7] (0,0) rectangle (1,1);
  \draw[line width=2pt, draw=black] (0,0) rectangle (1,1);
  \draw[use as bounding box,draw=black,line width=1mm, rounded corners](-3.5,-2.5) rectangle (4.5,3.5);
  \end{tikzpicture}
  }
\scalebox{0.75}{
  \begin{tikzpicture}
    \clip[rounded corners](-3.5,-2.5) rectangle (4.5,3.5);
\foreach \x in {-4,...,4}
      \foreach \y in {-3,...,3}
        \pgfmathifthenelse{\y<-2 || \y > 2 || \x < -2 || \x > 2}{"\noexpand\filldraw[fill=lightergray,draw=black,thick] (\x,\y) rectangle (\x+1, \y+1);"}{}\pgfmathresult;
      \foreach \x in {-2,...,2}
      \foreach \y in {-2,...,2}
        \pgfmathifthenelse{\y<-1 || \y > 1 || \x < -1 || \x > 1}{"\noexpand\filldraw[fill=far,draw=black,thick,fill opacity=0.7] (\x,\y) rectangle (\x+1, \y+1);"}{}\pgfmathresult;
     \foreach \x in {-1,...,1}
      \foreach \y in {-1,...,1}
        \pgfmathifthenelse{\y<0 || \y > 0 || \x < 0 || \x > 0}{"\noexpand\filldraw[fill=near,draw=black,thick,fill opacity=0.7] (\x,\y) rectangle (\x+1, \y+1);"}{}\pgfmathresult;
  % \foreach \x in {-4,...,4}
  %     \foreach \y in {-3,...,3}
  %       \ifthenelse{\y<-2 \OR \y >2 \OR \x<-2 \OR \x>2}{\filldraw[draw=black,fill=lightergray,thick] (\x,\y) rectangle (\x+1, \y+1)}{\filldraw[fill=near,draw=black,thick] (\x,\y) rectangle (\x+1, \y+1)};
  % \foreach \x in {-2,...,2}
  %     \foreach \y in {-2,...,2}
  %       \ifthenelse{\y<-1 \OR \y >1 \OR \x<-1 \OR \x>1}{\filldraw[fill=far,draw=black,thick] (\x,\y) rectangle (\x+1, \y+1)}{\filldraw[fill=near,draw=black,thick] (\x,\y) rectangle (\x+1, \y+1)};
  \filldraw[fill=box,draw=black,thick,fill opacity=0.7] (0,0) rectangle (1,1);

          \draw[line width=2pt, draw=black] (-1,-1) rectangle (2,2);

                  \draw[line width=2pt, draw=black] (0,0) rectangle (1,1);
  \draw[dotted,line width = 1mm, draw=black] (0.5,0.5) circle (2.5);
  \draw[use as bounding box,draw=black,line width=1mm, rounded corners](-3.5,-2.5) rectangle (4.5,3.5);
  \end{tikzpicture}
  }

  \caption{\label{fig:neighbors} Considering the DOFs $\B$ interior to the brown box on the left, the near-field DOFs $\N$ are those interior to the blue boxes and the far-field DOFs $\F$ are those interior to the red boxes.  Note that not all boxes in the far-field are shown.  On the right, we draw a ``proxy surface'' $\Gamma$ (the dashed circle) around the DOFs $\B$ such that the far-field $\F$ can be further decomposed into DOFs belonging to boxes inside the proxy surface, $\notprox$ (still red) and DOFs belonging to boxes outside the proxy surface, $\prox$.  }
\end{figure}
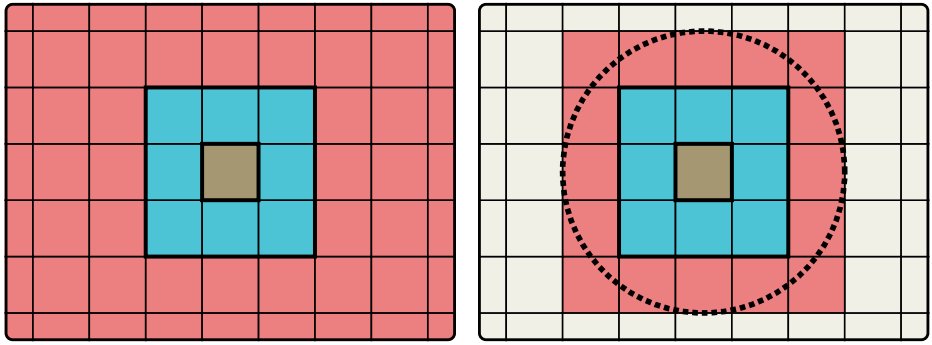

\subsection{Strong skeletonization}\label{sec:strongskel}
Given a matrix $\m{A}\in\C^{N\times N}$ indexed by points in our domain, we begin by selecting DOFs $\B$ corresponding to a leaf-level \response{box at} the finest level of the tree.

Letting $\N$ and $\F$ be the sets of near- and far-field DOFs as in \cref{fig:neighbors} (left), an appropriate permutation $\m{P}$ gives the block structure
\begin{align*}
\m{P}^*\m{A}\m{P} &= \left\lmat \begin{array}{l|l|l}
\idx{A}{\B}{\B}&\idx{A}{\B}{\N}&\idx{A}{\B}{\F} \\\hline
\idx{A}{\N}{\B}&\idx{A}{\N}{\N}&\idx{A}{\N}{\F} \\ \hline
\idx{A}{\F}{\B}&\idx{A}{\F}{\N}&\idx{A}{\F}{\F}\\
\end{array}\right\rmat.
\end{align*}
By assumption, the blocks $\idx{A}{\B}{\F}$ and $\idx{A}{\F}{\B}$ corresponding to the far-field interactions of $\B$ are numerically low-rank and thus compressible.  Given some tolerance $\epsilon$, we partition $\B$ into its redundant and skeleton DOFs $\B = \rd \cup \sk$ via the ID

\begin{align}\label{eq:skelid}
\left\lmat \begin{array}{c}
\idx{A}{\F}{\B}\\
\idx{A}{\B}{\F}^*
\end{array}\right\rmat &= \left\lmat \begin{array}{c c}
\idx{A}{\F}{\rd} & \idx{A}{\F}{\sk}\\
\idx{A}{\rd}{\F}^* & \idx{A}{\sk}{\F}^*
\end{array}\right\rmat \approx \left\lmat \begin{array}{c}
\idx{A}{\F}{\sk}\\
\idx{A}{\sk}{\F}^*
\end{array}\right\rmat \lmat\begin{array}{cc}\m{T} & \m{I}\end{array}\rmat,
\end{align}
which yields a skeleton set $\sk$ and interpolation matrix $\m{T}$ that can be used to represent both the columns of $\idx{A}{\F}{\B}$ and the rows of $\idx{A}{\B}{\F}$, \ie, $\idx{A}{\F}{\rd} \approx \idx{A}{\F}{\sk}\m{T}$ and $\idx{A}{\rd}{\F}\approx \m{T}^*\idx{A}{\sk}{\F}$.  Note that in \eqref{eq:skelid} we have assumed for clarity of exposition that the redundant DOFs $\rd$ are ordered first within $\B$ such that no further permutation is necessary.   We now partition blocks of $\m{P}^*\m{A}\m{P}$ according to this ID to obtain
\begin{align*}
\m{P}^*\m{A}\m{P}
&\approx \left\lmat \begin{array}{ll|l|l}
\idx{A}{\rd}{\rd}&\idx{A}{\rd}{\sk}&\idx{A}{\rd}{\N}&\m{T}^*\idx{A}{\sk}{\F} \\
\idx{A}{\sk}{\rd}&\idx{A}{\sk}{\sk}&\idx{A}{\sk}{\N}&\idx{A}{\sk}{\F} \\\hline
\idx{A}{\N}{\rd}&\idx{A}{\N}{\sk}&\idx{A}{\N}{\N}&\idx{A}{\N}{\F} \\ \hline
\idx{A}{\F}{\sk}\m{T}&\idx{A}{\F}{\sk}&\idx{A}{\F}{\N}&\idx{A}{\F}{\F}\\
\end{array}\right\rmat.
\end{align*}

Because of the explicit linear dependence between far-field blocks of the matrix, the redundant DOFs $\rd$ can now be decoupled from the far-field DOFs $\F$ using elementary block row and column operations.  Defining the elimination matrices
\begin{align*}
\m{U}_\m{T}\equiv\left\lmat \begin{array}{cc|c|c}
\m{I}&-\m{T}^*&&\\
&\m{I}&& \\\hline
&&\m{I}& \\ \hline
&&&\m{I} \\
\end{array}\right\rmat, \quad
\m{L}_\m{T} &\equiv\left\lmat \begin{array}{cc|c|c}
\m{I}&&&\\
-\m{T}&\m{I}&& \\\hline
&&\m{I}& \\ \hline
&&&\m{I} \\
\end{array}\right\rmat,
\end{align*}
we see that application of these operators on the left and right gives
\begin{align*}
\m{U}_\m{T}\left\lmat \begin{array}{ll|l|l}
\idx{A}{\rd}{\rd}&\idx{A}{\rd}{\sk}&\idx{A}{\rd}{\N}&\m{T}^*\idx{A}{\sk}{\F} \\
\idx{A}{\sk}{\rd}&\idx{A}{\sk}{\sk}&\idx{A}{\sk}{\N}&\idx{A}{\sk}{\F} \\\hline
\idx{A}{\N}{\rd}&\idx{A}{\N}{\sk}&\idx{A}{\N}{\N}&\idx{A}{\N}{\F} \\ \hline
\idx{A}{\F}{\sk}\m{T}&\idx{A}{\F}{\sk}&\idx{A}{\F}{\N}&\idx{A}{\F}{\F}\\
\end{array}\right\rmat \m{L}_\m{T} &\approx \left\lmat \begin{array}{ll|l|l}
\idx{X}{\rd}{\rd}&\idx{X}{\rd}{\sk}&\idx{X}{\rd}{\N}&\\
\idx{X}{\sk}{\rd}&\idx{A}{\sk}{\sk}&\idx{A}{\sk}{\N}&\idx{A}{\sk}{\F} \\\hline
\idx{X}{\N}{\rd}&\idx{A}{\N}{\sk}&\idx{A}{\N}{\N}&\idx{A}{\N}{\F} \\ \hline
&\idx{A}{\F}{\sk}&\idx{A}{\F}{\N}&\idx{A}{\F}{\F}\\
\end{array}\right\rmat,
\end{align*}
where the modified \response{nonzero} blocks marked with $\m{X}$ correspond to some mixing of the second row and column\response{, respectively,} with the first row and column as a consequence of the elimination.

\response{Using $\idx{X}{\rd}{\rd}$ as a pivot block to eliminate the other blocks in the first row and column (\ie, performing block elimination as in \cref{sec:block_elimination} with $\I = \rd$, $\J=\sk\cup\N$, and $\K = \F$) we define the corresponding matrices $\m{L}$ and $\m{U}$ as in \eqref{eq:landu} to obtain}
\begin{align}\label{eq:zab}
\m{L}\m{U}_\m{T}\m{P}^*\m{A}\m{P}\m{L}_\m{T}\m{U}&\approx\left\lmat \begin{array}{ll|l|l}
\idx{X}{\rd}{\rd}&&&\\
&\idx{X}{\sk}{\sk}&\idx{X}{\sk}{\N}&\idx{A}{\sk}{\F} \\\hline
&\idx{X}{\N}{\sk}&\idx{X}{\N}{\N}&\idx{A}{\N}{\F} \\ \hline
&\idx{A}{\F}{\sk}&\idx{A}{\F}{\N}&\idx{A}{\F}{\F}\\
\end{array}\right\rmat \equiv \SF{\m{A}}{\B},
\end{align}
whereupon we see that the redundant DOFs $\rd$ are now completely decoupled from the rest of the problem.

We refer to this process as \emph{strong skeletonization} of $\m{A}$ with respect to the DOFs $\B$, as it is a direct modification of the skeletonization procedure of Martinsson \& Rokhlin \cite{martinsson-rokhlin} for the strong admissibility setting using the multiplicative formulation of Ho \& Ying \cite{hifie}.  We note that, while ID-based compression of far-field interactions has been used in the context of kernel-independent FMMs \cite{mrfmm,pan}, its use for the construction of (approximate) direct solvers is novel.

For a purely notational convenience, we will define the left and right skeletonization operators $\m{V}$ and $\m{W}$ as
\begin{align}\label{eq:skelop}
\m{V} \equiv \m{P}\m{U}^{-1}_\m{T}\m{L}^{-1},\quad \m{W} \equiv \m{U}^{-1}\m{L}^{-1}_\m{T}\m{P}^*,
\end{align}
with the understanding that these matrices will always be stored and used in the factored form given for efficiency.  In particular, recall that the block unit-triangular matrices $\m{U}$, $\m{U}_\m{T}$ and so on may be inverted by toggling the sign of the nonzero off-diagonal block.
With this shorthand we obtain
\begin{align}\label{eq:zabcompact}
\SF{\m{A}}{\B}\approx \m{V}^{-1}\m{A}\m{W}^{-1},
\end{align}
a more compact representation of $\SF{\m{A}}{\B}$ in \eqref{eq:zab}.

\subsection{The use of a proxy surface}\label{sec:proxy}
For optimal \response{complexity} it is desirable to avoid computation with blocks indexed by $\F$ in the construction of $\SF{\m{A}}{\B}$, since $|\F|$ is in general large.   By design, the only part of computing $\SF{\m{A}}{\B}$ that involves blocks indexed by $\F$ is constructing the ID, that is, finding the \response{partition} $\B=\rd\cup\sk$ and the interpolation matrix $\m{T}$ in \eqref{eq:skelid}.  For simplicity, we will drop the block $\idx{A}{\B}{\F}^*$ in this section and explain how the subblock $\idx{A}{\F}{\B}$ can be compressed in an indirect way that is more efficient than operating on $\idx{A}{\F}{\B}$ directly.  The \response{``transpose''} of these ideas can be used for the full stacked matrix including $\idx{A}{\B}{\F}^*$.

For concreteness, consider the case where $\m{A}\equiv\m{K}$ is a discretization such as \eqref{eq:discbie} with $b(x) \equiv c(x)\equiv 1$ and the 2D Laplace kernel $\response{K(z)=-\frac{1}{2\pi}\log(\|z\|)}$.  In this case, entries of $\idx{A}{\F}{\B}$ are given (up to a factor of $1/N$, which we drop in our discussion) directly by $K(x_i - x_j)$ for $x_i\in\F$ and $x_j \in \B$.

 Suppose $\Gamma$ is a circle (or sphere in 3D) enclosing the DOF set $\B$ with radius normalized to a multiple of 5/2 the sidelength of the corresponding box.  As in \cref{fig:neighbors} (right), we partition the far-field DOFs of $\B$ as $\F=\notprox\cup\prox,$ where
\begin{align*}
\prox\equiv \{i\in\F | \text{ $i$ is contained in a box entirely outside of $\Gamma$}\}
\end{align*}
and $\notprox \equiv \F\setminus\prox$, \ie, $\notprox$ is the set of indices corresponding to the \response{square of boxes containing the proxy surface} in the figure.

Define $\phi_i(x) \equiv K(x_i - x)$ for $x_i\in\prox$ such that $\phi_i(x_j)$ is the ``incoming'' harmonic field generated at $x_j\in\B$ by a source at $x_i$.   Because the DOFs in $\B$ are contained in the closed region with boundary $\Gamma$ and the DOFs in $\prox$ are contained in the complementary region, we may (under mild assumptions \cite{mclean}) use a form of Green's identity to write $\phi_i(x_j)$ for any $x_j\in\B$ in terms of a density $\psi_i(y)$ on $\Gamma$ as
\begin{align}\label{eq:greens}
\phi_i(x_j)&= \int_\Gamma \psi_i(y)K(x_j-y)\,dy,
\end{align}
where $\psi_i(y)$ depends on $x_i$ but not on $x_j$.

Because $\phi_i(x_j) = K(x_i - x_j)$ for $x_i\in\prox\subset\F$ and $x_j\in\B$, the left-hand side of \eqref{eq:greens} is an entry of $\idx{A}{\F}{\B}$.  The right-hand side is a so-called ``single-layer'' representation of this entry.  Discretizing this representation by replacing the analytic integral over $\Gamma$ with numerical integration using $n_p$ points $y_1,\dots,y_{n_p}$, we see that, up to discretization error,
\begin{align}\label{eq:proxytrick}
\idx{A}{\F}{\B} = \left\lmat \begin{array}{c}\idx{A}{\notprox}{\B} \\ \idx{A}{\prox}{\B} \end{array}\right\rmat \approx \left\lmat\begin{array}{cc} \m{I} & \\ & \idx{M}{\prox}{\Gamma} \end{array}\right\rmat\left\lmat \begin{array}{c}\idx{A}{\notprox}{\B} \\ \idx{G}{\Gamma}{\B} \end{array}\right\rmat,
\end{align}
where $\idx{G}{\Gamma}{\B}$ has entries $K(x_j-y_i)$ for $x_j\in\B$ and $\idx{M}{\prox}{\Gamma}$ is the matrix that approximately maps $\idx{G}{\Gamma}{\B}$ to $\idx{A}{\prox}{\B}$ via a discretization of \eqref{eq:greens} for each $x_i\in\prox$.  We do not give an explicit form of $\idx{M}{\prox}{\Gamma}$, as we need only be assured of its existence for what follows.

The \emph{proxy trick} for accelerated compression, which is heavily employed in the literature \cite{martinsson-rokhlin,id,domainsAd,corona2013,gg,rskel,hifie,pan,kifmm,mrfmm}, makes use of two key observations regarding \eqref{eq:proxytrick}.  Firstly, if $n_p \ll |\prox|$ (for example, if we take $n_p=\O(1)$ and $N$ to be large), then it is relatively inexpensive to compute the ID
\begin{align}\label{eq:smallid}
\left\lmat \begin{array}{c}\idx{A}{\notprox}{\B} \\ \idx{G}{\Gamma}{\B} \end{array}\right\rmat &= \left\lmat \begin{array}{cc}\idx{A}{\notprox}{\rd} & \idx{A}{\notprox}{\sk} \\ \idx{G}{\Gamma}{\rd} & \idx{G}{\Gamma}{\sk}\end{array}\right\rmat \approx \left\lmat \begin{array}{c}\idx{A}{\notprox}{\sk} \\ \idx{G}{\Gamma}{\sk} \end{array}\right\rmat\left\lmat \begin{array}{cc}\m{T} & \m{I} \end{array}\right\rmat.
\end{align}
Furthermore, because $\Gamma$ is in the far-field of $\B$, $|\sk|$ should be small by assumption.
Secondly, using the discrete Green's identity represented by $\idx{M}{\prox}{\Gamma}$, combining $\eqref{eq:proxytrick}$ and \eqref{eq:smallid} yields
\begin{align*}
\idx{A}{\F}{\B} \approx \left\lmat\begin{array}{cc} \m{I} & \\ & \idx{M}{\prox}{\Gamma} \end{array}\right\rmat\left\lmat \begin{array}{c}\idx{A}{\notprox}{\sk} \\ \idx{G}{\Gamma}{\sk} \end{array}\right\rmat\left\lmat \begin{array}{cc}\m{T} & \m{I} \end{array}\right\rmat \approx \idx{A}{\F}{\sk}\left\lmat\begin{array}{cc} \m{T} & \m{I}\end{array}\right\rmat,
\end{align*}
\ie, the partitioning $\B = \rd\cup\sk$ and interpolation matrix $\m{T}$ in \eqref{eq:smallid} also give an ID of $\idx{A}{\F}{\B}$.  We caution that the amplification of the approximation error in this ID depends on $\|\idx{M}{\prox}{\Gamma}\|$, among other factors, but note that this does not appear to be an issue in practice (see \cref{sec:results}).

By using these ideas or modifications thereof to obtain the ID \eqref{eq:skelid} as opposed to using a rank-revealing QR on the far-field blocks directly, the complexity of compression now depends on the number of proxy points $n_p$ and not on the total number of points in the domain and is thus substantially reduced.  To apply this acceleration, we require only a way to evaluate kernel interactions with $\Gamma$ and that the entries of $\idx{A}{\prox}{\B}$ satisfy a Green's identity.  While we have discussed only the Laplace kernel, the same trick holds for, \eg, the Stokes, or elasticity kernel.  A more thorough discussion of the use of a proxy surface can be found in Ho \& Ying\cite{hifie}, which discusses the case of more general $a(x)$ and $b(x)$ in \eqref{eq:bie}.

\subsection{Algorithm and complexity}
We turn now to a 1D sketch of our factorization approach using strong skeletonization.  Suppose we wish to solve an integral equation over the unit line segment $\Omega = [0,1]$ \response{using the trapezoid rule} to construct the matrix $\m{K}$.  Partitioning $\Omega$ into eight subintervals with corresponding DOF sets $\B_i$ for \response{$i=1,\dots,8$}, we consider the first DOF set $\B_1$ and its corresponding near-field DOFs $\N_1$ and far-field DOFs $\F_1$.  This gives a block partitioning and labeling of $\response{\m{K}}$ as in \cref{fig:1d} (top-left).  We use the strong skeletonization algorithm of \cref{sec:strongskel} to decouple redundant DOFs in $\B_1$ to approximately obtain the new matrix
\begin{align}\label{eq:skel}
\SF{\m{K}}{\B_1} =\left\lmat \begin{array}{ll|l|l}
\idx{X}{\rd_1}{\rd_1}&&&\\
&\idx{X}{\sk_1}{\sk_1}&\idx{X}{\sk_1}{\N_1}&\idx{K}{\sk_1}{\F_1} \\\hline
&\idx{X}{\N_1}{\sk_1}&\idx{X}{\N_1}{\N_1}&\idx{K}{\N_1}{\F_1} \\ \hline
&\idx{K}{\F_1}{\sk_1}&\idx{K}{\F_1}{\N_1}&\idx{K}{\F_1}{\F_1}\\
\end{array}\right\rmat
\end{align}
as in \eqref{eq:zabcompact}.
We call the redundant DOFs $\rd_1$ \emph{inactive}, because they no longer play an active role in our factorization procedure.  In contrast, any DOFs that have not yet been decoupled will be referred to as \emph{active}.

Moving on to the next set of DOFs $\B_2$ with near-field DOFs $\N_2$ and far-field DOFs $\F_2$, we observe that most of the nonzero entries of the matrix $\SF{\m{K}}{\B_1}$ are unchanged from their original value in $\response{\m{K}}$, see \cref{fig:1d} (top-right).  It is therefore reasonable to perform strong skeletonization of this new matrix with respect to $\B_2$ and expect compression of the corresponding far-field interactions.  This renders the redundant DOFs $\rd_2$ inactive and yields the new matrix
\begin{align*}
\SF{\m{K}}{\B_1,\B_2} \equiv \SF{\SF{\m{K}}{\B_1}}{\B_2},
\end{align*}
where we define $\SF{\m{K}}{\I_1,\I_2,\dots,\I_k}$ to be the result of skeletonizing $\m{K}$ with respect to the DOFs $\I_1$, skeletonizing the result with respect to $\I_2$, and so on.
In successive steps of strong skeletonization we skeletonize $\SF{\m{K}}{\B_1,\B_2}$ with respect to the DOFs $\B_3$ through $\B_8$ as in \cref{fig:1d} (bottom).

\begin{figure}[h!]
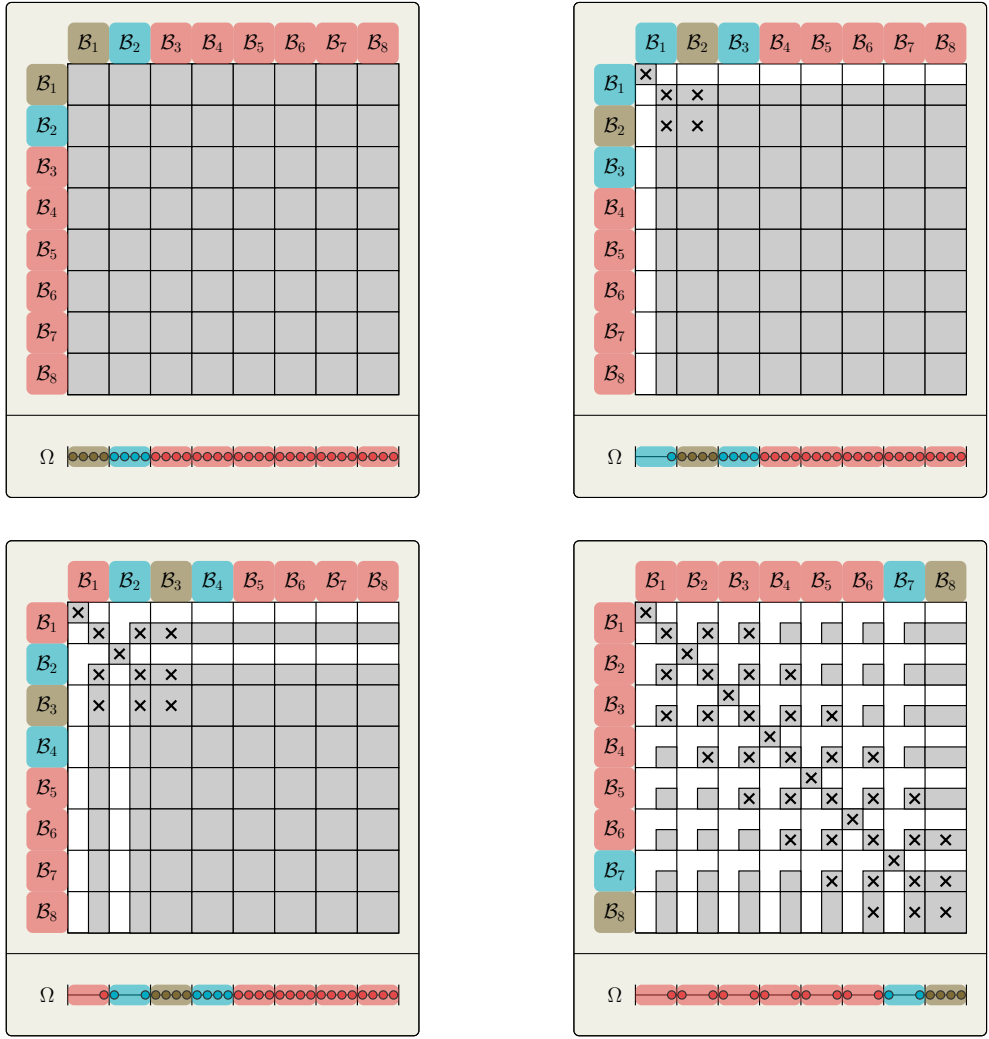

\centering
\begin{minipage}[b]{0.45\linewidth}\centering
\mypic{0}
\vspace*{1.5em}
\mypic{2}
\end{minipage}
\begin{minipage}[b]{0.45\linewidth}\centering
\mypic{1}
\vspace*{1.5em}
\mypic{7}
\end{minipage}
  \caption{We partition the 1D domain $\Omega$ into eight subintervals and block the corresponding matrix $\m{K}$ accordingly. Identifying the DOFs $\B_1$ (brown), $\N_1$ (blue), and $\F_1$ (red) in the top-left figure, we skeletonize with respect to $\B_1$.  Using the same color scheme for box, near-field, and far-field DOFs in the top-right figure, we see that some redundant DOFs have been completely decoupled from the rest and some blocks of the matrix have been modified through Schur complement updates (blocks marked with ``X'').  We proceed to skeletonize with respect to the DOFs $\B_2$, and then $\B_3$ (bottom-left figure) all the way through to $\B_8$ (bottom-right figure).  Below each matrix, we show \response{only} the remaining active DOFs at that step, \ie, we do not show the redundant DOFs corresponding to decoupled diagonal blocks.\label{fig:1d}}
\end{figure}

After skeletonization with respect to $\B_8$, the matrix $\SF{\m{K}}{\B_1,\B_2,\dots,\B_8}$ has many diagonal blocks corresponding to completely decoupled redundant DOF sets $\rd_i$ for $\response{i=1,\dots,8}$ as well as many blocks corresponding to interactions between the remaining active skeleton DOF sets $\sk_i$.  We construct a new partitioning of $\Omega$ into 4 subintervals and define the DOF sets
\begin{align}\label{eq:newDOFs}
\B_9 \equiv \sk_1\cup\sk_2, \quad \B_{10}\equiv\sk_3\cup\sk_4, \quad \B_{11} \equiv \sk_5\cup\sk_6, \quad\B_{12} \equiv \sk_7\cup\sk_8.
\end{align}
Permuting $\SF{\m{K}}{\B_1,\B_2,\dots,\B_8}$ such that these DOF sets are contiguous with the inactive DOF sets $\rd_i$ for \response{$i=1,\dots,8$} permuted to the end for visualization purposes, we obtain a matrix as in \cref{fig:1dlevel2} (left), at which point, we may skeletonize successively with respect to $\B_9$ through $\B_{12}$.  It is not possible to again double the size of a subinterval and expose any further compressible blocks, so we stop.  The final matrix
$\SF{\m{K}}{\B_1,\B_2,\dots,\B_{12}}$ can be permuted to a block-diagonal matrix with small blocks defined by the DOF sets $\rd_i$ for \response{$i=1,\dots,12$} and the DOF set $\sk_9\cup\sk_{10}\cup\sk_{11}\cup\sk_{12}$.  This block-diagonal structure can then be exploited to efficiently solve the linear system \eqref{eq:linear}.

\begin{figure}[h!]
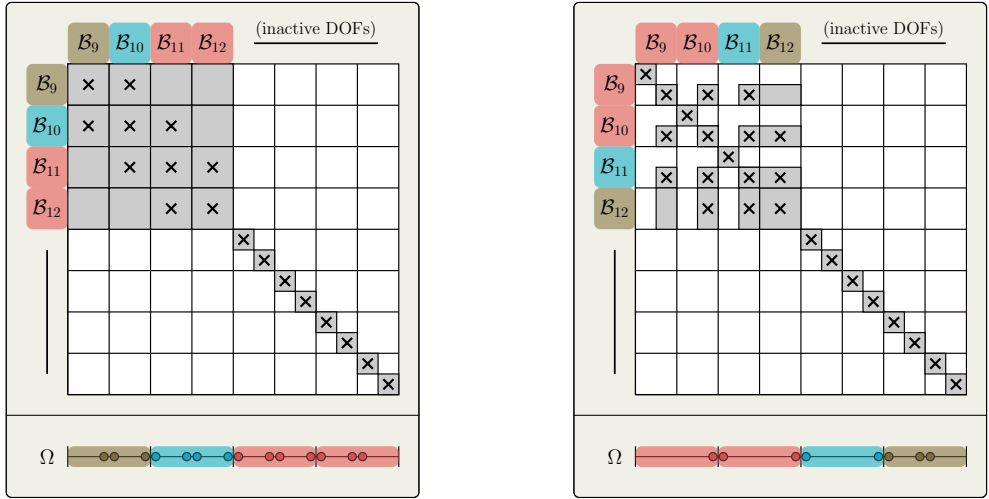

\centering
\begin{minipage}[b]{0.45\linewidth}\centering
\mypicnext
\end{minipage}
\begin{minipage}[b]{0.45\linewidth}\centering
\mypicnextnext
\end{minipage}
  \caption{After skeletonization with respect to the DOFs $\B_8$ in \cref{fig:1d}, we permute all decoupled redundant DOF \response{sets} $\rd_1,\dots,\rd_8$ to the end and define the next-level DOF sets $\B_9,\dots,\B_{12}$ as in \eqref{eq:newDOFs} such that blocks of the matrix in the left figure correspond to aggregating blocks from the previous level.  Using the same color scheme as in \cref{fig:1d}, in skeletonizing with respect to $\B_9$, we see that blocks of interactions between $\B_9$ (brown) and $\N_9$ (blue) have modifications from Schur complement updates (marked with ``X'') from the previous level.  We continue to skeletonize with respect to DOF sets at this level up through to $\B_{12}$ (right figure), decoupling additional redundant DOFs as we go.\label{fig:1dlevel2}}
\end{figure}

\subsubsection{The general case: first level}\label{sec:leaf}
Having given the flavor of our approach in 1D, we flesh out the details for the more general case.  Suppose the integral equation \eqref{eq:bie} is discretized over $\Omega\subset \R^d$ for $d=2$ or $3$.  Given a tree decomposition of the domain such that each leaf box contains a constant number of unknowns, we number the levels of \response{the} tree starting from the finest level $(\ell = 1)$ up to the root level $(\ell = L)$.  We require a fixed but arbitrary \response{bottom-up level-by-level traversal} of the tree and order the boxes accordingly such that a box at level $1$ is ordered before any box at level $2$ and so on. This ordering on all boxes of the tree induces corresponding orderings on the boxes within each level of the tree,  $\L_\ell$ for $\ell = 1,\dots,L$. For example, in the case of a regular grid with $2^{d(L-\ell)}$ boxes at level $\ell$ we obtain the orderings
\begin{align*}
\L_1 &= \left\{1,\, 2,\, \dots,\, 2^{d(L-1)}\right\}, \\
\L_2 &= \left\{2^{d(L-1)}+1,\, 2^{d(L-1)}+2,\, \dots,\, 2^{d(L-1)}+2^{d(L-2)}\right\},
\end{align*}
and so on.  \response{We do not require a regular grid of discretization points, but use a regular grid in all figures for illustration.  Note in particular that this implies \response{leaf boxes} may belong to levels other than $\ell=1$ in the general case.}

Beginning at level $\ell=1$, we select the first leaf box and label the corresponding DOFs as $\B_1$ with near-field DOFs $\N_1$ and far-field DOFs $\F_1$.  We decouple and render inactive redundant DOFs in $\B_1$ through strong skeletonization to obtain
\begin{align*}
\m{V}_1^{-1}\m{K}\m{W}_1^{-1} \approx \SF{\m{K}}{\B_1}
\end{align*}
with $\SF{\m{K}}{\B_1}$ as in \eqref{eq:skel} and $\m{V}_1$ and $\m{W}_1$ the left and right skeletonization operators corresponding to $\B_1$ as in \eqref{eq:skelop}.

Selecting the next box with corresponding DOFs $\B_2$, we define the DOF sets $\N_2$ and $\F_2$ as
\begin{align*}
\N_2 &\equiv \{\text{\emph{active} DOFs in the near-field of $\B_2$}\},\\
\F_2 &\equiv \{\text{\emph{active} DOFs in the far-field of $\B_2$}\},
\end{align*}
so as to avoid unnecessary further computation with the inactive DOFs indexing the first block row and column of $\SF{\m{K}}{\B_1}$.
To efficiently perform strong skeletonization of $\m{A}=\SF{\m{K}}{\B_1}$ with respect to the DOFs $\B_2$, it is necessary to assume that the blocks $\idx{A}{\B_2}{\F_2}$ and $\idx{A}{\F_2}{\B_2}$ are still compressible.  Note that this is not immediate, as \emph{these blocks need not be original blocks of $\m{K}$}.  In particular, if $\B_2\subset\N_1$ and $\F_2\cap\N_1 \ne \emptyset$, then the block $\idx{X}{\N_1}{\N_1}$ in \eqref{eq:skel} includes updated interactions between $\B_2$ and $\F_2\cap\N_1$, which has the potential to increase the numerical rank of said interactions.  \response{This is illustrated in \cref{fig:schurupdates}.}

Note that, drawing the proxy surface $\Gamma$ around $\B_2$, it is still true due to geometric considerations that the interactions $\idx{A}{\prox_2}{\B_2}$ between $\B_2$ and far-field points $\prox_2$ contained in boxes outside of $\Gamma$ are unchanged and, therefore, the accelerated compression scheme in \cref{sec:proxy} using a proxy surface is still justified.

\begin{figure}
\centering
% Proxy surface
\scalebox{0.75}{
  \begin{tikzpicture}
    \clip[rounded corners](-3.5,-2.5) rectangle (4.5,3.5);
\foreach \x in {-4,...,4}
      \foreach \y in {-3,...,3}
        \pgfmathifthenelse{\y<-2 || \y > 2 || \x < -2 || \x > 2}{"\noexpand\filldraw[fill=lightergray,draw=black,thick] (\x,\y) rectangle (\x+1, \y+1);"}{}\pgfmathresult;
      \foreach \x in {-2,...,2}
      \foreach \y in {-2,...,2}
        \pgfmathifthenelse{\y<-1 || \y > 1 || \x < -1 || \x > 1}{"\noexpand\filldraw[fill=far,draw=black,thick,fill opacity=0.7] (\x,\y) rectangle (\x+1, \y+1);"}{}\pgfmathresult;
     \foreach \x in {-1,...,1}
      \foreach \y in {-1,...,1}
        \pgfmathifthenelse{\y<0 || \y > 0 || \x < 0 || \x > 0}{"\noexpand\filldraw[fill=near,draw=black,thick,fill opacity=0.7] (\x,\y) rectangle (\x+1, \y+1);"}{}\pgfmathresult;
  % \foreach \x in {-4,...,4}
  %     \foreach \y in {-3,...,3}
  %       \ifthenelse{\y<-2 \OR \y >2 \OR \x<-2 \OR \x>2}{\filldraw[draw=black,fill=lightergray,thick] (\x,\y) rectangle (\x+1, \y+1)}{\filldraw[fill=near,draw=black,thick] (\x,\y) rectangle (\x+1, \y+1)};
  % \foreach \x in {-2,...,2}
  %     \foreach \y in {-2,...,2}
  %       \ifthenelse{\y<-1 \OR \y >1 \OR \x<-1 \OR \x>1}{\filldraw[fill=far,draw=black,thick] (\x,\y) rectangle (\x+1, \y+1)}{\filldraw[fill=near,draw=black,thick] (\x,\y) rectangle (\x+1, \y+1)};
  \filldraw[fill=box,draw=black,thick,fill opacity=0.7] (0,0) rectangle (1,1);

          \draw[line width=2pt, draw=black] (-1,-1) rectangle (2,2);

                  \draw[line width=2pt, draw=black] (0,0) rectangle (1,1);
  \draw[dotted,line width = 1mm, draw=black] (0.5,0.5) circle (2.5);
  \draw (0.5,0.5) node {\Large $\B_1$};
  \draw (-0.5,-0.5) node {\Large $\B_2$};
  \draw (-0.5,1.5) node {\Large $\B_3$};
  \draw[use as bounding box,draw=black,line width=1mm, rounded corners](-3.5,-2.5) rectangle (4.5,3.5);
  \end{tikzpicture}
  }
\scalebox{0.75}{
  \begin{tikzpicture}
    \clip[rounded corners](-2.5,-1.5) rectangle (5.5,4.5);
\foreach \x in {-4,...,5}
      \foreach \y in {-3,...,4}
        \pgfmathifthenelse{\y<-2 || \y > 2 || \x < -2 || \x > 2}{"\noexpand\filldraw[fill=lightergray,draw=black,thick] (\x,\y) rectangle (\x+1, \y+1);"}{}\pgfmathresult;
      \foreach \x in {-2,...,2}
      \foreach \y in {-2,...,2}
        \pgfmathifthenelse{\y<-1 || \y > 1 || \x < -1 || \x > 1}{"\noexpand\filldraw[fill=far,draw=black,thick,fill opacity=0.7] (\x,\y) rectangle (\x+1, \y+1);"}{}\pgfmathresult;
     \foreach \x in {-1,...,1}
      \foreach \y in {-1,...,1}
        \pgfmathifthenelse{\y<0 || \y > 0 || \x < 0 || \x > 0}{"\noexpand\filldraw[fill=near,draw=black,thick,fill opacity=0.7] (\x,\y) rectangle (\x+1, \y+1);"}{}\pgfmathresult;
  % \foreach \x in {-4,...,4}
  %     \foreach \y in {-3,...,3}
  %       \ifthenelse{\y<-2 \OR \y >2 \OR \x<-2 \OR \x>2}{\filldraw[draw=black,fill=lightergray,thick] (\x,\y) rectangle (\x+1, \y+1)}{\filldraw[fill=near,draw=black,thick] (\x,\y) rectangle (\x+1, \y+1)};
  % \foreach \x in {-2,...,2}
  %     \foreach \y in {-2,...,2}
  %       \ifthenelse{\y<-1 \OR \y >1 \OR \x<-1 \OR \x>1}{\filldraw[fill=far,draw=black,thick] (\x,\y) rectangle (\x+1, \y+1)}{\filldraw[fill=near,draw=black,thick] (\x,\y) rectangle (\x+1, \y+1)};
  \filldraw[fill=box,draw=black,thick,fill opacity=0.7] (0,0) rectangle (1,1);

          \draw[line width=2pt, draw=black] (-1,-1) rectangle (2,2);

                  \draw[line width=2pt, draw=black] (0,0) rectangle (1,1);
  \draw[dotted,line width = 1mm, draw=black] (0.5,0.5) circle (2.5);

  \draw (1.5,1.5) node {\Large $\sk_1$};
  \draw (0.5,0.5) node {\Large $\B_2$};
  \draw (0.5,2.5) node {\Large $\B_3$};

  \draw[use as bounding box,draw=black,line width=1mm, rounded corners](-2.5,-1.5) rectangle (5.5,4.5);
  \end{tikzpicture}
  }

  \caption{\label{fig:schurupdates} \response{Left: Because the DOF sets $\B_2$ and $\B_3$ are both in the near-field of $\B_1$, skeletonization with respect to $\B_1$ leads to updated interactions between $\B_2$ and $\B_3$ as a subblock of $\idx{X}{\N_1}{\N_1}$ in \eqref{eq:skel}.  Right: Considering now the next box, we see the DOF set $\B_3$ is in the far-field of $\B_2$, so the previous Schur complement update has led to modified far-field interactions that must be compressed when skeletonizing with respect to $\B_2$.  However, interactions between $\B_2$ and DOFs corresponding to boxes outside the proxy surface $\Gamma$ are still unmodified due to geometric considerations as guaranteed by \cref{thm:farprox}. }}
\end{figure}
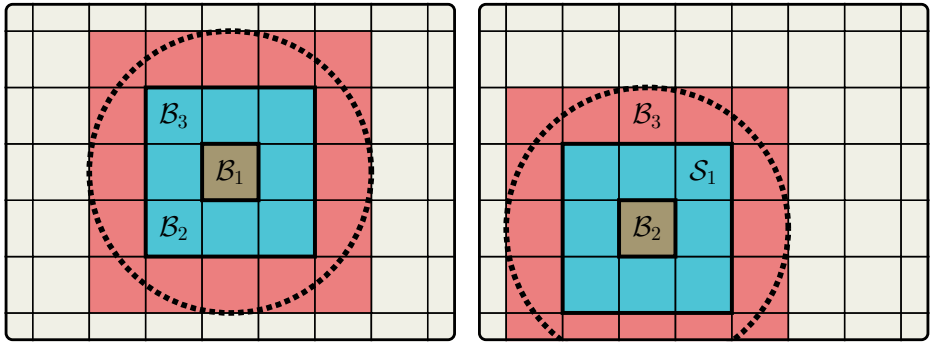

\begin{theorem}\label{thm:farprox}
Skeletonization with respect to the DOFs $\B_i$ does not modify interactions between $\B_j$ and $\prox_j$ for $j\ge i$.  That is, if
$\m{A} = \SF{\m{K}}{\B_1,\B_2,\dots,\B_i},$
then
\begin{align*}
\left\lmat\begin{array}{c}\idx{A}{\prox_j}{\B_j} \\ \idx{A}{\B_j}{\prox_j}^*\end{array}\right\rmat &= \left\lmat\begin{array}{c}\idx{K}{\prox_j}{\B_j} \\ \idx{K}{\B_j}{\prox_j}^*\end{array}\right\rmat.
\end{align*}
\end{theorem}
\begin{proof}
Suppose $D$ is the sidelength of the box with corresponding DOFs $\B_i$ and consider skeletonizing $\m{A}$ with respect to $\B_i$. As in \eqref{eq:skel}, we see the only \response{updated interactions between active DOFs} are between $\sk_i$ and $\N_i$.  However, by definition of $\prox_j$ we know that $\B_j$ and $\prox_j$ correspond to DOF sets separated by at least two boxes of sidelength $D'\ge D$ due to the fact that $j\ge i$ and our ordering corresponds to a \response{bottom-up} traversal of the tree.   Therefore, since the near-field DOFs $\N_i$ span a distance of no more than $3D$ in each axial direction, either $\prox_j\cap(\sk_i\cup\N_i) = \emptyset$ or $\B_j\cap(\sk_i\cup\N_i)=\emptyset$, which implies skeletonization with respect to $\B_i$ did not modify interactions between $\B_j$ and $\prox_j$.\qquad
\end{proof}

With the above in mind, we skeletonize $\m{A}$ with respect to $\B_2$ and obtain
\begin{align*}
\m{V}_2^{-1}\m{A}\m{W}_2^{-1}\approx \SF{\m{A}}{\B_2} = \SF{\m{K}}{\B_1,\B_2}
\end{align*}
 with
\begin{align*}
 \SF{\m{A}}{\B_2} =\left\lmat \begin{array}{l|ll|l|l}
\idx{X}{\rd_1}{\rd_1}& & & &\\\hline
&\idx{X}{\rd_2}{\rd_2}&&&\\
&&\idx{X}{\sk_2}{\sk_2}&\idx{X}{\sk_2}{\N_2}&\idx{A}{\sk_2}{\F_2} \\\hline
&&\idx{X}{\N_2}{\sk_2}&\idx{X}{\N_2}{\N_2}&\idx{A}{\N_2}{\F_2} \\ \hline
&&\idx{A}{\F_2}{\sk_2}&\idx{A}{\F_2}{\N_2}&\idx{A}{\F_2}{\F_2}\\
\end{array}\right\rmat,
\end{align*}
where the DOFs $\rd_2$ have been made inactive as well.
We note that, while the nonzero entries in the last block row and block column are unmodified from what they were in $\m{A}=\SF{\m{K}}{\B_1}$, this does not necessarily mean they are unmodified from what they were in $\m{K}$. \response{For example, consider that in \cref{fig:schurupdates} we have $\sk_1 \subset \N_2$ and $\B_3\subset\F_2$, but skeletonization with respect to $\B_1$ led to modified interactions between $\sk_1$ and $\B_3$.}

Proceeding as in the 1D case, we loop over each box at \response{level $\ell=1$}, identify its corresponding DOFs $\B_i$ and active near- and far-field DOFs $\N_i$ and $\F_i$, and perform strong skeletonization using the proxy trick to capture interactions with $\prox_i$ implicitly.  This process can be \response{seen in} \cref{fig:boxesleaf}, where the subfigures show all active DOFs at various times during the skeletonization at level $\response{\ell=1}$ of a regular discretization over the unit square.  Supposing that there are $r$ boxes at level \response{$\ell=1$} (\ie, $\L_1 = \{1,2,\dots,r\} = [r]$), and defining $[r]'\equiv\{r,r-1,\dots,1\}$ as the reversal of $[r]$, the resulting matrix at this point is
% \begin{align}\label{eq:boxm}
% \SF{\m{K}}{\B_1,\B_2,\dots,\B_{n_1}} \approx \left(\prod_{i=1}^{n_1} \m{V}_i^{-1}\right)\m{K}\left(\prod_{i=n_1}^1\m{W}_i^{-1}\right),
% \end{align}
\begin{align}\label{eq:boxm}
\begin{split}
\SF{\m{K}}{\B_1,\B_2,\dots,\B_{r}} &\approx \left(\m{V}_r^{-1}\dots \m{V}_2^{-1}\m{V}_1^{-1}\right)\m{K}\left(\m{W}_1^{-1}\m{W}_2^{-1}\dots\m{W}_r^{-1}\right)\\
&\equiv\left(\prod_{i\in[r]'}\m{V}_i^{-1}\right)\m{K}\left(\prod_{i\in[r]}\m{W}_i^{-1}\right).\end{split}
\end{align}

Note that in $\SF{\m{K}}{\B_1,\B_2,\dots,\B_{r}}$ each redundant DOF set $\rd_i$ is completely decoupled from the rest of the problem, but each skeleton DOF set $\sk_i$ remains coupled to all other skeleton DOFs.

\begin{figure}
\centering
\begin{minipage}[b]{0.45\linewidth}\centering
\includegraphics[scale=0.325]{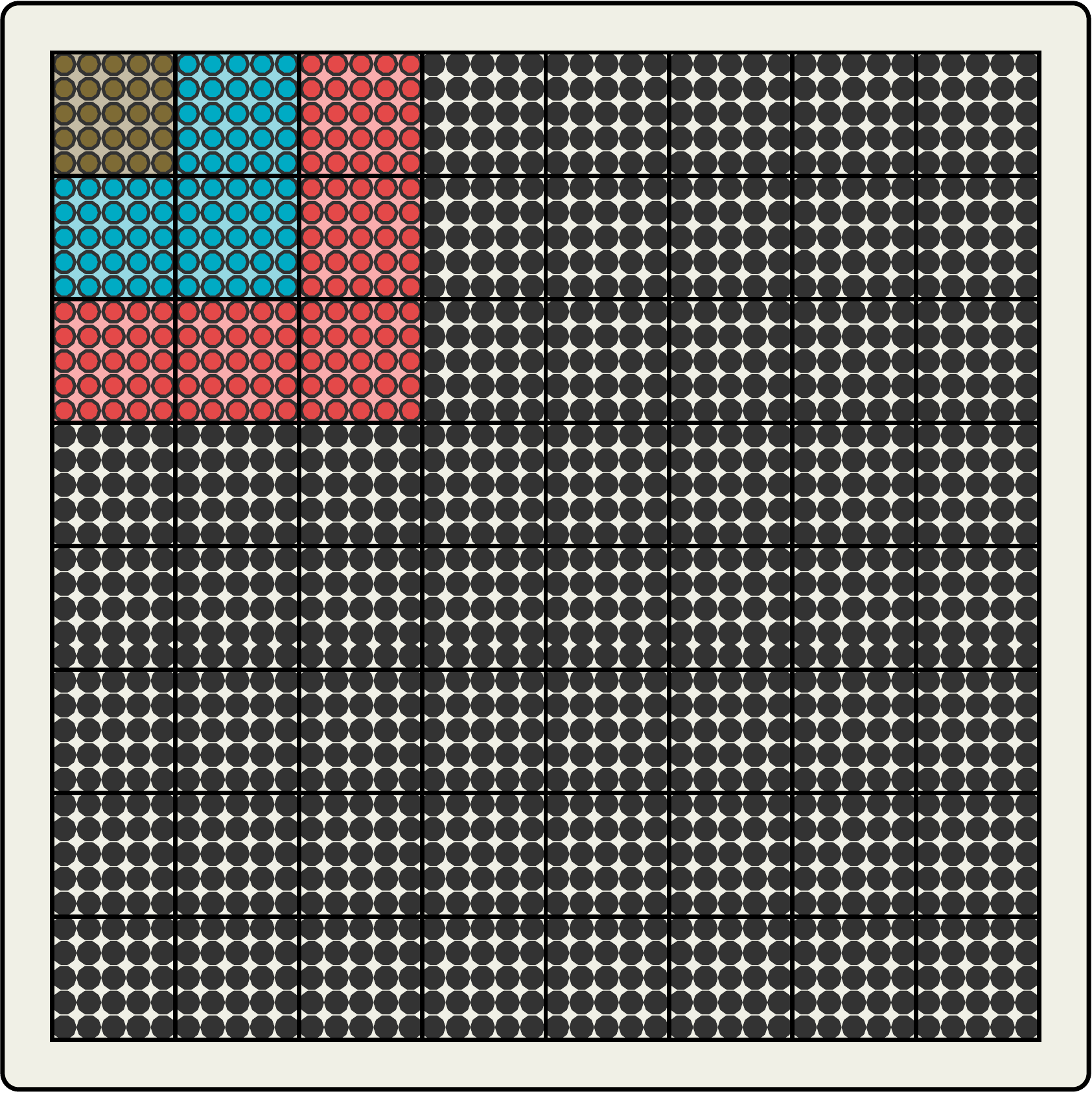}
\\\vspace*{1.5em}
\includegraphics[scale=0.325]{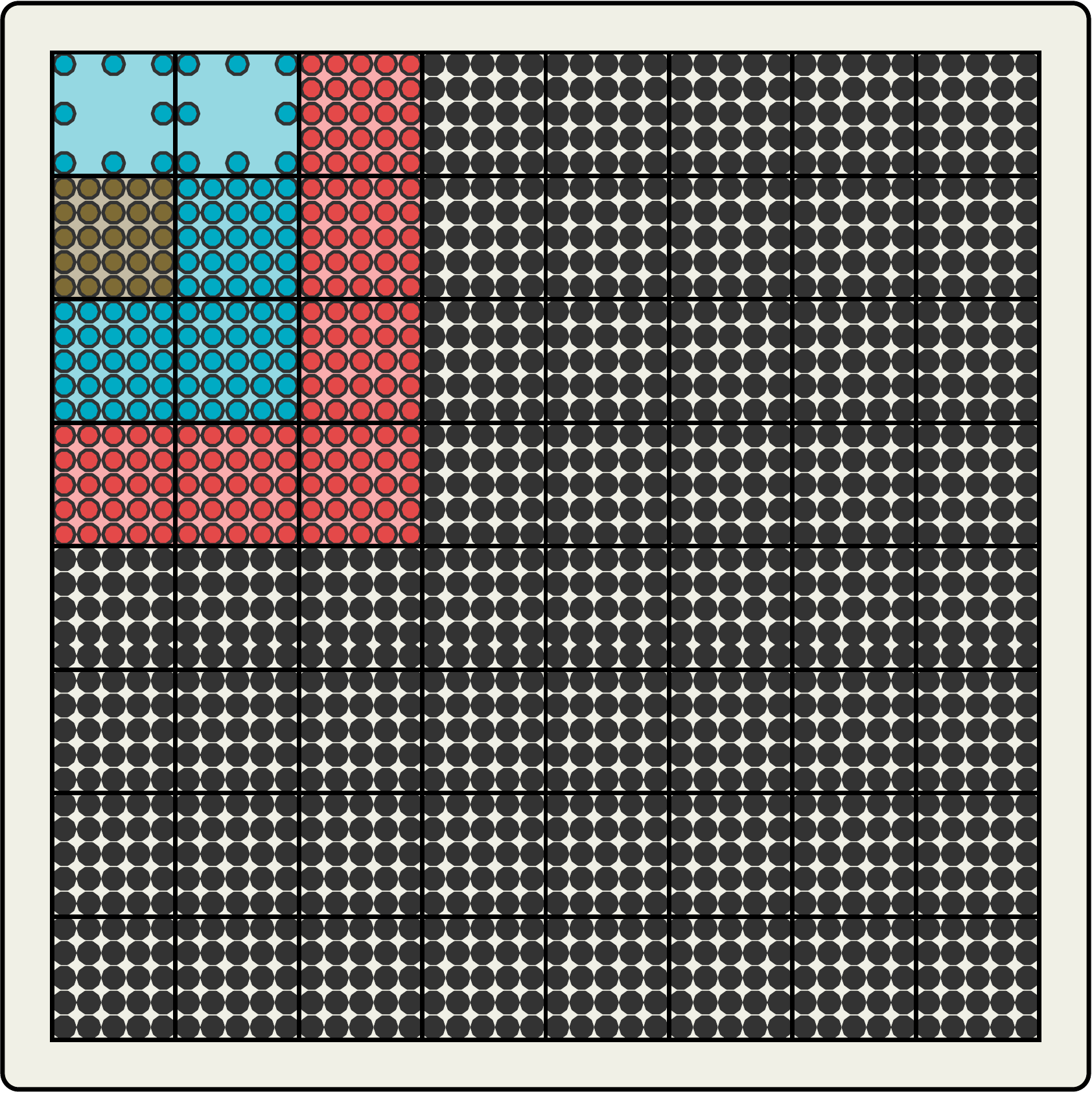}
\end{minipage}
\begin{minipage}[b]{0.45\linewidth}\centering
\includegraphics[scale=0.325]{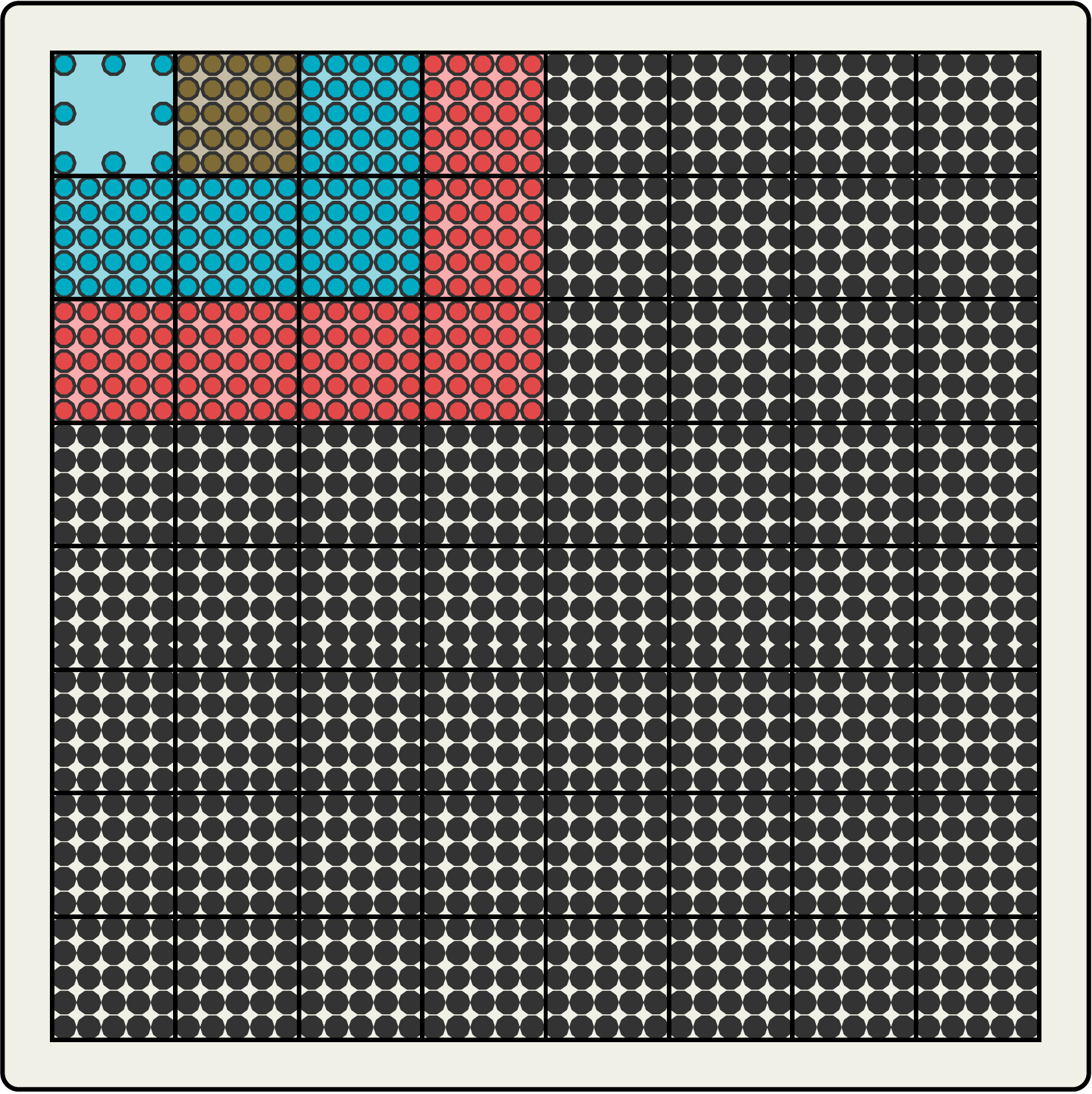}
\\\vspace*{1.5em}
\includegraphics[scale=0.325]{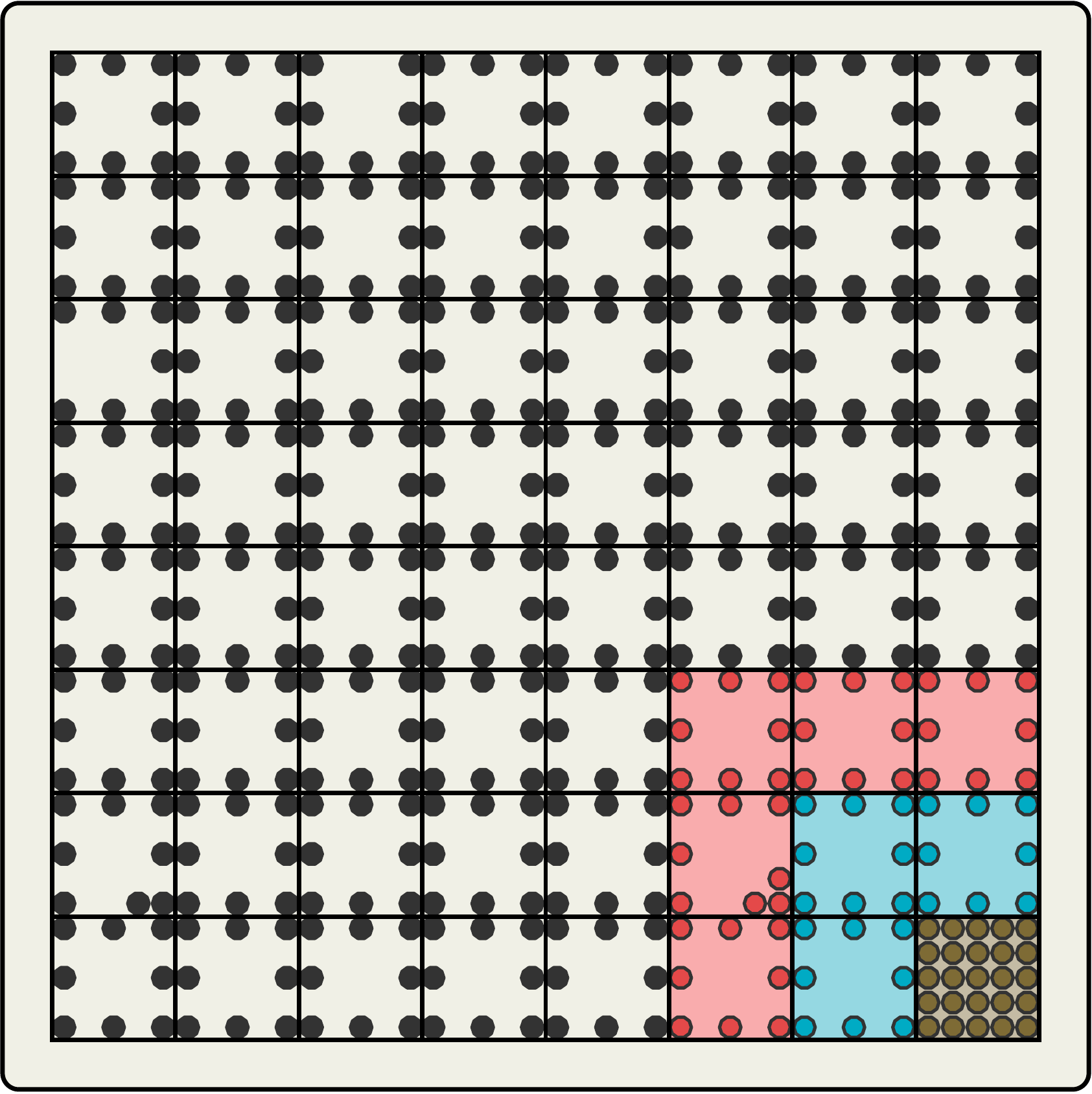}
\end{minipage}
\caption{Considering an integral equation discretized uniformly over the 2D domain $\Omega=[0,1]^2,$ we display the active DOFs at the point of skeletonization with respect to $\B_1$ (top-left), $\B_2$ (top-right), $\B_3$ (bottom-left), and $\B_{64}$ (bottom-right).  When skeletonizing with respect to $\B_i$ (brown DOFs) the near-field DOFs $\N_i$ are colored blue and the far-field DOFs \response{$\notprox_i$ corresponding to boxes inside the proxy surface are colored red.}  Note that the full set of far-field DOFs $\F_i$ includes not just the red DOFs but also all gray DOFs.\label{fig:boxesleaf}}
\end{figure}

\subsubsection{The general case: subsequent levels}
Having finished level $\ell=1$ of the tree, we step up to the next level of the spatial hierarchy, wherein boxes are twice as large in each axial direction.  Similar to our definitions of $\F_i$ and $\N_i$, for a level $\ell > 1$ we define the DOFs $\B_i$ of a box to be any \emph{active} DOFs geometrically contained in that box, that is, if
\begin{align*}
C_i \equiv \left\{j \,\vert \text{ box }j \text{ is a child box of box }i \right\},
\end{align*}
then $\B_i \equiv \bigcup_{j\in C_i} \sk_j$ (\ie, it is the union of the \emph{skeleton} DOFs of its child boxes).  With this definition in tow we state the following corollary of \cref{thm:farprox}.

\begin{corollary}\label{cor:far}
Suppose $r$ is the number of the last DOF set at level $\ell \ge 1$.  Then, at the beginning of level $\ell+1$, all far-field interactions between active DOFs at level $\ell+1$ are unmodified from their initial values in $K$.  That is, if
$\m{A} = \SF{\m{K}}{\B_1,\B_2,\dots,\B_{r}},$
then
\begin{align*}
\left\lmat\begin{array}{c}\idx{A}{\F_j}{\B_j} \\ \idx{A}{\B_j}{\F_j}^*\end{array}\right\rmat &= \left\lmat\begin{array}{c}\idx{K}{\F_j}{\B_j} \\ \idx{K}{\F_j}{\B_j}^*\end{array}\right\rmat
\end{align*}
for all $j>r$.
\end{corollary}

\Cref{cor:far} tells us that, while one might fear that Schur complement updates would propagate beyond interactions between $\B_j$ and $\notprox_j$ and thus require increasing the size of the proxy surface, the opposite is in fact true: at the beginning of a level, interactions with all of $\F_j$ are unmodified.  Therefore, our argument in \cref{sec:proxy} for the use of a proxy surface still holds and it is straightforward to loop over each box on level $\response{\ell= 2}$ and perform strong skeletonization with respect to the corresponding DOFs $\B_i$, as is visualized in \cref{fig:boxes}.  We repeat level-by-level for $\ell = 3,4,\dots,L-2$, noting that at level $\ell=L-1$ all boxes are adjacent and thus all sets of active far-field DOFs are empty so compression of this form is not possible.

\begin{figure}
\centering
\begin{minipage}[b]{0.45\linewidth}\centering
\includegraphics[scale=0.325]{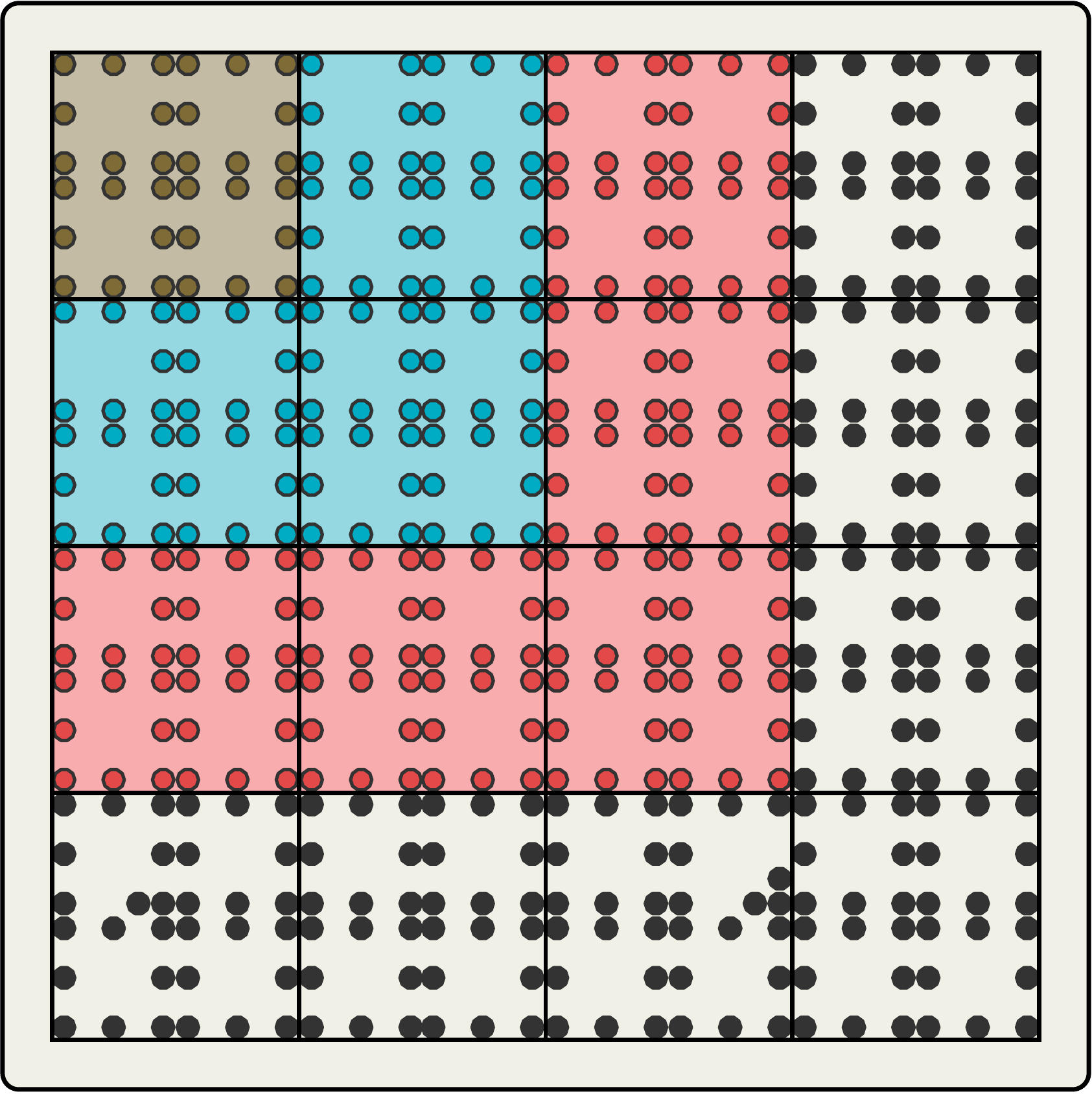}
\end{minipage}
\begin{minipage}[b]{0.45\linewidth}\centering
\includegraphics[scale=0.325]{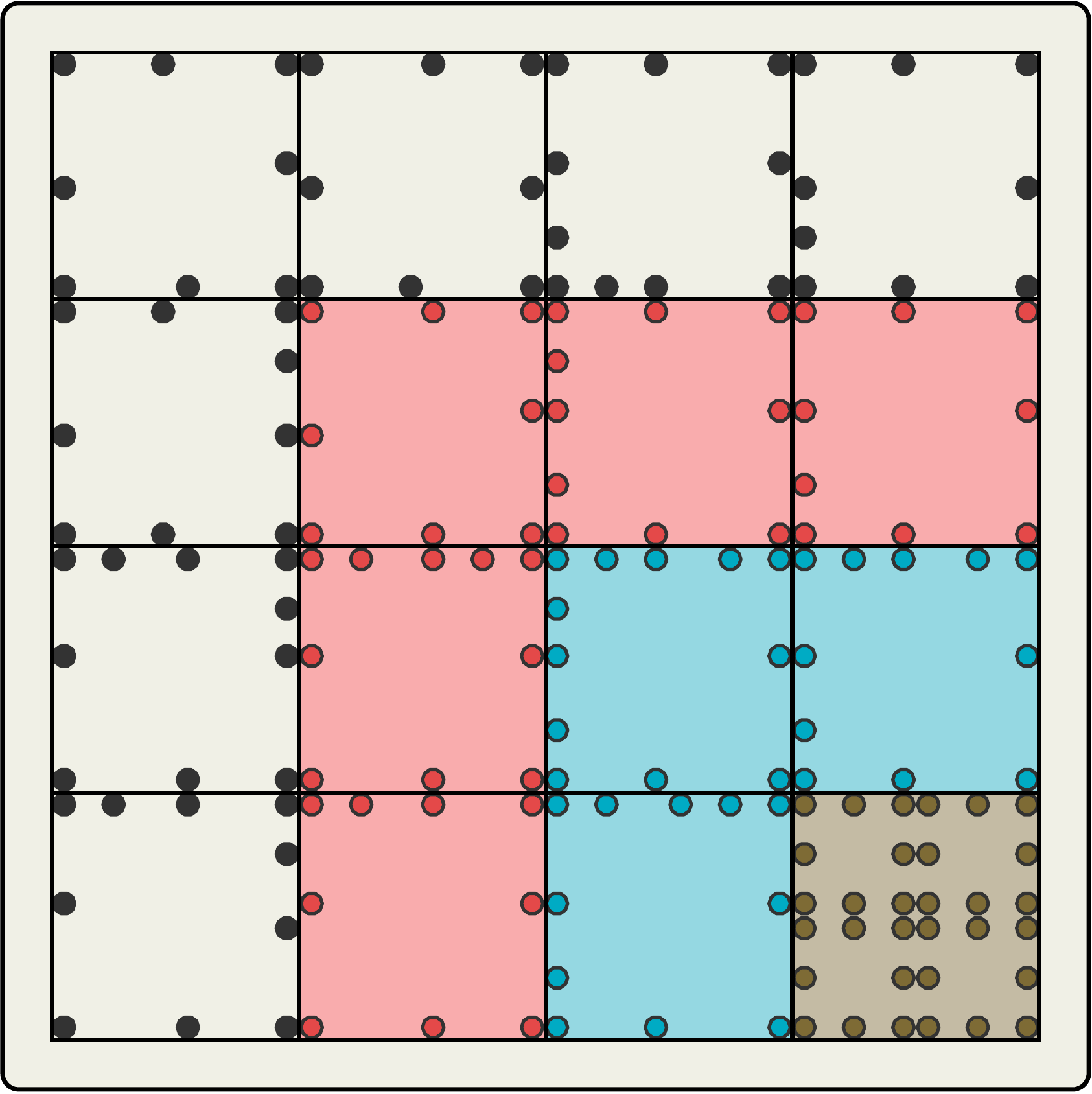}
\end{minipage}

\caption{After skeletonizing with respect to $\B_{64}$ at the end of \cref{fig:boxesleaf}, we skeletonize with respect to $\B_{65}$, the set of DOFs corresponding to the first larger box at the next level (left).  Proceeding to skeletonize with respect to each box through to $\B_{80}$, the last box at this level, we see that further compression has been attained as the number of remaining active DOFs has been reduced substantially (right).  All DOFs are colored as in \cref{fig:boxesleaf}.\label{fig:boxes}}
\end{figure}

\subsubsection{The final factorization}
Supposing that the last set of DOFs at level $\ell=L-2$ is $\B_{n}$, the matrix $\m{A} = \SF{\m{K}}{\B_1,\B_2,\dots,\B_n}$ has the same form as \eqref{eq:boxm}, albeit with more factors.  Defining $\sktop$ to be the set of all active DOFs remaining at this level of the tree,
 a last permutation to order the DOFs $\sktop$ contiguously yields the block-diagonal matrix

\begin{align*}
\m{D} &\equiv  \left\lmat \begin{array}{llll}
\idx{X}{\rd_1}{\rd_1}& & & \\
&\ddots&&\\
&&\idx{X}{\rd_{n}}{\rd_{n}}&\\
&&&\idx{A}{\sktop}{\sktop}
\end{array}\right\rmat\\
&\approx \m{P}_{t}^*\left(\prod_{i\in[n]'} \m{V}_i^{-1}\right)\m{K}\left(\prod_{i\in[n]}\m{W}_i^{-1}\right)\m{P}_{t}.
\end{align*}

Rearranging, we obtain an approximate factorization as
\begin{align}\label{eq:rss}
\m{F} \equiv \left(\prod_{i\in[n]} \m{V}_i\right) \m{P}_{t}\m{D}\m{P}_{t}^* \left(\prod_{i\in[n]'} \m{W}_i\right) \approx \m{K},
\end{align}
which we refer to as the \emph{strong recursive skeletonization factorization (RS-S)} of $\m{K}$.  The process of computing $\m{F}$ is summarized in \cref{alg:rskelf}.

Note that an approximate factorization of $\m{K}^{-1}$ may be obtained directly as
\begin{align}
\m{F}^{-1} = \left(\prod_{i\in[n]}\m{W}_i^{-1}\right) \m{P}_{t}\m{D}^{-1}\m{P}_{t}^* \left(\prod_{i\in[n]'} \m{V}_i^{-1}\right)\approx \m{K}^{-1},
\end{align}
though the approximation error will in general increase by a factor of the condition number of $\m{K}$.  Further, in the case where $\m{K}$ is positive definite, we have for all $i$ that $\m{V}_i = \m{W}_i^*$.  This means that, assuming our approximation is accurate enough that $\m{D}$ still admits a square-root $\m{D}^{1/2}$, we may obtain a generalized square-root $\m{F}^{1/2}$ with $\m{F} = \m{F}^{1/2}(\m{F}^{1/2})^*$ as
\begin{align*}
\m{F}^{1/2} &\equiv \left(\prod_{i\in[n]} \m{V}_i\right) \m{P}_{t}\m{D}^{1/2},
\end{align*}
which differs from any \response{generalized} square-root $\m{K}^{1/2}$ of $\m{K}$ (\eg, the Cholesky factor) by a unitary matrix in the ideal case ($\epsilon=0$).

In this setting, we may also compute an approximate log-determinant of $\m{K}$ using the fact that $\log|\m{D}|\approx\log|\m{K}|$, which is useful for applications in statistics \cite{mindengp}.

\begin{algorithm}
\caption{The strong recursive skeletonization factorization (RS-S)\label{alg:rskelf}}
\begin{algorithmic}[1]
\STATE{\texttt{// Initialize}}
\STATE{$\m{A} := \m{K}$}
\FOR{$\ell := 1$ \textbf{to} $L-2$}
  \FOR{{\bf each} \response{box $i\in\L_\ell$}}
    \STATE{\texttt{// Identify relevant DOFs for strong skeletonization}}
    \STATE{$[\B_i,\N_i,\F_i] := \{\text{active DOFs in box/near-field/far-field}\}$}
    \STATE{\texttt{// Perform strong skeletonization with respect to DOFs}}
    \STATE{$\m{A} := \SF{\m{A}}{\B_i} \approx \m{V}^{-1}_i\m{A}\m{W}_i^{-1}$}
  \ENDFOR
\ENDFOR
\STATE{\texttt{// Store middle block diagonal matrix and permutation}}
\STATE{$\m{D} := \m{P}_{t}^*\m{A}\m{P}_{t}$}

\STATE{\textbf{Output: }$\m{F}$ as in \eqref{eq:rss}}
\end{algorithmic}
\end{algorithm}

\subsubsection{Complexity}
\response{As written, \cref{alg:rskelf} applies to an arbitrary tree decomposition of space, \ie, some regions of space may be more refined than others in an adaptive fashion.  To compute meaningful complexity bounds, however, it is necessary to impose some structure on the tree. As is standard, we assume a tree} with $L=\O(\log N)$ levels in $d$ dimensions is given such that each leaf box contains at most a constant number of DOFs independent of $N$.  Letting $k_\ell$ denote the maximum of $|\sk_i|$ over all DOF sets corresponding to boxes on level $\ell$ and assuming $k_\ell \le k_{\ell+1}$ for all $\ell$, we obtain the following complexity result.
\begin{theorem}\label{thm:complexity}
Under the above assumptions and assuming further that a constant number of points is used to discretize the proxy surface $\Gamma$, we have that the cost $t_f$ of constructing the RS-S factorization $\m{F}$ according to \cref{alg:rskelf} and the cost $t_s$ of applying $\m{F}$ or $\m{F}^{-1}$ are given, respectively, as
 \begin{align*}
 t_f &= \O(N) + \sum_{\ell=1}^{L-2} \O(2^{d(L-\ell)}k_\ell^3),\quad  t_s = \O(N) + \sum_{\ell=1}^{L-2} \O(2^{d(L-\ell)}k_\ell^2).
 \end{align*}
 The memory requirement is trivially $m_f=\O(t_s)$.
\end{theorem}
\begin{proof}
Let $k_{0} = 1$ for convenience.  Note that, for a DOF set $\B_i$ corresponding to a box at level $\ell$, we have $|\B_i| = \O(k_{\ell-1})$, $|\N_i| = \O(k_{\ell-1})$, and $|\notprox_i| = \O(k_{\ell-1})$, since for leaf boxes the number of DOFs is bounded by a constant and for non-leaf boxes the DOFs are given by aggregating skeleton DOFs of child boxes at the previous level.

Because the proxy surface $\Gamma$ is discretized with a constant number of points, \response{the first matrix in \eqref{eq:smallid} used to compute an ID for the skeletonization with respect to $\B_i$ is of size $\O(|\notprox_i|)\times\O(|\B_i|)$}.  This implies that the cost of skeletonizing with respect to the DOFs $\B_i$ corresponding to a box at level $\ell$ is $\O(k_\ell^3)$ using the complexity result in \cref{sec:block_elimination}.

Finally, at each level $\ell$, there are at most $2^{d(L-\ell)}$ boxes, which gives the stated complexity for $t_f$ using the fact that $2^{dL}=\O(N)$.  The complexity for $t_s$ follows a similar argument, noting that all the block \response{unit-triangular} matrices can be trivially inverted.
\qquad
\end{proof}

For kernel functions $\response{K(z)}$ such as we consider here, \ie, relatively non-oscillatory Green's functions arising from elliptic PDEs, standard multipole-type estimates \cite{fastmultipole,fmm3d} can typically be used to show far-field interaction blocks $\idx{K}{\B_i}{\F_i}$ have ranks depending only weakly on $N$.  As previously mentioned in \cref{sec:leaf}, however, \cref{alg:rskelf} involves compressing also far-field interaction blocks that have received Schur complement updates to some of their entries from earlier steps of skeletonization.  For such entries, multipole estimates no longer directly apply, but ample numerical experimentation seem to
indicate similar rank behavior (see \cref{sec:results}) and thus it is common to assume that updated blocks of this nature still exhibit multipole-like rank behavior \cite{hifie,corona2013,ifmm2}.  Proceeding under this assumption, we obtain a more explicit complexity estimate.

\begin{corollary}\label{cor:scaling}
Suppose that for any fixed tolerance $\epsilon$ we have \response{$k_\ell = \O(\ell^q)$ in \cref{thm:complexity} for some $q>0$, \ie, the skeleton sets grow only as some power of the level index $\ell$ and $k_{L-2} = \O(\log^q N)$}. Then the RS-S factorization cost $t_f$, apply/solve cost $t_s$, and memory requirement $m_f$ scale as
 \begin{align*}
 t_f &= \O(N),\quad  t_s = \O(N), \quad m_f = \O(N),
 \end{align*}
 with constants depending on the tolerance $\epsilon$ and dimension $d$.
\end{corollary}

Note that the construction of the initial tree decomposition of space requires an additional upfront cost of $\O(N\log N)$, but in practice this cost is negligible compared to the factorization itself.

\subsection{Extension: hybrid skeletonization}
Algorithmically, the RS-S factorization has much in common with the RS factorization \cite{martinsson-rokhlin,hifie}.  The key distinction between the two is exactly what is meant by ``skeletonization''.  In \cref{sec:strongskel}, the strong skeletonization process we describe is used to compress far-field interactions (\eg, $\idx{K}{\B}{\F}$), leading to ranks essentially independent of $N$ under our assumptions.  In contrast, in the traditional skeletonization procedure both the far-field \emph{and} the near-field are compressed, \ie, blocks such as $\idx{K}{\B}{\comp{\B}}$ with $\comp{\B} = \N\cup\F$. As a consequence, the skeleton set grows with the rank of the near-field interactions, which typically goes as $\O(N^\frac{d-1}{d})$ at the top levels as has been illustrated in previous work \cite{rskel}.  After sparse elimination analogous to the strong case, we are left with
\begin{align*}
 \widetilde{\m{U}}_\m{T}\widetilde{\m{P}}^* \m{A}\widetilde{\m{P}}\widetilde{\m{L}}_\m{T} &=
 \left\lmat
 \begin{array}{ll|l}{}
 \idx{X}{\rd}{\rd} & & \\
&\idx{X}{\sk}{\sk}&\idx{A}{\sk}{\comp{\B}} \\\hline
&\idx{A}{\comp{\B}}{\sk}&\idx{A}{\comp{\B}}{\comp{\B}}
\end{array}
\right\rmat
\equiv \SFW{\m{A}}{\B}.
\end{align*}
Defining the notation
\begin{align}\label{eq:weakskel}
\widetilde{\m{V}}\equiv \widetilde{\m{P}}\widetilde{\m{U}}_\m{T}^{-1}, \quad \widetilde{\m{W}}\equiv\widetilde{\m{L}}^{-1}_\m{T}\widetilde{\m{P}}^*
\end{align}
analogously to \eqref{eq:skelop}, we obtain $\SFW{\m{A}}{\B}\approx \widetilde{\m{V}}^{-1}\m{A}\widetilde{\m{W}}^{-1}$.
We refer to this near-field compression and subsequent decoupling as \emph{weak} skeletonization to distinguish it from its strong counterpart.

While strong skeletonization typically leads to asymptotically more efficient factorizations \response{than} weak skeletonization due to the higher rank of near-field interactions compared to far-field interactions, it suffers from a higher storage cost.  This is because, in contrast to the weak case, strong skeletonization requires an additional step to explicitly decouple redundant DOFs from their near-field, and the corresponding block elimination operators must be stored.  To decrease the constant factor in the asymptotic storage cost of strong skeletonization, we can combine both weak and strong skeletonization in alternating fashion to obtain the \emph{hybrid \response{recursive} skeletonization factorization (RS-WS)} in \cref{alg:rskelfh}.

Using exactly the same tree decomposition as before, we begin by looping over each box at the bottom level $\ell=1$ and performing weak skeletonization with respect to the corresponding DOF sets $\widetilde{\B}_i$ for $i\in\L_1$, where here we use a tilde to explicitly mark that we are performing \emph{weak} skeletonization as in \eqref{eq:weakskel}. Assuming $|\L_1| = r$, this yields
\begin{align*}
\begin{split}
\SFW{\m{K}}{\widetilde{\B}_1,\widetilde{\B}_2,\dots,\widetilde{\B}_{r}} &\approx \left(\prod_{i\in\L_1'}\widetilde{\m{V}}_i^{-1}\right)\m{K}\left(\prod_{i\in\L_1^{\phantom{\circ}}}\widetilde{\m{W}}_i^{-1}\right).\end{split}
\end{align*}

Having now decoupled some number of DOFs via weak skeletonization without modifying any nonzero off-diagonal blocks, it is now possible to loop \emph{again} over $\L_1$, this time performing \emph{strong} skeletonization with respect to each set of active DOFs $\B_i$ on the level.  With $\m{A} = \SFW{\m{K}}{\widetilde{\B}_1,\widetilde{\B}_2,\dots,\widetilde{\B}_{r}}$, this gives
\begin{align*}
\SF{\m{A}}{\B_1,\B_2,\dots,\B_{r}} \approx \left(\prod_{i\in\L_1'}\m{V}_i^{-1}\right)\left(\prod_{i\in\L_1'}\widetilde{\m{V}}_i^{-1}\right)\m{K}\left(\prod_{i\in\L_1^{\phantom{\circ}}}\widetilde{\m{W}}_i^{-1}\right)\left(\prod_{i\in\L_1^{\phantom{\circ}}}\m{W}_i^{-1}\right).
\end{align*}
We repeat this process of weak skeletonization followed by strong skeletonization at each level.  A final step of permutation leads to the RS-WS factorization $\m{F}\approx \m{K}$ with
\begin{align}\label{eq:rws}
\m{F} \equiv \left[\prod_{\ell\in[L-2]}\left(\prod_{i\in\L_\ell} \widetilde{\m{V}}_i\right)\left(\prod_{i\in\L_\ell} \m{V}_i\right)\right] \m{P}_{t}\m{D}\m{P}_{t}^* \left[\prod_{\ell\in[L-2]'}\left(\prod_{i\in\L_\ell'} {\m{W}}_i\right)\left(\prod_{i\in\L_\ell'} \widetilde{\m{W}}_i\right)\right],
\end{align}
which is analogous to \eqref{eq:rss} but more cumbersome notationally due to the need to loop over the boxes on each level twice.  We remark that, as a modification to the above, it is possible to perform a final step of weak skeletonization at level $\ell=L-1$ even though subsequent strong skeletonization is not possible, and this is what we do in practice.

\begin{algorithm}
\caption{The hybrid recursive skeletonization factorization (RS-WS)\label{alg:rskelfh}}
\begin{algorithmic}[1]
\STATE{\texttt{// Initialize}}
\STATE{$\m{A} := \m{K}$}
\FOR{$\ell := 1$ \textbf{to} $L-2$}
  \FOR{{\bf each} box $i\in\L_\ell$}
    \STATE{\texttt{// Identify relevant DOFs for weak skeletonization}}
    \STATE{$\left[\widetilde{\B}_i,\comp{\widetilde{\B}}_i\right] := \{\text{active DOFs in box/complement}\}$}
    \STATE{\texttt{// Perform weak skeletonization with respect to DOFs}}
    \STATE{$\m{A} := \SFW{\m{A}}{\widetilde{\B_i}} \approx \widetilde{\m{V}}^{-1}_i\m{A}\widetilde{\m{W}}_i^{-1}$}
  \ENDFOR
    \FOR{{\bf each} box $i\in\L_\ell$}
      \STATE{\texttt{// Identify relevant DOFs for strong skeletonization}}
      \STATE{$[\B_i,\N_i,\F_i] := \{\text{active DOFs in box/near-field/far-field}\}$}
      \STATE{\texttt{// Perform strong skeletonization with respect to DOFs}}
      \STATE{$\m{A} := \SF{\m{A}}{\B_i} \approx \m{V}^{-1}_i\m{A}\m{W}_i^{-1}$}
    \ENDFOR
\ENDFOR
\STATE{\texttt{// Store middle block diagonal matrix and permutation}}
\STATE{$\m{D} := \m{P}_{t}^*\m{A}\m{P}_{t}$}

\STATE{\textbf{Output:} $\m{F}$ as in \eqref{eq:rws}}
\end{algorithmic}
\end{algorithm}

\section{Numerical results}
\label{sec:results}
To evaluate the performance of the \mbox{RS-S} and \mbox{RS-WS} factorizations, we implemented a number of examples in MATLAB\textsuperscript{\textregistered} on top of the FLAM library (\url{https://github.com/klho/FLAM/}).  \response{Our research code is available at \url{https://github.com/victorminden/strong-skel/}.}  We use adaptive quadtrees and octrees as appropriate, refining until the number of DOFs per leaf box is bounded by $n_\text{occ} =\O(1)$.  Diagonal blocks of $\m{D}$ were factored using the Cholesky decomposition for \response{positive definite} $\m{K}$ and the LU decomposition otherwise.

The primary quantities of interest for our examples (where applicable) are given in the following legend:
\begin{itemize}
\item $\epsilon$: tolerance parameter for all IDs;
\item $N$: total number of DOFs;
\item $t_f$: wall clock time to construct the factorization $\m{F}$, in seconds;
\item $t_s$: wall clock time to solve $\m{F}x=b$ for $x$, in seconds;
\item $m_f$: memory required to store $\m{F}$, in GB;
\item $e_a$: estimate of $\|\m{K}-\m{F}\|/\|\m{K}\|$;
\item $e_s$: estimate of $\|\m{I}-\m{K}\m{F}^{-1}\|\ge\|\m{K}^{-1}-\m{F}^{-1}\|/\|\m{K}^{-1}\|$;
\item $n_i$: number of iterations to solve \eqref{eq:linear} using CG to a tolerance of $10^{-12}$ on the relative residual norm, where the right-hand side $b$ is given up to scaling by $b=\m{K}x$ for $x$ a vector with normally-distributed entries.
\end{itemize}

We estimate the operator errors using the power method \cite{power2,power1
} to a tolerance of $10^{-2}$ in the relative error.  For this and for CG, the matrices $\m{K}$ and $\m{K}^*$ were applied via the fast Fourier transform or a kernel-independent fast multipole method as appropriate.

All computations were performed in MATLAB\textsuperscript{\textregistered} R2015b on a 64-bit Linux server with Intel\textsuperscript{\textregistered} Xeon\textsuperscript{\textregistered} E7-8890 v3 CPUs at 2.50 GHz using up to 72 cores through BLAS multithreading, where the number of cores at any time \response{was} chosen adaptively by MATLAB\textsuperscript{\textregistered}.

\subsection{Example 1: unit square in 2D}
We begin with a simple 2D example on the unit square $\Omega=[0,1]^2$.  Taking $a(x)\equiv 0$, $b(x)\equiv c(y) \equiv 1$, and $\response{K(z)\equiv -\frac{1}{2\pi}\log(\|z\|)}$ in \eqref{eq:bie}, we discretize the resulting first-kind volume integral equation using piecewise-constant collocation on a uniform $\sqrt{N}$ by $\sqrt{N}$ grid such that $\m{K}_{ij} = \frac{1}{N}K(x_i-x_j)$ for $i\ne j$.  The diagonal entries $\m{K}_{ii}$ are approximated adaptively using the \texttt{dblquad} function in MATLAB\textsuperscript{\textregistered} for simplicity such that
\begin{align*}
\m{K}_{ii} \approx \int_{-h/2}^{h/2}\int_{-h/2}^{h/2} K(x-y)\,dx\,dy,
\end{align*}
where $h\equiv 1/\sqrt{N}$.  Note that this is essentially a Nystr\"om method, but viewing it as piecewise-constant collocation makes sense of the modified diagonal.  The order of accuracy of this quadrature is not high compared to other more accurate quadratures based on the idea of local corrections near the singularity, but its simplicity makes it a good candidate for illustrating our approach.

We compare four different skeletonization-based approaches to factorizing $\m{K}$: the (weak) recursive skeletonization factorization (RS) \cite{hifie,martinsson-rokhlin}, the hierarchical interpolative factorization (HIF) \cite{hifie}, and the strong and hybrid recursive skeletonization factorizations introduced in \cref{sec:rs} (RS-S and RS-WS).  Since it is based strictly on near-field compression, we expect RS to exhibit fundamentally different asymptotic scaling, whereas the other three methods should all exhibit essentially linear scaling under our rank assumptions.  All cases were run across a range of $N$ for tolerances $\epsilon=10^{-6}$ and $\epsilon=10^{-9}$, with results visualized in \cref{fig:2dresults} and corresponding data given in \cref{tab:2Dsquare,tab:2Dsquareapply}.  For HIF and RS, we used an occupancy parameter $n_\text{occ}=64$, whereas for RS-S and RS-WS, we used $n_\text{occ}=256$, hand-tuned in each case for optimal performance. For all methods we used $n_p=64$ proxy points to discretize the proxy surface.

In this 2D example, we see that all methods remain relatively competitive in terms of factorization time $t_f$ for both tolerances.  Looking at the plots of solve time $t_s$, we see that, while all methods seem to scale as $\O(N\log N)$ or $\O(N)$, for $\epsilon=10^{-9}$, RS-S is significantly ($\sim4$X to $10$X) slower than the other methods.  Since $t_s$ decreases drastically when using RS-WS (where we add additional levels of weak skeletonization) we hypothesize that the jump in solve time for RS-S is due to caching and the increased size of the subblocks comprising factors of $\m{F}$ for RS-S as compared to the other methods.  This belief is reinforced by the scaling plots for the memory $m_f$, in which we see that, in 2D, memory usage for RS-S tends to be the highest, followed by RS, RS-WS, and HIF.  We note as well the sizable difference in $m_f$ between RS-S and RS-WS, which shows that hybrid skeletonization is effective at reducing memory usage in this setting.

\begin{figure}
\centering
\includegraphics[scale=0.325]{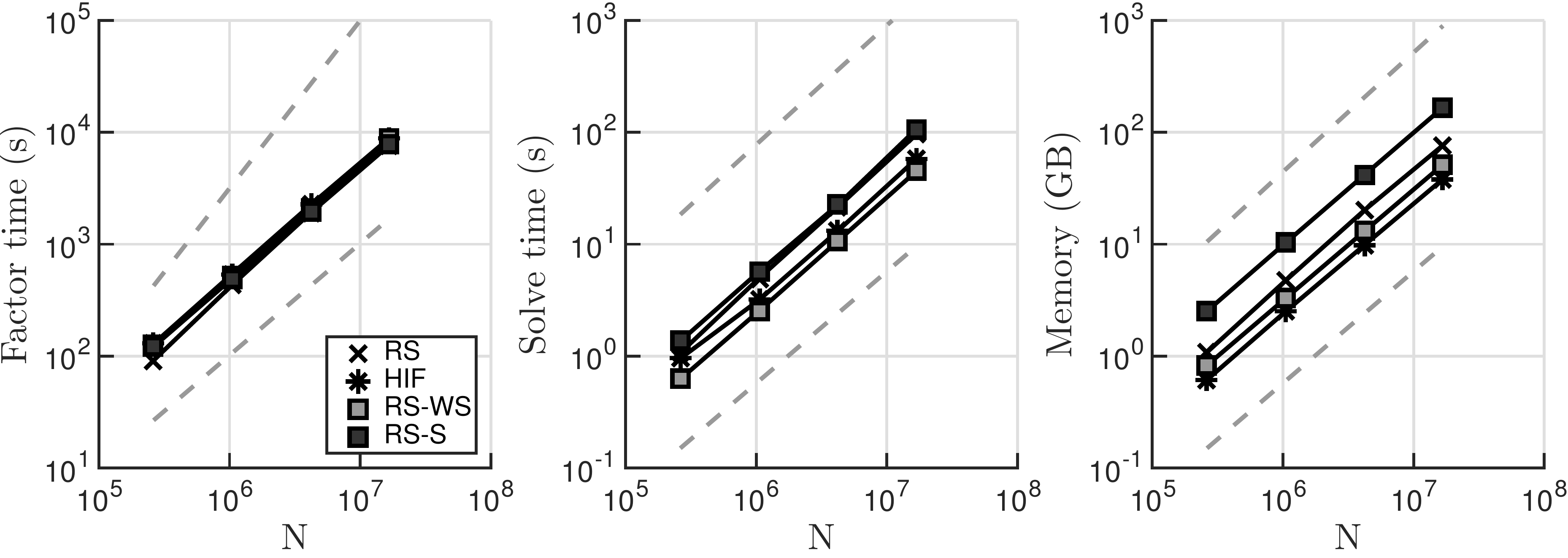}

\vspace{0.5cm}

\includegraphics[scale=0.325]{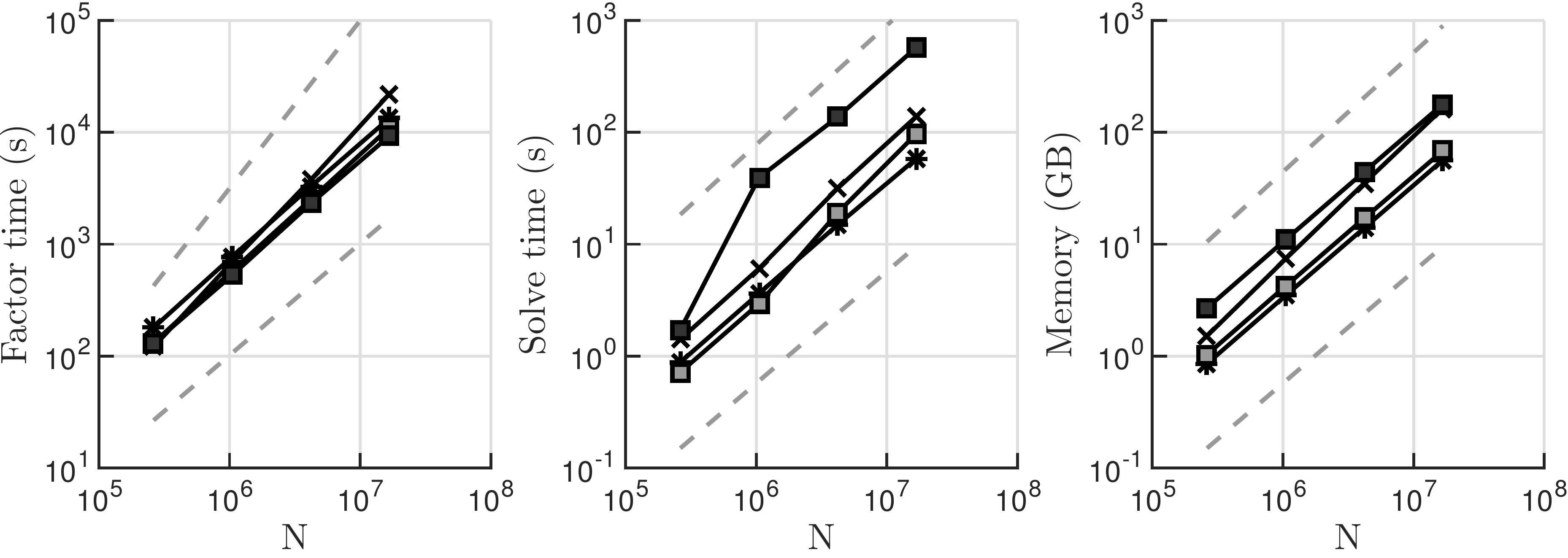}

\caption{Wall clock factor times $t_f$ and solve times $t_s$ and memory usage $m_f$ from Example 1 are shown for $\epsilon=10^{-6}$ (top row) and $\epsilon=10^{-9}$ (bottom row).  Each plot follows the top-left legend, with additional reference scaling curves $\O(N^{3/2})$ and $\O(N)$ (left subplots),  and $\O(N\log N)$ and $\O(N)$ (center and right subplots).  Corresponding data are given in \cref{tab:2Dsquare}.\label{fig:2dresults}}
\end{figure}

\begin{sidewaystable}[ph!]
\small
\ra{1.0}
\centering

\caption{Timing and memory results for Example 1  \label{tab:2Dsquare}}
\begin{tabular}{@{}@{\extracolsep{-1pt}}cc|ccc|ccc|ccc|ccc@{}}\toprule
&\multicolumn{1}{c}{}&\multicolumn{3}{c}{RS} & \multicolumn{3}{c}{HIF} & \multicolumn{3}{c}{RS-S} & \multicolumn{3}{c}{RS-WS}\\ \cmidrule(lr){3-5} \cmidrule(lr){6-8} \cmidrule(lr){9-11} \cmidrule(lr){12-14}
$ \epsilon$ & $N$ &$t_f$ &$t_s$ &$m_f$ & $t_f$& $t_s$ & $m_f$ &$t_f$ & $t_s$ &$m_f$&$t_f$ & $t_s$ & $m_f$ \\\midrule
   \multirow{ 4}{*}{$10^{-6}$}
 &$512^2$   & $\scinote{9.1}{+}{1}$ & $\scinote{1.1}{+}{0}$ & $\scinote{1.1}{+}{0}$
 & $\scinote{1.3}{+}{2}$ &  $\scinote{9.5}{-}{1}$ & $\scinote{6.1}{-}{1}$
 &  $\scinote{1.2}{+}{2}$ & $\scinote{1.4}{+}{0}$ & $\scinote{2.5}{+}{0}$
 & $\scinote{1.3}{+}{2}$ & $\scinote{6.2}{-}{1}$ & $\scinote{8.3}{-}{1}$

\\

 &$1024^2$  & $\scinote{4.3}{+}{2}$ & $\scinote{4.9}{+}{0}$ & $\scinote{4.7}{+}{0}$
 & $\scinote{5.4}{+}{2}$ &  $\scinote{3.2}{+}{0}$ & $\scinote{2.5}{+}{0}$
 &  $\scinote{4.9}{+}{2}$ & $\scinote{5.7}{+}{0}$ & $\scinote{1.0}{+}{1}$
 & $\scinote{5.1}{+}{2}$ & $\scinote{2.5}{+}{0}$ & $\scinote{3.3}{+}{0}$\\

 &$2048^2$  & $\scinote{1.9}{+}{3}$ & $\scinote{2.1}{+}{1}$ & $\scinote{2.0}{+}{1}$
 & $\scinote{2.3}{+}{3}$ &  $\scinote{1.3}{+}{1}$ & $\scinote{9.8}{+}{0}$
 &  $\scinote{1.9}{+}{3}$ & $\scinote{2.2}{+}{1}$ & $\scinote{4.2}{+}{1}$
 & $\scinote{2.1}{+}{3}$ & $\scinote{1.1}{+}{1}$ & $\scinote{1.3}{+}{1}$\\

  &$4096^2$  & $\scinote{7.7}{+}{3}$ & $\scinote{9.6}{+}{1}$ & $\scinote{7.7}{+}{1}$
 & $\scinote{8.9}{+}{3}$ &  $\scinote{5.8}{+}{1}$ & $\scinote{3.8}{+}{1}$
 &  $\scinote{7.8}{+}{3}$ & $\scinote{1.1}{+}{2}$ & $\scinote{1.7}{+}{2}$
 & $\scinote{8.8}{+}{3}$ & $\scinote{4.5}{+}{1}$ & $\scinote{5.2}{+}{1}$
\\\midrule    \multirow{4}{*}{$10^{-9}$}
 &$512^2$  &  $\scinote{1.2}{+}{2}$ & $\scinote{1.4}{+}{0}$ & $\scinote{1.5}{+}{0}$
 & $\scinote{1.8}{+}{2}$ &  $\scinote{8.8}{-}{1}$ & $\scinote{8.5}{-}{1}$
 &  $\scinote{1.3}{+}{2}$ & $\scinote{1.7}{+}{0}$ & $\scinote{2.7}{+}{0}$
 & $\scinote{1.3}{+}{2}$ & $\scinote{7.2}{-}{1}$ & $\scinote{1.0}{+}{0}$\\

 &$1024^2$  & $\scinote{7.0}{+}{2}$ & $\scinote{6.1}{+}{0}$ & $\scinote{7.4}{+}{0}$
 & $\scinote{7.8}{+}{2}$ &  $\scinote{3.6}{+}{0}$ & $\scinote{3.5}{+}{0}$
 &  $\scinote{5.3}{+}{2}$ & $\scinote{3.9}{+}{1}$ & $\scinote{1.1}{+}{1}$
 & $\scinote{5.8}{+}{2}$ & $\scinote{2.9}{+}{0}$ & $\scinote{4.2}{+}{0}$
\\
 &$2048^2$ & $\scinote{3.8}{+}{3}$ & $\scinote{3.2}{+}{1}$ & $\scinote{3.5}{+}{1}$
 & $\scinote{3.3}{+}{3}$ &  $\scinote{1.5}{+}{1}$ & $\scinote{1.4}{+}{1}$
 &  $\scinote{2.3}{+}{3}$ & $\scinote{1.4}{+}{2}$ & $\scinote{4.4}{+}{1}$
 & $\scinote{2.5}{+}{3}$ & $\scinote{1.9}{+}{1}$ & $\scinote{1.7}{+}{1}$

\\
  &$4096^2$ & $\scinote{2.2}{+}{4}$ & $\scinote{1.4}{+}{2}$ & $\scinote{1.6}{+}{2}$
 & $\scinote{1.3}{+}{4}$ &  $\scinote{5.8}{+}{1}$ & $\scinote{5.7}{+}{1}$
 &  $\scinote{9.3}{+}{3}$ & $\scinote{5.7}{+}{2}$ & $\scinote{1.8}{+}{2}$
 & $\scinote{1.1}{+}{4}$ & $\scinote{9.7}{+}{1}$ & $\scinote{6.9}{+}{1}$

\\
\bottomrule\\
\end{tabular}

\bigskip\bigskip

\caption{Accuracy results for Example 1 \label{tab:2Dsquareapply}}
\begin{tabular}{@{}@{\extracolsep{-1pt}}cc|ccc|ccc|ccc|ccc@{}}\toprule
&\multicolumn{1}{c}{}&\multicolumn{3}{c}{RS}& \multicolumn{3}{c}{HIF}&\multicolumn{3}{c}{RS-S} & \multicolumn{3}{c}{RS-WS}\\ \cmidrule(lr){3-5} \cmidrule(lr){6-8} \cmidrule(lr){9-11} \cmidrule(lr){12-14}
$ \epsilon$ & $N$ & $e_a$ &$e_s$  &$n_i$& $e_a$ &$e_s$ & $n_i$& $e_a$ &$e_s$ &  $n_i$& $e_a$ &$e_s$& $n_i$ \\\midrule
   \multirow{ 3}{*}{$10^{-6}$}
 &$512^2$   & $\scinote{2.5}{-}{07}$ & $\scinote{3.2}{-}{04}$ & 3
 & $\scinote{2.9}{-}{07}$ & $\scinote{5.8}{-}{04}$ &  3
 & $\scinote{4.0}{-}{08}$ &  $\scinote{4.0}{-}{04}$ & 3
 & $\scinote{3.0}{-}{07}$ &$\scinote{7.0}{-}{04}$ & 3

\\
 &$1024^2$   & $\scinote{7.3}{-}{07}$ & $\scinote{3.6}{-}{04}$ & 3
 & $\scinote{3.8}{-}{07}$ & $\scinote{5.0}{-}{04}$ &  3
 & $\scinote{4.2}{-}{08}$ &  $\scinote{6.0}{-}{04}$ & 3
 & $\scinote{3.0}{-}{07}$ &$\scinote{1.6}{-}{03}$ & 3

\\
 &$2048^2$   & $\scinote{1.2}{-}{06}$ & $\scinote{5.2}{-}{04}$ & 3
 & $\scinote{5.1}{-}{07}$ & $\scinote{9.7}{-}{04}$ &  3
 & $\scinote{4.4}{-}{08}$ &  $\scinote{7.9}{-}{04}$ & 3
 & $\scinote{3.0}{-}{07}$ &$\scinote{1.9}{-}{03}$ & 3
 \\
 & $4096^2$& $\scinote{1.3}{-}{06}$ & $\scinote{9.2}{-}{04}$ & 3
 & $\scinote{6.0}{-}{07}$ & $\scinote{1.0}{-}{03}$ &  3
 & $\scinote{5.0}{-}{08}$ &  $\scinote{1.6}{-}{03}$ & 3
 & $\scinote{3.1}{-}{07}$ &$\scinote{2.0}{-}{03}$ & 3   \\\midrule
   \multirow{ 3}{*}{$10^{-9}$}&
 $512^2$   &
$\scinote{1.7}{-}{10}$ & $\scinote{3.1}{-}{07}$ & 2
 & $\scinote{3.6}{-}{10}$ & $\scinote{5.3}{-}{07}$ &  2
 & $\scinote{2.7}{-}{11}$ &  $\scinote{3.3}{-}{07}$ & 2
 & $\scinote{1.5}{-}{10}$ &$\scinote{1.1}{-}{06}$ & 2

\\
 &$1024^2$   & $\scinote{6.8}{-}{10}$ & $\scinote{4.5}{-}{07}$ & 2
 & $\scinote{1.9}{-}{10}$ & $\scinote{5.6}{-}{07}$ &  2
 & $\scinote{2.7}{-}{11}$ &  $\scinote{5.5}{-}{07}$ & 2
 & $\scinote{1.9}{-}{10}$ &$\scinote{1.7}{-}{06}$ & 2

\\
 &$2048^2$   & $\scinote{7.6}{-}{10}$ & $\scinote{5.6}{-}{07}$ & 2
 & $\scinote{8.5}{-}{10}$ & $\scinote{9.7}{-}{07}$ &  2
 & $\scinote{2.9}{-}{11}$ &  $\scinote{7.8}{-}{07}$ & 2
 & $\scinote{3.0}{-}{10}$ &$\scinote{2.5}{-}{06}$ & 2
 \\
 & $4096^2$& $\scinote{7.9}{-}{10}$ & $\scinote{1.5}{-}{06}$ & 2
 & $\scinote{9.5}{-}{10}$ & $\scinote{9.9}{-}{07}$ &  2
 & $\scinote{3.0}{-}{11}$ &  $\scinote{1.5}{-}{06}$ & 2
 & $\scinote{2.2}{-}{10}$ &$\scinote{3.5}{-}{06}$ & 2   \\
\bottomrule\\
\end{tabular}

\end{sidewaystable}

Looking at \cref{tab:2Dsquareapply}, we see that the forward error $e_a$ of the approximate operator is roughly $\O(\epsilon)$.  This seems to indicate that the relative operator-norm error of the factorization is well-controlled by the relative error in the IDs of off-diagonal blocks, which is difficult to prove for factorizations based on skeletonization.  The bound $e_s$ on the error for the inverse operator exhibits similar behavior, though we lose accuracy due to ill-conditioning of $\m{K}$.  Finally, while at these accuracy levels we may \response{use} $\m{F}$ as a moderate-accuracy direct solver, these factorizations  perform exceedingly well when used as preconditioners for CG, as exhibited by the small number of iterations $n_i$ required to attain a relative residual norm of $10^{-12}$.  Note that the unpreconditioned method fails to converge within 100 iterations.

\subsection{Example 2: unit cube in 3D}
We turn now to the 3D analogue of Example 1, a first-kind volume integral equation on the unit cube $\Omega=[0,1]^3$ with $a(x)\equiv 0$, $b(x)\equiv c(y) \equiv 1$ and $\response{K(z)\equiv \frac{1}{4\pi \|z\|}}$ in \eqref{eq:bie}.  As before, we use piecewise-constant collocation on a regular grid  with adaptive quadrature for the diagonal entries.  In the interest of constructing efficient preconditioners or low-precision direct solvers, we consider the tolerance levels  $\epsilon=10^{-3}$ and $\epsilon=10^{-6}$ with results visualized in \cref{fig:3dresults} and corresponding data given in \cref{tab:3Dcube,tab:3Dcubeapply}.  For $\epsilon=10^{-3}$ we used the occupancy parameter $n_\text{occ}=64$ for all octrees, whereas for $\epsilon=10^{-6}$ we used $n_\text{occ}=512$ for RS-S and RS-WS and $n_\text{occ}=64$ for RS and HIF.  To discretize the proxy surface, we choose $n_p=512$ points randomly distributed on the sphere.

\begin{figure}
\centering
\includegraphics[scale=0.325]{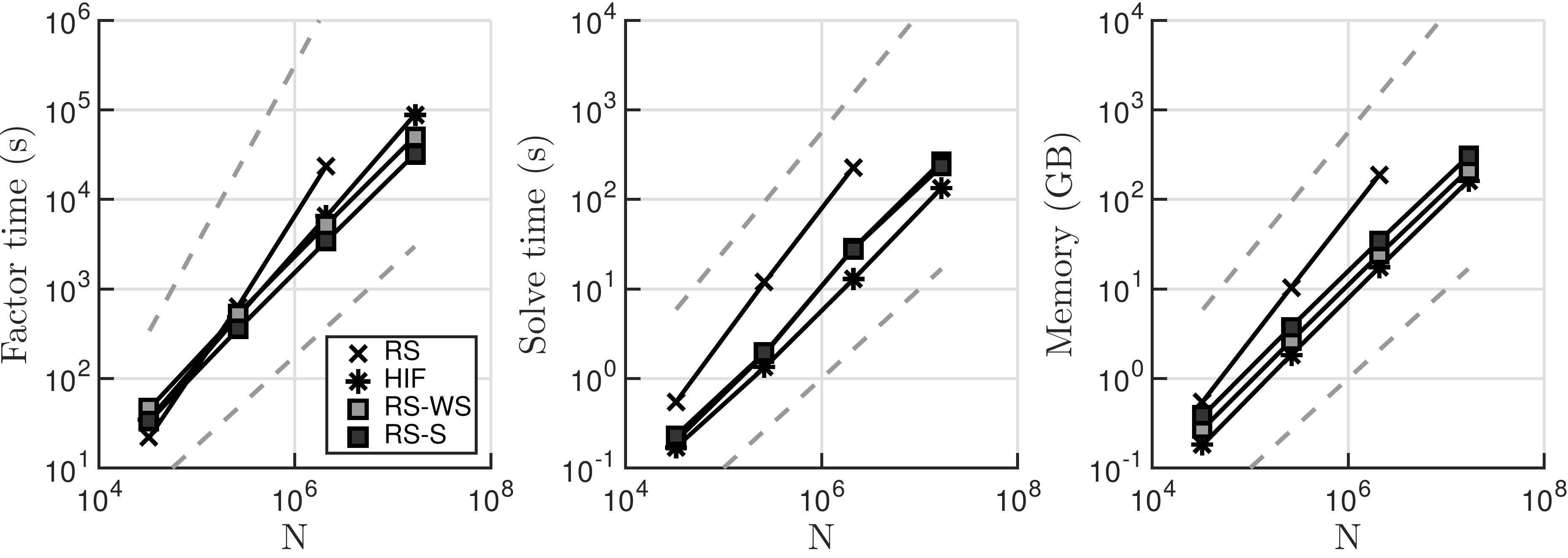}

\vspace{0.5cm}

\includegraphics[scale=0.325]{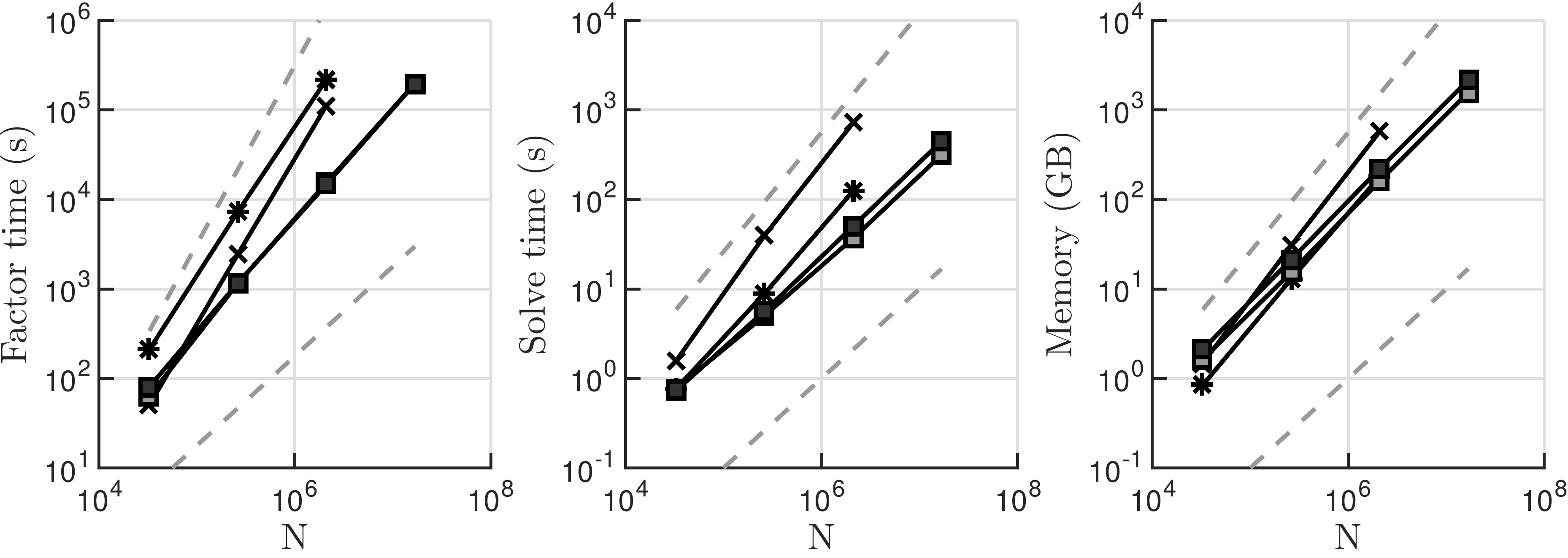}

\caption{Wall clock factor times $t_f$ and solve times $t_s$ and memory usage $m_f$ from Example 2 are shown for $\epsilon=10^{-3}$ (top row) and $\epsilon=10^{-6}$ (bottom row).  Each plot follows the top-left legend, with additional reference scaling curves $\O(N^{2})$ and $\O(N)$ (left subplots),  and $\O(N^{4/3})$ and $\O(N)$ (center and right subplots).  Corresponding data are given in \cref{tab:3Dcube}.  Note that in several cases the curves for RS-S and RS-WS lie nearly on top of each other.
\label{fig:3dresults}}
\end{figure}

Contrary to the 2D case, we immediately observe the difference in scaling between RS and the other methods considered for each of $t_f$, $t_s$, and $m_f$.  Further, for $\epsilon=10^{-3}$ we see that HIF, RS-S, and RS-WS all scale approximately like $\O(N)$ with a clear trade-off between factorization time and memory usage --- RS-S gives the smallest $t_f$ but the largest $m_f$ of the three, the opposite is true for HIF, and RS-WS is somewhere between RS and HIF.  \response{In particular, note that the memory usage for RS-WS is less than for RS-S, as desired.}

For $\epsilon=10^{-6}$, we see that $t_f$ for RS-S and RS-WS is markedly less than for HIF.  We believe that \response{some of} this overhead is due to auxiliary data structures and functions associated with the more complicated geometry exploited in HIF, and could therefore potentially be reduced via more sophisticated implementations.  For all three, however, we see scaling of $t_f$ that appears better than $\O(N^{2})$ but not quite $\O(N)$ as would be predicted under our assumptions.  One possibility is \emph{boundary effects}: essentially, we are perhaps not yet in the asymptotic regime because
\begin{enumerate}
\item[(a)] boxes at the boundary of $\Omega$ have smaller sets of near- and far-field DOFs than do interior boxes, and are therefore cheaper to skeletonize, and,
\item[(b)] the number of interior boxes is relatively small in 3D for small $N$.
\end{enumerate}
Another possibility is simply that our assumption on rank behavior is incorrect at $\epsilon=10^{-6}$ for this example.  Unfortunately, distinguishing between these cases requires larger tests than are currently feasible due to memory constraints.

\begin{sidewaystable}[ph!]
\small
\ra{1.0}
\centering

\caption{Timing and memory results for Example 2 \label{tab:3Dcube}}
\begin{tabular}{@{}@{\extracolsep{-1pt}}cc|ccc|ccc|ccc|ccc@{}}\toprule
&\multicolumn{1}{c}{}&\multicolumn{3}{c}{RS} & \multicolumn{3}{c}{HIF} & \multicolumn{3}{c}{RS-S} & \multicolumn{3}{c}{RS-WS}\\ \cmidrule(lr){3-5} \cmidrule(lr){6-8} \cmidrule(lr){9-11} \cmidrule(lr){12-14}
$ \epsilon$ & $N$ &$t_f$ &$t_s$ &$m_f$ & $t_f$& $t_s$ & $m_f$ &$t_f$ & $t_s$ &$m_f$&$t_f$ & $t_s$ & $m_f$ \\\midrule
   \multirow{ 4}{*}{$10^{-3}$}
 &$32^3$   &$\scinote{2.2}{+}{1}$ & $\scinote{5.5}{-}{1}$ & $\scinote{5.5}{-}{1}$
 & $\scinote{3.5}{+}{1}$ &  $\scinote{1.7}{-}{1}$ & $\scinote{1.8}{-}{1}$
 &  $\scinote{3.3}{+}{1}$ & $\scinote{2.3}{-}{1}$ & $\scinote{3.8}{-}{1}$
 & $\scinote{4.6}{+}{1}$ & $\scinote{2.0}{-}{1}$ & $\scinote{2.7}{-}{1}$

 \\

 &$64^3$  & $\scinote{6.4}{+}{2}$ & $\scinote{1.2}{+}{1}$ & $\scinote{1.0}{+}{1}$
 & $\scinote{4.9}{+}{2}$ &  $\scinote{1.3}{+}{0}$ & $\scinote{1.8}{+}{0}$
 &  $\scinote{3.5}{+}{2}$ & $\scinote{1.9}{+}{0}$ & $\scinote{3.8}{+}{0}$
 & $\scinote{5.2}{+}{2}$ & $\scinote{1.9}{+}{0}$ & $\scinote{2.6}{+}{0}$
\\

 &$128^3$  & $\scinote{2.4}{+}{4}$ & $\scinote{2.3}{+}{2}$ & $\scinote{1.9}{+}{2}$
 & $\scinote{6.4}{+}{3}$ &  $\scinote{1.3}{+}{1}$ & $\scinote{1.8}{+}{1}$
 &  $\scinote{3.4}{+}{3}$ & $\scinote{2.8}{+}{1}$ & $\scinote{3.5}{+}{1}$
 & $\scinote{5.2}{+}{3}$ & $\scinote{2.8}{+}{1}$ & $\scinote{2.5}{+}{1}$\\

  &$256^3$  & -- & -- & --
 & $\scinote{8.9}{+}{4}$ &  $\scinote{1.3}{+}{2}$ & $\scinote{1.6}{+}{2}$
 &  $\scinote{3.2}{+}{4}$ & $\scinote{2.4}{+}{2}$ & $\scinote{3.0}{+}{2}$
 & $\scinote{5.0}{+}{4}$ & $\scinote{2.7}{+}{2}$ & $\scinote{2.1}{+}{2}$
\\\midrule
   \multirow{ 4}{*}{$10^{-6}$}
 &$32^3$  & $\scinote{5.1}{+}{1}$ & $\scinote{1.6}{+}{0}$ & $\scinote{1.5}{+}{0}$
 & $\scinote{2.1}{+}{2}$ &  $\scinote{7.7}{-}{1}$ & $\scinote{8.5}{-}{1}$
 &  $\scinote{8.1}{+}{1}$ & $\scinote{7.3}{-}{1}$ & $\scinote{2.1}{+}{0}$
 & $\scinote{6.3}{+}{1}$ & $\scinote{7.7}{-}{1}$ & $\scinote{1.6}{+}{0}$\\

 &$64^3$  & $\scinote{2.4}{+}{3}$ & $\scinote{4.0}{+}{1}$ & $\scinote{3.1}{+}{1}$
 & $\scinote{7.3}{+}{3}$ &  $\scinote{8.9}{+}{0}$ & $\scinote{1.3}{+}{1}$
 &  $\scinote{1.2}{+}{3}$ & $\scinote{5.7}{+}{0}$ & $\scinote{2.1}{+}{1}$
 & $\scinote{1.2}{+}{3}$ & $\scinote{5.0}{+}{0}$ & $\scinote{1.5}{+}{1}$
\\
 &$128^3$ &  $\scinote{1.1}{+}{5}$ & $\scinote{7.2}{+}{2}$ & $\scinote{5.8}{+}{2}$
 & $\scinote{2.2}{+}{5}$ &  $\scinote{1.2}{+}{2}$ & $\scinote{1.8}{+}{2}$
 &  $\scinote{1.5}{+}{4}$ & $\scinote{5.0}{+}{1}$ & $\scinote{2.2}{+}{2}$
 & $\scinote{1.5}{+}{4}$ & $\scinote{3.7}{+}{1}$ & $\scinote{1.6}{+}{2}$
\\
  &$256^3$ & --  & -- & --
 &  -- & --  & --
 &  $\scinote{2.0}{+}{5}$ & $\scinote{4.5}{+}{2}$ & $\scinote{2.2}{+}{3}$
 & $\scinote{2.0}{+}{5}$ & $\scinote{3.2}{+}{2}$ & $\scinote{1.6}{+}{3}$

\\
\bottomrule\\
\end{tabular}

\bigskip\bigskip

\caption{Accuracy results for Example 2\label{tab:3Dcubeapply}}
\begin{tabular}{@{}@{\extracolsep{-1pt}}cc|ccc|ccc|ccc|ccc@{}}\toprule
&\multicolumn{1}{c}{}&\multicolumn{3}{c}{RS}& \multicolumn{3}{c}{HIF}&\multicolumn{3}{c}{RS-S} & \multicolumn{3}{c}{RS-WS}\\ \cmidrule(lr){3-5} \cmidrule(lr){6-8} \cmidrule(lr){9-11} \cmidrule(lr){12-14}
$ \epsilon$ & $N$ & $e_a$ &$e_s$  &$n_i$& $e_a$ &$e_s$ & $n_i$& $e_a$ &$e_s$ &  $n_i$& $e_a$ &$e_s$& $n_i$ \\\midrule
   \multirow{ 4}{*}{$10^{-3}$}
 &$32^3$   & $\scinote{2.3}{-}{04}$ & $\scinote{2.3}{-}{02}$ & 6
 & $\scinote{3.1}{-}{04}$ & $\scinote{2.3}{-}{02}$ &  6
 & $\scinote{9.1}{-}{05}$ &  $\scinote{1.2}{-}{02}$ & 5
 & $\scinote{1.9}{-}{04}$ &$\scinote{8.8}{-}{02}$ & 7

\\
 &$64^3$   & $\scinote{3.3}{-}{04}$ & $\scinote{4.5}{-}{02}$ & 7
 & $\scinote{3.6}{-}{04}$ & $\scinote{5.4}{-}{02}$ &  6
 & $\scinote{1.1}{-}{04}$ &  $\scinote{2.5}{-}{02}$ & 6
 & $\scinote{2.3}{-}{04}$ &$\scinote{1.0}{-}{01}$ & 8

\\
 &$128^3$   & $\scinote{1.1}{-}{03}$ & $\scinote{1.2}{-}{01}$ & 8
 & $\scinote{1.2}{-}{03}$ & $\scinote{8.3}{-}{02}$ &  8
 & $\scinote{1.2}{-}{04}$ &  $\scinote{4.3}{-}{02}$ & 6
 & $\scinote{3.0}{-}{04}$ &$\scinote{9.7}{-}{02}$ & 9
 \\
 & $256^3$& -- & -- & --
 & $\scinote{3.2}{-}{03}$ & $\scinote{2.0}{-}{01}$ &  11
 & $\scinote{1.3}{-}{04}$ &  $\scinote{9.0}{-}{02}$ & 7
 & $\scinote{7.9}{-}{04}$ &$\scinote{3.4}{-}{01}$ & 11    \\\midrule
   \multirow{ 4}{*}{$10^{-6}$}&

 $32^3$   &$\scinote{1.8}{-}{07}$ & $\scinote{4.3}{-}{05}$ & 3
 & $\scinote{1.2}{-}{07}$ & $\scinote{2.8}{-}{05}$ &  2
 & $\scinote{2.0}{-}{08}$ &  $\scinote{1.8}{-}{05}$ & 2
 & $\scinote{2.2}{-}{07}$ &$\scinote{1.6}{-}{04}$ & 3

\\
 &$64^3$   & $\scinote{3.2}{-}{07}$ & $\scinote{7.1}{-}{05}$ & 3
 & $\scinote{2.4}{-}{07}$ & $\scinote{9.2}{-}{05}$ &  3
 & $\scinote{3.1}{-}{08}$ &  $\scinote{2.7}{-}{05}$ & 2
 & $\scinote{2.7}{-}{07}$ &$\scinote{1.9}{-}{04}$ & 3

\\
 &$128^3$   & $\scinote{6.2}{-}{07}$ & $\scinote{1.7}{-}{04}$ & 3
 & $\scinote{3.5}{-}{07}$ & $\scinote{1.7}{-}{04}$ &  3
 & $\scinote{4.0}{-}{08}$ &  $\scinote{4.0}{-}{05}$ & 3
 & $\scinote{4.7}{-}{07}$ &$\scinote{3.9}{-}{04}$ & 3
 \\
 & $256^3$& -- & -- & --
 & -- & -- &  --
 & $\scinote{4.5}{-}{08}$ &  $\scinote{8.2}{-}{05}$ & 3
 & $\scinote{8.0}{-}{07}$ &$\scinote{5.9}{-}{04}$ & 3    \\
\bottomrule\\
\end{tabular}
\end{sidewaystable}

As in the 2D case, \cref{tab:3Dcubeapply} shows that the approximate relative operator-norm error $e_a$ seems relatively well-controlled by $\epsilon$, though we observe small growth in $N$.  Similarly, the bound $e_s$ on the relative error for the inverse operator again loses a few digits compared to $e_a$ due to conditioning.  For $\epsilon=10^{-3}$ the number of CG iterations to convergence grows with $N$, albeit slowly.  For $\epsilon=10^{-6}$, however, $n_i$ remains stable.

\subsection{Example 3: unit sphere in 3D}
As a final example, we move to a more complicated 3D geometry.  Letting $\response{G(z) = \frac{1}{4\pi \|z\|}}$, we take $a(x)\equiv -1/2$, and $b(x)\equiv c(x) \equiv 1$ on the unit sphere $\Omega=S^2$ to obtain the second-kind boundary integral equation
\begin{align}\label{eq:ex3}
-\frac{1}{2}u(x) + \int_{S^2} \frac{\partial G}{\partial n_y}(x-y)u(y)\,dy = f(x), \quad x\in S^2,
\end{align}
where our kernel is the normal derivative of the Green's function for the Laplace equation in 3D.  \response{This corresponds to a double-layer potential solution representation for the interior Dirichlet Laplace problem on the unit sphere, that is, taking
\begin{align}\label{eq:getsol}
w(x) \equiv \int_{S^2} \frac{\partial G}{\partial n_y}(x-y)u(y)\,dy
\end{align}
 we have $\Delta w(x) = 0$ for $x$ inside the unit ball and $w(x) = f(x)$ on $S^2$.}

While it possible to build a periodic quadtree on a 2D parameterization of $S^2$, we treat the discretization of the sphere as points in $\R^3$ and use an octree. We use a centroid collocation scheme to discretize \eqref{eq:ex3},
wherein we represent $S^2$ as a collection of flat triangles and treat
all near-field interactions using fourth-order tensor-product
Gauss-Legendre quadrature, where we define near-field interactions as
interactions between triangles separated by a distance less than the
average triangle diameter.  Note that this leads to \response{an unsymmetric
matrix} $\m{K}$.  We choose the occupancy parameter $n_\text{occ}=256$
and $n_p=512$ random proxy points as in Example 2.

Timing and memory results for $\epsilon=10^{-3}$ and $\epsilon=10^{-6}$ can be seen in \cref{fig:sphereresults} with corresponding data for RS-S and RS-WS in \cref{tab:3Dsphere}.  For $\epsilon=10^{-3}$, all quantities behave definitively linearly as expected, with RS-WS again offering a trade-off between runtime and memory usage.  For $\epsilon=10^{-6}$, however, we actually observe \emph{sublinear} scaling of $t_f$ with $N$, which clearly indicates non-asymptotic behavior.  Further, looking at $m_f$ we see that memory usage is not significantly lessened with RS-WS for this example due to the fact that near-field compression is only mildly effective at this precision level. %\vminden{Anything else to say here?}

In \cref{tab:3Dsphereapply} we provide $e_a$ and $e_s$ for this example as well as a new quantity, $e_p$.  The explanation of this is as follows.  First, we choose 16 random sources $\{y_j\}$ with $\|y_j\| = 2$ and construct the harmonic field
\begin{align*}
v(x) &\equiv \sum_j G(x - y_j)q_j,
\end{align*}
where each $q_j$ is a standard normal random variable.  Taking the boundary data $f(x)$ in \eqref{eq:ex3} to be $v(x)$ on $S^2$, \response{the analytic solution to the interior Laplace boundary value problem is exactly $v(x)$ from uniqueness.  Numerically, we may use \eqref{eq:getsol} to reconstruct $\hat v(x) \approx v(x)$.}  Taking 16 random targets $\{z_j\}$ with $\|z_j\|=1/2$, we compute the relative error $e_p$ between $\{v(z_j)\}$ and $\{\hat v(z_j)\}$.  These results show that both RS-S and RS-{}WS both solve the integral equation \eqref{eq:ex3} up to discretization error.  \response{We do not provide preconditioning results for this example, as the linear operator is well-conditioned even without preconditioning.}

\begin{figure}
\centering
\includegraphics[scale=0.325]{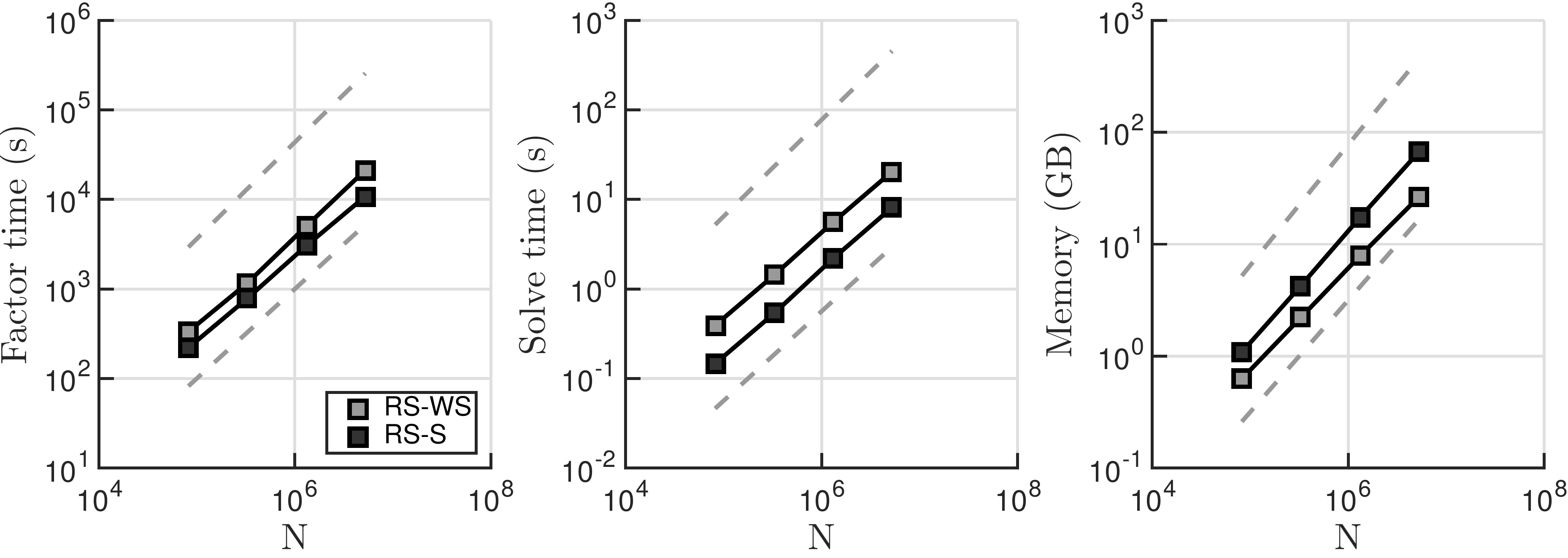}

\vspace{0.5cm}

\includegraphics[scale=0.325]{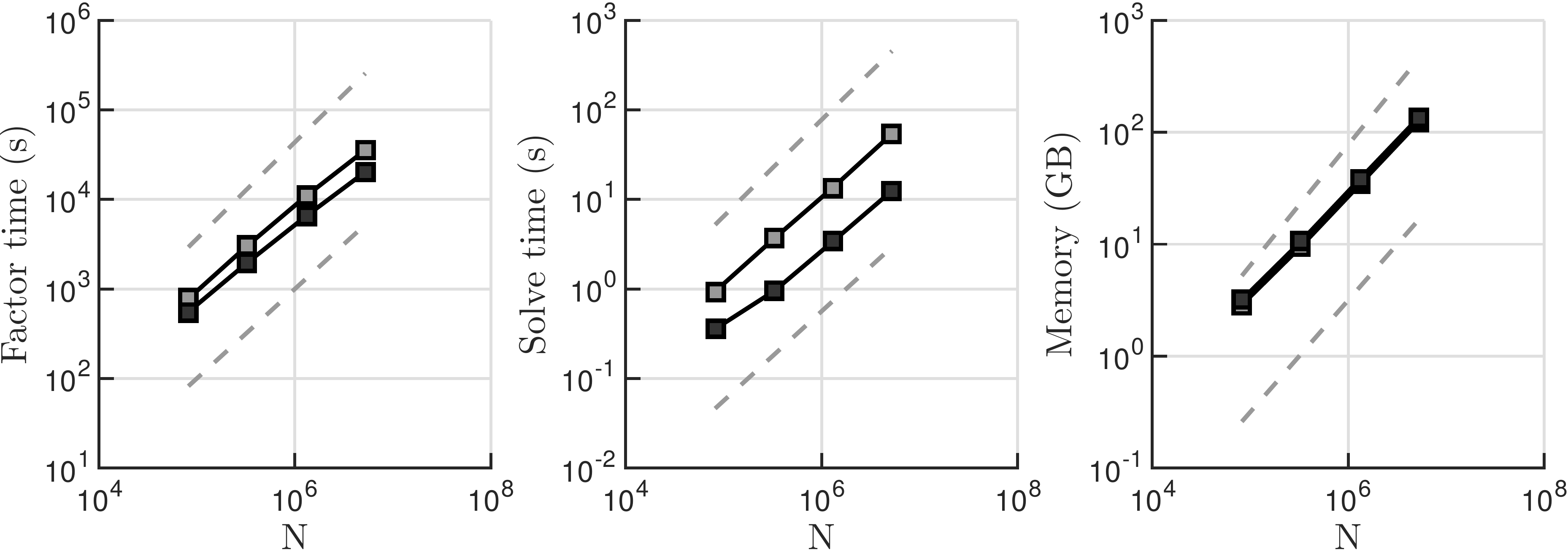}

\caption{Wall clock factor times $t_f$ and solve times $t_s$ and memory usage $m_f$ from Example 3 are shown for $\epsilon=10^{-3}$ (top row) and $\epsilon=10^{-6}$ (bottom row).  Each plot follows the top-left legend, with additional reference scaling curves $\O(N\log N)$ and $\O(N)$.  Corresponding data are given in \cref{tab:3Dcube}.  Note that in the last subplot the curves for RS-S and RS-WS lie nearly on top of each other.
\label{fig:sphereresults}}
\end{figure}

\begin{table}
\small
\ra{1.0}
\centering

\caption{Timing and memory results for Example 3 \label{tab:3Dsphere}}
\begin{tabular}{@{}@{\extracolsep{-1pt}}cc|ccc|ccc@{}}\toprule
&\multicolumn{1}{c}{}& \multicolumn{3}{c}{RS-S} & \multicolumn{3}{c}{RS-WS}\\ \cmidrule(lr){3-5} \cmidrule(lr){6-8}
$ \epsilon$ & $N$ &$t_f$ &$t_s$ &$m_f$ & $t_f$& $t_s$ & $m_f$ \\\midrule
   \multirow{ 4}{*}{$10^{-3}$}
 &$81920$   &$\scinote{2.2}{+}{2}$  & $\scinote{1.5}{-}{1}$ & $\scinote{1.1}{+}{0}$
 & $\scinote{3.3}{+}{2}$  & $\scinote{3.9}{-}{1}$ & $\scinote{6.3}{-}{1}$\\

 &$327680$  & $\scinote{7.8}{+}{2}$  & $\scinote{5.3}{-}{1}$ & $\scinote{4.2}{+}{0}$
 & $\scinote{1.2}{+}{3}$ & $\scinote{1.4}{+}{0}$ & $\scinote{2.2}{+}{0}$ \\

 &$1310720$  & $\scinote{3.0}{+}{3}$  & $\scinote{2.2}{+}{0}$ & $\scinote{1.7}{+}{1}$
 & $\scinote{5.0}{+}{3}$  & $\scinote{5.6}{+}{0}$ & $\scinote{7.8}{+}{0}$\\

  &$5242880$  & $\scinote{1.1}{+}{4}$  & $\scinote{8.1}{+}{0}$ & $\scinote{6.7}{+}{1}$
 & $\scinote{2.1}{+}{4}$  & $\scinote{2.0}{+}{1}$ & $\scinote{2.6}{+}{1}$
\\\midrule
   \multirow{ 4}{*}{$10^{-6}$}
 &$81920$  &  $\scinote{5.5}{+}{2}$  & $\scinote{3.6}{-}{1}$ & $\scinote{3.2}{+}{0}$
 & $\scinote{7.8}{+}{2}$  & $\scinote{9.4}{-}{1}$ & $\scinote{2.9}{+}{0}$
\\

 &$327680$  & $\scinote{2.0}{+}{3}$  & $\scinote{9.5}{-}{1}$ & $\scinote{1.1}{+}{1}$
 & $\scinote{3.0}{+}{3}$  & $\scinote{3.7}{+}{0}$ & $\scinote{9.5}{+}{0}$
\\
 &$1310720$ &  $\scinote{6.5}{+}{3}$& $\scinote{3.4}{+}{0}$  & $\scinote{3.8}{+}{1}$
 & $\scinote{1.1}{+}{4}$  & $\scinote{1.4}{+}{1}$ & $\scinote{3.4}{+}{1}$

\\
  &$5242880$ & $\scinote{2.0}{+}{4}$  & $\scinote{1.3}{+}{1}$ & $\scinote{1.4}{+}{2}$
 & $\scinote{3.6}{+}{4}$ & $\scinote{5.4}{+}{1}$ & $\scinote{1.2}{+}{2}$

\\
\bottomrule\\
\end{tabular}
\end{table}

\begin{table}
\small
\ra{1.0}
\centering

\caption{Accuracy results for Example 3\label{tab:3Dsphereapply}}
\begin{tabular}{@{}@{\extracolsep{-1pt}}cc|ccc|ccc@{}}\toprule
&\multicolumn{1}{c}{}&\multicolumn{3}{c}{RS-S} & \multicolumn{3}{c}{RS-WS}\\ \cmidrule(lr){3-5} \cmidrule(lr){6-8}
$ \epsilon$ & $N$ & $e_a$ &$e_s$ & $e_p$& $e_a$ &$e_s$  & $e_p$\\\midrule
   \multirow{ 4}{*}{$10^{-3}$}

 &$81920$   & $\scinote{2.2}{-}{04}$ & $\scinote{2.2}{-}{04}$ & $\scinote{4.0}{-}{04}$
 & $\scinote{6.2}{-}{04}$ & $\scinote{6.8}{-}{04}$ &  $\scinote{4.8}{-}{04}$

\\
 &$327680$   & $\scinote{4.3}{-}{04}$ & $\scinote{4.3}{-}{04}$ & $\scinote{2.2}{-}{04}$
 & $\scinote{9.7}{-}{04}$ & $\scinote{9.8}{-}{04}$ &  $\scinote{2.6}{-}{04}$
 \\
 & $1310720$& $\scinote{7.5}{-}{04}$ & $\scinote{7.5}{-}{04}$ & $\scinote{1.6}{-}{04}$
 & $\scinote{1.4}{-}{03}$ & $\scinote{1.4}{-}{03}$ &  $\scinote{2.3}{-}{04}$   \\

  &$5242880$   & $\scinote{1.1}{-}{03}$ & $\scinote{1.1}{-}{03}$ & $\scinote{1.6}{-}{04}$
 & $\scinote{1.9}{-}{03}$ & $\scinote{1.9}{-}{03}$ &  $\scinote{2.3}{-}{04}$
\\\midrule
    \multirow{ 4}{*}{$10^{-6}$}

 &$81920$   & $\scinote{6.9}{-}{07}$ & $\scinote{6.9}{-}{07}$ & $\scinote{3.8}{-}{04}$
 & $\scinote{1.0}{-}{06}$ & $\scinote{1.0}{-}{06}$ &  $\scinote{3.7}{-}{04}$

\\
 &$327680$   & $\scinote{1.3}{-}{06}$ & $\scinote{1.3}{-}{06}$ & $\scinote{1.8}{-}{04}$
 & $\scinote{1.9}{-}{06}$ & $\scinote{1.9}{-}{06}$ &  $\scinote{1.8}{-}{04}$
 \\
 & $1310720$& $\scinote{1.7}{-}{06}$ & $\scinote{1.7}{-}{06}$ & $\scinote{9.0}{-}{05}$
 & $\scinote{2.8}{-}{06}$ & $\scinote{2.8}{-}{06}$ &  $\scinote{8.9}{-}{05}$  \\

  &$5242880$   & $\scinote{2.2}{-}{06}$ & $\scinote{2.2}{-}{06}$ & $\scinote{4.4}{-}{05}$
 & $\scinote{3.9}{-}{06}$ & $\scinote{3.9}{-}{06}$ &  $\scinote{4.4}{-}{05}$\\
\bottomrule\\
\end{tabular}
\end{table}

\section{Conclusions}
\label{sec:conclusions}
By modifying the recursive skeletonization process of Martinsson \& Rokhlin \cite{martinsson-rokhlin} to operate directly on strongly-admissible structure (\ie, perform only far-field compression), we obtain a new factorization, RS-S, that is useful for solving integral equations in $\R^2$ and $\R^3$ both as a medium-accuracy direct solver and as an excellent preconditioner for iterative methods.  \response{As a high-accuracy direct solver, the performance of RS-S in 3D is less practical due to memory requirements, though 2D performance remains competitive.}
We further offer a modification of our approach, RS-WS, which gives a trade-off between runtime complexity and storage complexity through additional levels of compression.

We apply both factorizations to a number of examples to evaluate performance according to a number of metrics.  While we focus in this paper on solving integral equations, the linear algebraic machinery developed can be applied much more broadly for general structured matrices (\eg, kernelized covariance matrices \cite{mindengp}).

Compared to other skeletonization-based methods for obtaining linear or nearly-linear complexity factorizations such as HIF \cite{hifie} or the method of Corona et al. \cite{corona2013} (for 2D problems), our approach is competitive, but more importantly is simpler to implement.  In particular, these previous methods based on near-field compression have obtained better runtime complexity at the cost of increased algorithmic complexity by introducing additional geometric information beyond the tree decomposition of space.  By working directly with strong admissibility, this becomes unnecessary.  \response{For 3D problems, we also observe better runtime performance than HIF, see \cref{fig:3dresults}.}

In contrast to the IFMM of Coulier et al. \cite{ifmm2} and Ambikasaran \& Darve \cite{ifmm1}, which provides a fast method for solving integral equations based on exploiting strong admissibility through a telescoping additive decomposition, our method takes the form of a multiplicative factorization.  This gives greater flexibility in that we may compute approximate \response{generalized} square-roots or log-determinants --- essentially, we get the benefit of having a true (albeit approximate) triangular factorization.  Further, building on the skeletonization framework allows \response{accelerated compression throughout our algorithm} due to the use of a proxy surface as described in \cref{sec:proxy}.

The most pressing direction for future research is understanding the rank behavior of far-field blocks that have been subject to Schur complement updates through the skeletonization process as discussed in \cref{sec:leaf}.  The efficiency of our factorizations hinges on these blocks remaining low-rank as the algorithm progresses, which seems to be more-or-less true in our numerical experiments.  Another direction of future research is understanding the additional approximation error introduced through the use of a proxy surface in \cref{sec:proxy}.  In particular, while it follows from the discussion of \cref{sec:proxy} that an exact ID using proxy points leads to an exact compression of the far-field interactions, it is not immediately evident how tight a bound on the relative error in an approximate ID might be attainable when the IDs are no longer exact and when the discrete approximation to Green's identity is used.

Finally, due to the simple tree structure, RS-S and RS-WS are both easily parallelizable.  For example, on a regular 2D grid we may use a four-coloring of the boxes on each level as in \cref{fig:parallel}.  In this case, we may perform strong skeletonization with respect to the DOF sets of each brown leaf box in parallel, then similarly for each blue leaf box and so on.  This ordering should also, in principle, allow for adaptation of fast factorization-updating algorithms such as we \response{described} in earlier work \cite{mindenupdate}.

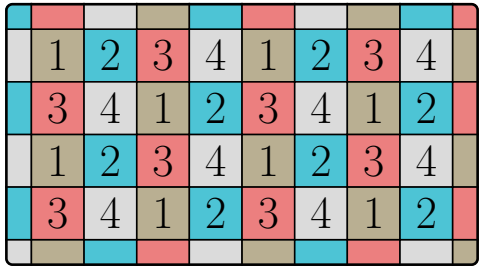
\begin{figure}
\centering
% Proxy surface
\scalebox{0.7}{
  \begin{tikzpicture}
    \clip[rounded corners](-0.5,-0.5) rectangle (8.5,4.5);
 %  \foreach \x in {-1,...,8}
 %      \foreach \y in {-1,...,4}
 % \filldraw[fill=box,draw=black,thick,fill opacity=0.7] (\x,\y) rectangle (\x+1, \y+1);

\foreach \y in {-2}{
   \filldraw[fill=lightgray,draw=black,thick,fill opacity=0.7] (-1,\y+1) rectangle (0, \y+2);

{}
   \filldraw[fill=box,draw=black,thick,fill opacity=0.54] (0,\y+1) rectangle (1, \y+2);

\filldraw[fill=near,draw=black,thick,fill opacity=0.7] (1,\y+1) rectangle (2, \y+2);

  \filldraw[fill=far,draw=black,thick,fill opacity=0.7] (2,\y+1) rectangle (3, \y+2);

    \filldraw[fill=lightgray,draw=black,thick,fill opacity=0.7] (3,\y+1) rectangle (4, \y+2);

      \filldraw[fill=box,draw=black,thick,fill opacity=0.54] (4,\y+1) rectangle (5, \y+2);

\filldraw[fill=near,draw=black,thick,fill opacity=0.7] (5,\y+1) rectangle (6, \y+2);

      \filldraw[fill=far,draw=black,thick,fill opacity=0.7] (6,\y+1) rectangle (7, \y+2);

      \filldraw[fill=lightgray,draw=black,thick,fill opacity=0.7] (7,\y+1) rectangle (8, \y+2);

     \filldraw[fill=box,draw=black,thick,fill opacity=0.54] (8,\y+1) rectangle (9, \y+2);
  }

 \foreach \y in {0,2}{
   \filldraw[fill=lightgray,draw=black,thick,fill opacity=0.7] (-1,\y+1) rectangle (0, \y+2);

   \filldraw[fill=box,draw=black,thick,fill opacity=0.54] (0,\y+1) rectangle (1, \y+2);
  \node at (0+0.5,\y+1.5) {\Huge$1$};

\filldraw[fill=near,draw=black,thick,fill opacity=0.7] (1,\y+1) rectangle (2, \y+2);
  \node at (1+0.5,\y+1.5) {\Huge$2$};

  \filldraw[fill=far,draw=black,thick,fill opacity=0.7] (2,\y+1) rectangle (3, \y+2);
  \node at (2+0.5,\y+1.5) {\Huge$3$};

    \filldraw[fill=lightgray,draw=black,thick,fill opacity=0.7] (3,\y+1) rectangle (4, \y+2);
  \node at (3+0.5,\y+1.5) {\Huge$4$};

      \filldraw[fill=box,draw=black,thick,fill opacity=0.54] (4,\y+1) rectangle (5, \y+2);
  \node at (4+0.5,\y+1.5) {\Huge$1$};

\filldraw[fill=near,draw=black,thick,fill opacity=0.7] (5,\y+1) rectangle (6, \y+2);
  \node at (5+0.5,\y+1.5) {\Huge$2$};

      \filldraw[fill=far,draw=black,thick,fill opacity=0.7] (6,\y+1) rectangle (7, \y+2);
  \node at (6+0.5,\y+1.5) {\Huge$3$};
      \filldraw[fill=lightgray,draw=black,thick,fill opacity=0.7] (7,\y+1) rectangle (8, \y+2);
  \node at (7+0.5,\y+1.5) {\Huge$4$};

     \filldraw[fill=box,draw=black,thick,fill opacity=0.54] (8,\y+1) rectangle (9, \y+2);
  }

   \foreach \y in {-1,1}{
      \filldraw[fill=near,draw=black,thick,fill opacity=0.7] (-1,\y+1) rectangle (0, \y+2);
     \filldraw[fill=far,draw=black,thick,fill opacity=0.7] (0,\y+1) rectangle (1, \y+2);
  \node at (0+0.5,\y+1.5) {\Huge$3$};
      \filldraw[fill=lightgray,draw=black,thick,fill opacity=0.7] (1,\y+1) rectangle (2, \y+2);
  \node at (1+0.5,\y+1.5) {\Huge$4$};

      \filldraw[fill=box,draw=black,thick,fill opacity=0.54] (2,\y+1) rectangle (3, \y+2);
  \node at (2+0.5,\y+1.5) {\Huge$1$};

\filldraw[fill=near,draw=black,thick,fill opacity=0.7] (3,\y+1) rectangle (4, \y+2);
  \node at (3+0.5,\y+1.5) {\Huge$2$};

    \filldraw[fill=far,draw=black,thick,fill opacity=0.7] (4,\y+1) rectangle (5, \y+2);
  \node at (4+0.5,\y+1.5) {\Huge$3$};
      \filldraw[fill=lightgray,draw=black,thick,fill opacity=0.7] (5,\y+1) rectangle (6, \y+2);
  \node at (5+0.5,\y+1.5) {\Huge$4$};

  \filldraw[fill=box,draw=black,thick,fill opacity=0.54] (6,\y+1) rectangle (7, \y+2);
  \node at (6+0.5,\y+1.5) {\Huge$1$};

  \filldraw[fill=near,draw=black,thick,fill opacity=0.7] (7,\y+1) rectangle (8, \y+2);
  \node at (7+0.5,\y+1.5) {\Huge$2$};

     \filldraw[fill=far,draw=black,thick,fill opacity=0.7] (8,\y+1) rectangle (9, \y+2);
  }

  \foreach \y in {3}{
      \filldraw[fill=near,draw=black,thick,fill opacity=0.7] (-1,\y+1) rectangle (0, \y+2);
     \filldraw[fill=far,draw=black,thick,fill opacity=0.7] (0,\y+1) rectangle (1, \y+2);

      \filldraw[fill=lightgray,draw=black,thick,fill opacity=0.7] (1,\y+1) rectangle (2, \y+2);

      \filldraw[fill=box,draw=black,thick,fill opacity=0.54] (2,\y+1) rectangle (3, \y+2);

\filldraw[fill=near,draw=black,thick,fill opacity=0.7] (3,\y+1) rectangle (4, \y+2);

    \filldraw[fill=far,draw=black,thick,fill opacity=0.7] (4,\y+1) rectangle (5, \y+2);

      \filldraw[fill=lightgray,draw=black,thick,fill opacity=0.7] (5,\y+1) rectangle (6, \y+2);

  \filldraw[fill=box,draw=black,thick,fill opacity=0.54] (6,\y+1) rectangle (7, \y+2);

  \filldraw[fill=near,draw=black,thick,fill opacity=0.7] (7,\y+1) rectangle (8, \y+2);

     \filldraw[fill=far,draw=black,thick,fill opacity=0.7] (8,\y+1) rectangle (9, \y+2);
  }
% \foreach \y in {0,...,3}
%   \node[right] at (8,\y+0.5) {\huge$\hdots$};
%   \foreach\x in {0,...,7}
%   \node[below] at (\x+0.5,0) {\Huge$\vdots$};
        % \ifthenelse{\y<-1 \OR \y >1 \OR \x<-1 \OR \x>1}{\filldraw[fill=far,draw=black,thick] (\x,\y) rectangle (\x+1, \y+1)}{\filldraw[fill=near,draw=black,thick] (\x,\y) rectangle (\x+1, \y+1)};
  %\filldraw[fill=box,draw=black,thick,fill opacity=0.7] (0,0) rectangle (1,1);
  \draw[use as bounding box,draw=black,line width=1mm, rounded corners](-0.5,-0.5) rectangle (8.5,4.5);
  \end{tikzpicture}
  }

  \caption{\label{fig:parallel} To parallelize RS and RS-WS efficiently, a four-coloring of the domain (shown) can be used such that no box is colored the same as its neighbors.  All boxes of a given color may be skeletonized independently.  The 3D case is similar, albeit requiring more colors.}
\end{figure}

% \vminden{what about conditioning of the transfer operator?}

% \nocite{*}
\section*{Acknowledgments}
\response{The authors thank the referees for valuable feedback, particularly on strengthening the statement of \cref{cor:scaling}.  V.M. would also like to thank A. Benson, R. Estrin, Y. Li, B. Nelson, N. Skochdopole, and X. Suo for useful comments on early drafts of this manuscript, as well as Stanford University and the Stanford Research Computing Center for providing computational resources and support that have contributed to these research results.}
\bibliographystyle{siamplain}
\bibliography{rss}
\label{LastPage}
\end{document}